\documentclass[12pt]{amsart}
\usepackage[margin=2.5cm]{geometry}
\usepackage{amsmath,amssymb,amsfonts,amsthm,mathrsfs}
\usepackage[nobysame]{amsrefs} 
\usepackage{epic}
\usepackage{amsopn,amscd,graphicx}
\usepackage{color,transparent} 
\usepackage[usenames, dvipsnames]{xcolor}
\usepackage
[colorlinks, breaklinks,
bookmarks = false,
linkcolor = NavyBlue,
urlcolor = ForestGreen,
citecolor = ForestGreen,
hyperfootnotes = false
]
{hyperref}
\usepackage{scalerel}
\usepackage{stackengine,wasysym}
\usepackage{palatino, mathpazo}
\usepackage{multirow}
\usepackage[all,cmtip]{xy}
\usepackage{stmaryrd}
\usepackage{tikz-cd}
\usepackage{yhmath}
\usepackage{enumerate}
\usepackage{comment}
\usepackage[dvipsnames]{xcolor}
 1
 1
 1

\newtheorem{theorem}{Theorem}[section]
\newtheorem{corollary}[theorem]{Corollary}
\newtheorem{proposition}[theorem]{Proposition}
\newtheorem{problem}[theorem]{Problem}
\newtheorem{lemma}[theorem]{Lemma}
\newtheorem{remark}[theorem]{Remark}

\newtheorem{example}[theorem]{Example}
\theoremstyle{definition}
\newtheorem{definition}[theorem]{Definition}

\newtheoremstyle{named}{}{}{\itshape}{}{\bfseries}{.}{.5em}{\thmnote{#3}#1}
\theoremstyle{named}

\numberwithin{equation}{subsection}

\newcommand{\ol}{\overline}
\newcommand{\R}{\mathbb{R}}
\newcommand{\C}{\mathbb{C}}
\newcommand{\N}{\mathbb{N}}
\newcommand{\Z}{\mathbb{Z}}

\newcommand{\abs}[1]{\left\vert#1\right\vert}
\newcommand{\set}[1]{\left\{#1\right\}}
\newcommand{\To}{\rightarrow}
\def\cC{\mathscr{C}}

\newcommand{\vbunsec}[1]{\mathscr{C}^\infty\left(#1\right)}

\DeclareMathOperator{\supp}{supp}
\DeclareMathOperator{\Dom}{Dom}

\title[Approximation of Pseudohermitian Structures via Embeddings into Spheres]{Approximation of Pseudohermitian Structures\\ via Embeddings into Spheres}

\begin{document}

\author[Hendrik Herrmann]{Hendrik Herrmann} \address{University of
  Vienna, Faculty of Mathematics, Oskar-Morgenstern-Platz 1, 1090
  Vienna, Austria} \thanks{Hendrik Herrmann was partially supported
  the by the Austrian Science Fund (FWF) grant: The dbar Neumann
  operator and related topics (DOI 10.55776/P36884)}
\email{hendrik.herrmann@univie.ac.at or post@hendrik-herrmann.de}

\author[Chin-Yu Hsiao]{Chin-Yu Hsiao} \address{Department of
  Mathematics, National Taiwan University, Astronomy-Mathematics
  Building, No. 1, Sec. 4, Roosevelt Road, Taipei 10617, Taiwan}
\thanks{Chin-Yu Hsiao was partially supported by the NSTC Project
  113-2115-M-002 -011 -MY3} \email{chinyuhsiao@ntu.edu.tw or
  chinyu.hsiao@gmail.com}

\author[Bernhard Lamel]{Bernhard Lamel} \address{University of Vienna,
  Faculty of Mathematics, Oskar-Morgenstern-Platz 1, 1090 Vienna,
  Austria} \thanks{Bernhard Lamel was supported by the Austrian
  Science Fund (FWF) grant DOI 10.55776/I4557}
\email{bernhard.lamel@univie.ac.at}

\date{\today}

\begin{abstract}  
  Let \((X,T^{1,0}X)\) be a compact strictly pseudoconvex CR manifold
  which is CR embeddable into the complex Euclidean space. We show
  that \(T^{1,0}X\) can be approximated in
  \(\mathscr{C}^\infty\)-topology by a sequence of strictly
  pseudoconvex CR structures \(\{\mathcal{V}^k\}_{k\in \N}\) such that
  for each \(k\in \N\) we have that \((X,\mathcal{V}^k)\) is CR
  embeddable into the unit sphere of a complex Euclidean
  space. Furthermore, as a refinement of this statement, we show that
  given a one form \(\alpha\) on \(X\) such that
  \((X,T^{1,0}X,\alpha)\) is a pseudohermitian manifold (see Definition~\ref{def:PseudohermitianStructure}) we can
  approximate \((T^{1,0}X,\alpha)\) in \(\mathscr{C}^\infty\)-topology
  by a sequence of pseudohermitian structures
  \(\{(\mathcal{V}^k,\alpha^k)\}_{k\in \N}\) on \(X\) such that for
  each \(k\in \N\) we have that \((X,\mathcal{V}^k,\alpha^k)\) is
  isomorphic to a real analytic pseudohermitian submanifold of a
  sphere. A similar result for the Sasakian case was obtained earlier
  by Loi--Placini. Let \((X,T^{1,0}X,\mathcal{T})\) be a compact
  Sasakian manifold, i.e. \(\mathcal{T}\) is a real vector field transversal to \(\operatorname{Re}T^{1,0}X\), the flow of \(\mathcal{T}\) acts by CR
  automorphisms and the one form \(\alpha\) defined by
  \(\alpha(\mathcal{T})=1\) and \(\alpha(\operatorname{Re}T^{1,0}X)=0\)
  defines a pseudohermitian structure on
  \((X,T^{1,0}X)\). Loi--Placini showed that \((T^{1,0}X,\mathcal{T})\)
  can be smoothly approximated by a sequence of
  quasi-regular Sasakian structures
  \(\{(\mathcal{V}^k,\mathcal{T}^k)\}_{k\in \N}\) on \(X\) such that
  each \((X,\mathcal{V}^k,\mathcal{T}^k)\) admits a smooth equivariant
  CR embedding into a Sasakian sphere.  Applying our methods to the
  Sasakian case we show that it is possible to approximate with a
  sequence of Sasakian structures having the form
  \(\{(\mathcal{V}^k,\mathcal{T})\}_{k\in \N}\), i.e.~we can keep the
  vector field \(\mathcal{T}\). As a consequence of this refinement we
  are able to make statements on the structure of certain Sasakian
  deformations with respect to their embeddability into spheres. Further applications concerning the embedding of domains into balls and local approximation results are provided.  Our
  methods heavily depend on the semi-classical functional calculus for
  Toeplitz operators on CR manifolds and strictly pseudoconvex domains.
\end{abstract}
\maketitle
\tableofcontents


\tableofcontents
\section{Introduction and Statement of the Results}
\subsection{Introduction}
Let \((X,T^{1,0}X)\) be a compact strictly pseudoconvex CR manifold of
real dimension \(\dim_\R X=2n+1\) where \(T^{1,0}X\) denotes the CR
structure;  we
always assume that it is of hypersurface type in this paper, i.e.
$\dim_\C T^{1,0} X = n$. Answering the question whether \(X\) can be CR embedded
into the complex Euclidean space is a fundamental problem in CR
geometry. Boutet de Monvel~\cite{Bou75} showed that \(X\) can be CR
embedded into the complex Euclidean space provided that \(n\geq
2\). Rossi~\cite{R65} and Burns~\cite{Bu:77} (see also Grauert~\cite{Gr94}, Andreotti--Siu~\cite{AS70}) proved the existence of
non-embeddable \(X\) for the case \(n=1\). However,
Lempert~\cite{Lem92} (see also Marinescu--Yeganefar~\cite{MY07})
showed that for \(n=1\) we can guarantee CR embeddability of \(X\)
under the additional assumption that \(X\) carries a Sasakian
structure (see Definition~\ref{def:SasakianManifold}). A very useful
characterization for the embeddability of \(X\) follows from the works
of Boutet de Monvel \cite{Bou75}, Boutet de
Monvel--Sj\"ostrand~\cite{BS75}, Harvey--Lawson~\cite{HL75},
Burns~\cite{Bu:77}, Kohn~\cites{Koh85,Koh86}, Heunemann~\cite{Heu86},
Ohsawa~\cite{O84}, see also Catlin \cite{MR1128581}, and
Hill-Nacinovich~\cite{MR1289628} (cf.~\cites{MM07,HM17}): Given a
(connected) compact  strictly pseudoconvex CR manifold \((X,T^{1,0}X)\) the
following are equivalent.
\begin{itemize}
\item \(X\) is CR embeddable into the complex Euclidean space.
\item The Kohn Laplacian on \(X\) has closed range.
\item There is a complex manifold \(Y\) such that \(X\) is CR
  isomorphic to the boundary of a strictly pseudoconvex domain
  \(M\subset \subset Y\).
\end{itemize}
We refer the reader to Section~\ref{sec:Preliminaries} for a more detailed
explanation of this equivalence.  Assuming that the compact strictly
pseudoconvex CR manifold \((X,T^{1,0}X)\) is (CR) embeddable into the
complex Euclidean space one may ask for refined embeddings satisfying
additional properties. For example, Forn{\ae}ss~\cite{For76} showed
that \(X\) can be always CR embedded into the boundary of a strictly
convex domain in the complex Euclidean space. An important example for
a strictly convex relatively compact domain in \(\C^N\), \(N\geq 2\),
with real analytic boundary is the unit ball
\[\mathbb{B}^N=\{z\in \C^N\colon |z|^{2}<1\}.\]
Its boundary is the unit sphere
\[\mathbb{S}^{2N-1}=b\mathbb{B}^N=\{z\in \C^N\colon |z|^{2}=1\}.\]
We have that \((\mathbb{S}^{2N-1},T^{1,0}\mathbb{S}^{2N-1})\), where
\(T^{1,0}\mathbb{S}^{2N-1}\) is the CR structure induced by \(\C^N\)
(see Example~\ref{ex:PseudoHermSpheres}), is a compact real analytic
strictly pseudoconvex CR manifold which is embeddable. Spheres
are especially important from the point of view of
CR geometry, where they serve as locally flat models for
strictly pseudoconvex smooth CR structures. This is due to
work of Cartan~\cite{Cartan:1932,Cartan:1933}
for three-dimensional $X$ and work of Tanaka~\cite{Tanaka:1962} and Chern
and Moser~\cite{ChernMoser} for higher dimensions, which provided a
canonical Cartan connection and associated curvatures for strictly
pseudoconvex CR manifolds. From this point of view, spheres are
also maximally symmetric, in particular they carry plenty of different Sasakian
structures (see Remark~\ref{rmk:SasakianStructuresOnSpheres}). 
It is therefore, also in analogy to the Nash embedding theorem for
Riemannian manifolds, natural to ask whether every strictly
pseudoconvex manifold can actually be embedded into a sphere. Another analogy can
be drawn between embedding
compact embeddable strictly pseudoconvex CR manifolds into a sphere and
embedding compact complex manifolds with positive
holomorphic line bundles into complex projective space
(see Kodaira~\cite{Ko54}).

 However, it was shown by Faran~\cite{Far88} and 
Forstneri\v{c}~\cite{Fors86} that there exist real analytic strictly
pseudoconvex hypersurfaces in the complex Euclidean space which are
not CR embeddable into any finite dimensional sphere; some explicit
obstructions and examples were given by Zaitsev~\cite{zaitsev:2008}, Huang--Zaitsev~\cite{HuangZaitsev13}, Huang--Li--Xiao~\cite{HuLiXi15} and  Huang--Xiao~\cite{HuangXiao17}. Actually, it
follows from the results of Forstneri\v{c}~\cite{Fors86} that most of
the compact embeddable strictly pseudoconvex CR manifolds cannot be
embedded into finite dimensional spheres. However,
Lempert~\cite{Lem82,Lem91} proved that every compact, real analytic,
strictly pseudoconvex hypersurface in a complex Euclidean space can be
CR embedded into an infinite dimensional sphere. Furthermore, it
follows from the results of Forstneri\v{c}~\cite{Fors86} (see also
L{\o}w~\cite{Low85}) that any bounded strictly pseudoconvex domain in
a Stein manifold can be properly and holomorphically embedded into
some finite dimensional unit ball such that the embedding is
continuous up to the boundary. Our first main result is
  that embeddable CR structures which can be embedded into spheres are actually
  $\mathscr{C}^\infty$-dense in the space of embeddable CR structures.

\begin{theorem}\label{thm:MainThmIntroGeneralTeaser}
  Let \((X,T^{1,0}X)\) be a compact strictly pseudoconvex CR manifold
  which is CR embeddable into the complex Euclidean space. Then there
  exists a sequence of strictly pseudoconvex CR structures
  \(\{\mathcal{V}_k\}_{k\in \N}\) on \(X\) which approximate
  \(T^{1,0}X\) in \(\mathscr{C}^\infty\)-topology such that for each
  \(k\in \N\) there is \(N_k\in \N\) and a smooth embedding
  \(F_k\colon X\to \C^{N_k}\) which is CR with respect to
  \(\mathcal{V}_k\) and satisfies
  \(F_k(X)\subset \mathbb{S}^{2N_k-1}\). 
\end{theorem}

If \((X,T^{1,0}X)\) carries a Sasakian
structure it follows from the results of Ornea--Verbitsky~\cite{OV07}
that \(X\) can be CR embedded into a strictly pseudoconvex CR manifold
which is diffeomorphic to a sphere. More precisely, they showed that
any compact Sasakian manifold admits an equivariant CR embedding into
a Sasakian manifold which is diffeomorphic to a sphere. Under the
additional assumption that the first de Rham cohomology group of \(X\)
vanishes it was proved by van Coevering~\cite{vCoe11} that \(X\) can
be
embedded into a Sasakian graph over a sphere. Recently,  Herrmann--Hsiao--Marinescu--Shen~\cite{HHMS23} showed that in general 
any compact embeddable strictly pseudoconvex CR manifold \((X,T^{1,0}X)\)
can be embedded into arbitrary small perturbations of spheres asymptotically
preserving pseudohermitian structures on \((X,T^{1,0}X)\) with respect to  Sasakian structures on spheres.

One of the results in this paper is that for the 'most'  compact
Sasakian manifolds the CR structure cannot be obtained from Sasakian
embeddings into spheres (see
Theorem~\ref{thm:ForstnericForSasakian}). However, it was shown
earlier by Loi--Placini~\cite{LP22} that given a compact Sasakian
manifold \((X,T^{1,0}X,\mathcal{T})\) there exists a sequence of
quasi-regular Sasakian structures
 \(\{(\mathcal{V}^k,\mathcal{T}^k)\}_{k\in \N}\) on \(X\) smoothly approximating 
\((T^{1,0}X,\mathcal{T})\) 
such that each \((\mathcal{V}^k,\mathcal{T}^k)\) can be obtained from a
Sasakian embedding into a sphere, that is,
\((\mathcal{V}^k,\mathcal{T}^k)\) is induced by a weighted sphere (see
Definition~\ref{def:SasakianStructureSimpleDeformation}). As a
consequence of their result it follows that given a compact strictly
pseudoconvex CR manifold \((X,T^{1,0}X)\) carrying a Sasakian structure
there exists a sequence of strictly pseudoconvex CR structures
\(\{\mathcal{V}^k\}_{k\in \N}\) on \(X\) converging to \(T^{1,0}X\) in
\(\mathscr{C}^\infty\)-topology such that each \((X,\mathcal{V}^k)\)
can be CR embedded into a finite dimensional sphere.

Therefore,  if \((X,T^{1,0}X)\) carries a Sasakian structure
Theorem~\ref{thm:MainThmIntroGeneralTeaser} can be obtained from the
results of Loi--Placini~\cite{LP22}.
Theorem~\ref{thm:MainThmIntroGeneralTeaser} is a consequence of a more
general result on the approximation of pseudohermitian structures by
those coming from real analytic CR submanifolds of spheres (see
Theorem~\ref{thm:MainThmIntroGeneral}). This is the main result of
this work and is stated in the next section.
\subsection{Approximation of Pseudohermitian Structures}
In order to state the main result
Theorem~\ref{thm:MainThmIntroGeneral} we need to introduce some
notation.
\begin{definition}\label{def:PseudohermitianStructure}
  Let \(X\) be an orientable connected smooth manifold. A
  pseudohermitian structure on \(X\) is a pair \((T^{1,0}X,\alpha)\)
  where \(T^{1,0}X\) is a strictly pseudoconvex CR structure of
  codimension one and \(\alpha\in\Omega^1(X)\) is a non-vanishing real
  one form with \(\ker \alpha=\operatorname{Re}T^{1,0}X\) such that
  the respective Levi form \(\mathcal{L}^\alpha\) (see
  Definition~\ref{Def:LeviFormCRGeneral}) is positive definite.  The
  triple \((X,T^{1,0}X,\alpha)\) is called a pseudohermitian manifold.
\end{definition}
One can show that for any strictly pseudoconvex CR manifold
\((X,T^{1,0}X)\) one has that \(\operatorname{Re}T^{1,0}X\) defines a
contact structure on \(X\). Hence, given a pseudohermitian manifold
\((X,T^{1,0}X,\alpha)\) it follows from the definition that \(\alpha\)
is a contact form for the contact structure
\(\operatorname{Re}T^{1,0}X\). Furthermore, there exists a real vector
field \(\mathcal{T}\in \vbunsec{X,TX}\) uniquely defined by the
properties
\[\mathcal{T}\lrcorner\alpha\equiv1\,\,\,\mathcal{T}\lrcorner
  d\alpha\equiv 0\] which is called the Reeb vector field associated
to \(\alpha\).
\begin{example}\label{ex:PseudoHermSpheres}
  Let \(\mathbb{S}^{2N-1}=\{z\in \C^N\mid |z|=1\}\) be the unit sphere
  in \(\C^N\), \(N\geq 2\), and denote by
  \(\iota\colon \mathbb{S}^{2N-1}\to \C^N\) the inclusion. Given a
  vector \(\beta=(\beta_1,\ldots,\beta_N)\in (\R_+)^N\) of positive
  real numbers (called weight vector) consider the smooth real one
  form
  \[\alpha_\beta=\frac{1}{2i}\frac{\sum_{j=1}^N
      \overline{z}_jdz_j-z_jd\overline{z}_j}{\sum_{j=1}^N\beta_j|z_j|^2}\]
  on \(\C^{N}\setminus\{0\}\).  Define
  \(T^{1,0}\mathbb{S}^{2N=1}=\C T\mathbb{S}^{2N-1}\cap T^{1,0}\C^N\). We have that
  \((\mathbb{S}^{2N-1},T^{1,0}\mathbb{S}^{2N-1},\iota^*\alpha_\beta)\)
  is a pseudohermitian manifold.  Furthermore, consider the vector
  field
  \[\mathcal{T}_\beta= i\sum_{j=1}^N \beta_j\left(
      z_j\frac{\partial}{\partial
        z_j}-\overline{z}_j\frac{\partial}{\partial
        \overline{z}_j}\right)\] on \(\C^N\). We have that 
   \(\mathcal{T}_\beta\big|_{\mathbb{S}^{2N-1}}\) 
    is the Reeb vector field on
  \(\mathbb{S}^{2N-1}\) associated to \(\iota^*\alpha_\beta\).
\end{example}
\begin{remark}\label{rmk:SasakianStructuresOnSpheres}
  We have that for any \(\beta=(\beta_1,\ldots,\beta_N)\in (\R_+)^N\)
  the tuple
  \((T^{1,0}\mathbb{S}^{2N-1},\mathcal{T}_\beta\big|_{\mathbb{S}^{2N-1}})\) defines
  a Sasakian structure on \(\mathbb{S}^{2N-1}\) (see Definition~\ref{def:SasakianManifold}). This fact will be
  discussed later in more detail.
\end{remark}

\begin{definition}
  Let \((X_j,T^{1,0}X_j,\alpha_j)\), \(j=1,2\), be two pseudohermitian
  manifolds. We say that \((X_1,T^{1,0}X_1,\alpha_1)\) and
  \((X_2,T^{1,0}X_2,\alpha_2)\) are isomorphic if there exists a
  diffeomorphism \(F\colon X_1\to X_2\) such that \(F\) is CR
  (i.e. \(F_*T^{1,0}X_1=T^{1,0}X_2\)) and \(F^*\alpha_2=\alpha_1\).
\end{definition}
\begin{definition}\label{def:PseudoHermSubManiSphere}
  Let \(X\) be a smooth submanifold of the sphere
  \(\mathbb{S}^{2N-1}\subset \C^N\), \(N\geq2\), with
  \(\dim_\R X=2n+1\), \(n\geq 1\). Denote by \(\iota\colon X\to \C^N\)
  the inclusion map. Given a pseudohermitian structure
  \((T^{1,0}X,\alpha)\) on \(X\) we say that \((X,T^{1,0}X,\alpha)\)
  is a pseudohermitian submanifold of \(\mathbb{S}^{2N-1}\) if
  \(T^{1,0}X=\C TX\cap T^{1,0}\mathbb{S}^{2N-1}\) and
  \(\alpha=\iota^*\alpha_\beta\) for some \(\beta\in \R_+^{N}\) . We
  say that \((X,T^{1,0}X,\alpha)\) is a real analytic pseudohermitian
  submanifold of \(\mathbb{S}^{2N-1}\) if in addition \(X\) is a real
  analytic submanifold of \(\mathbb{S}^{2N-1}\).
\end{definition}

In order to state our main result we need the following notations. Let
\(X\) be a smooth manifold and \(c>0\). Put \(Y=X\times (-c,c) \) and
denote by \(\operatorname{pr}\colon Y\to X\), \((x,s)\mapsto x\) the
projection onto the first factor. Given a smooth function
\(\varphi\colon X\to (-c,c)\) we denote by
\(X_\varphi:=\{(x,s)\in Y\colon s=\varphi(x)\}\) the graph of
\(\varphi\) over \(X\) in \(Y\). We further denote by 
\(\iota_\varphi\colon X_\varphi\to Y\) the inclusion and by
\(\kappa_\varphi\colon X\to X_\varphi\) the diffeomorphism
\(\kappa_\varphi(x)=(x,\varphi(x)) \). Given a complex structure
\(T^{1,0}Y\) on \(Y\) we denote by
\(T^{1,0}X_{\varphi}=\C TX\cap T^{1,0}Y\) the induced CR structure on
\(X_\varphi\). 
\begin{theorem}\label{thm:MainThmIntroGeneral}
  Let \((X,T^{1,0}X,\alpha)\) be a compact pseudohermitian manifold of
  dimension \(\dim_\R X=2n+1\), \(n\geq 1\), and denote by \(\mathcal{T}\) the Reeb
  vector field associated to \(\alpha\). Assume that \((X,T^{1,0}X)\)
  is CR embeddable into the complex Euclidean space and choose
  \(0<\delta_1<\delta_2\). There is a \(c>0\), a complex structure
  \(T^{1,0}Y\) on \(Y=X\times (-c,c)\), non-negative integers
  \(\{N_k\}_{k\geq k_0}\), a family of weight vectors
  \(\{\beta(k)\in\R_+^{N_k}\}_{k\geq k_0}\), with
  \(\beta(k)_j\in[\delta_1k,\delta_2k]\), \(1\leq j\leq N_k\),
  \(k\geq k_0\), a family of smooth embeddings
  \(\{F_k\colon X\times[0,c)\to \C^{N_k}\}_{k\geq k_0}\) and a family
  of smooth functions \(\{\varphi_k:X\to (0,c)\}_{k\geq k_0}\)
  satisfying
  \begin{itemize}
  \item[-] \(\kappa_0\) is a CR diffeomorphism between
    \((X,T^{1,0}X)\) and
    \((X_0,T^{1,0}X_0)\), 
  \item[-] \(F_k\) is holomorphic on \(X\times (0,c)\) for all \(k\geq k_0\),
  \item[-] \(\|\varphi_k\|_{\mathscr{C}^m(X)}=O(k^{-2})\) for all
    \(m\in \N\),
  \item[-] \(N_k=O(k^{n+1})\),
  \item[-] \((x,k)\mapsto \varphi_k(x)\) is smooth on
    \(X\times[k_0,\infty)\),
  \item[-] \(\frac{\partial \varphi_k(x)}{\partial k}\leq -k^{-4}\)
    for all \(x\in X\) and \(k\geq k_0\),
  \end{itemize}
  such that the following holds: For each \(k\geq k_0\) we have that
  \((X_{\varphi_k}, T^{1,0}X_{\varphi_k},\gamma_k)\) with
  \(\gamma_k=\iota_{\varphi_k}^*F^*_k\alpha_{\beta(k)}\) is a real
  analytic pseudohermitian manifold such that
  \(F_k(X_{\varphi_k})\subset
  \mathbb{S}^{2N_k-1}\).  
  Moreover, let \(\{(\mathcal{V}^k,\alpha^k)\}_{k\geq k_0}\) be the
  induced family of pseudohermitian structures on \(X\) defined by
  \(\mathcal{V}^k=(\kappa_{\varphi_k})^{-1}_*T^{1,0}X_{\varphi_k}\)
  and \(\alpha_k=\kappa_{\varphi_k}^*\gamma_k\).  Put
  \(G_k=F_k\circ\kappa_{\varphi_k}\) and denote by \(\mathcal{T}^k\)
  the Reeb vector field associated to \(\alpha_k\). We then have
  \begin{itemize}
  \item[(i)] \(G_k\colon X\to \C^{N_k}\) is a CR embedding with
    respect to the CR structure \(\mathcal{V}^k\),
  \item[(ii)] \(G_k(X)\subset \mathbb{S}^{N_k-1}\),
  \item[(iii)] \(G_k^*\alpha_{\beta(k)}=\alpha^k\),
  \item[(iv)]
    \(\left|1-\frac{|(G_k)_*\mathcal{T}^{k}|}{|\mathcal{T}_{\beta(k)}\circ
        G_k||}\right|_{\mathscr{C}^m(X)}=O(k^{-1})\),
    \(\left|1-\frac{\langle(G_k)_*\mathcal{T}^{k},\mathcal{T}_{\beta(k)}\circ
        G_k\rangle}{|(G_k)_*\mathcal{T}^{k}||\mathcal{T}_{\beta(k)}\circ
        G_k|}\right|_{\mathscr{C}^m(X)}=O(k^{-1})\) for any
    \(m\in\N\).
  \item[(v)] \(\|\alpha-\alpha^k\|_{\mathscr{C}^m(X)} =O(k^{-1})\) and
    \(\|\mathcal{T}-\mathcal{T}^{k}\|_{\mathscr{C}^m(X)} =O(k^{-1})\)
    for any \(m\in\N\).
  \item[(vi)] \(\mathcal{V}^k\to T^{1,0}X\) in
    \(\mathscr{C}^\infty\)-topology as \(k\to \infty\).
  \item[(vii)] \(\mathcal{V}^k,\alpha^{k}\) and \(\mathcal{T}^k\)
    depend smoothly on \(k\).
  \end{itemize}
  As a consequence \((\mathcal{V}^k,\alpha^{k},\mathcal{T}^k)\)
  converges to \((T^{1,0}X,\alpha,\mathcal{T})\) in
  \(\mathscr{C}^\infty\)-topology for \(k\to\infty\) where each
  \((X,\mathcal{V}^k,\alpha^k)\) is isomorphic to a real analytic
  pseudohermitian submanifold of a sphere.
\end{theorem}
\begin{remark}{\color{white}{.}}
  \begin{itemize}\label{rmk:MainThmGeneral}
  \item [(1)] In Theorem~\ref{thm:MainThmIntroGeneral} we actually
    show that \(G_k(X)=F_k(X_{\varphi_k})\) is the transversal
    intersection of an \((n+1)\)-dimensional complex submanifold of
    \(\C^{N_k}\) and \(\mathbb{S}^{2N_k-1}\) (see
    Remark~\ref{rmk:RemarksProofGeneral}).
  \item [(2)] From
    \(\frac{\partial \varphi_k(x)}{\partial k}\leq -k^{-4}\) it
    follows that the family \(\{\varphi_k\}_{k\geq k_0}\) is strictly
    decreasing. Put
    \(K=\{(x,s)\in Y\colon 0\leq s < \varphi_{k_0}(x)\}\).  Then \(K\)
    is a complex manifold with boundary CR isomorphic to \(X\) and
    \(\{X_{\varphi_k}\}_{k\geq k_0}\) defines a smooth foliation of
    the set \(\operatorname{int}(K)\) of interior points of \(K\). In
    particular, \(\operatorname{int}(K)\) is foliated by strictly
    pseudoconvex hypersurfaces such that each leaf is real analytic in
    \(Y\) and CR embeddable into a sphere. Furthermore,
    \(\{X_{\varphi_k}\}_{k\geq k_0}\cup\{X_0\}\) with
    \(X_0=X\times\{0\}\) is a continuous foliation of \(K\).
  \item [(3)] By (vi) in Theorem~\ref{thm:MainThmIntroGeneral} we mean
    the convergence in the sense of
    Definition~\ref{def:BundleConvergence}. In principle, it means
    that the sections with values in a Grassmannian bundle defining
    \(\mathcal{V}^k\) and \(T^{1,0}X\) converge in
    \(\mathscr{C}^\infty\)-topology as \(k\to\infty\). In this sense
    we actually have \(\mathcal{V}^k=T^{1,0}X+O(k^{-1})\) in
    \(\mathscr{C}^\infty\)-topology.
  \item[(4)] By (vii) in Theorem~\ref{thm:MainThmIntroGeneral} we mean
    that on \(X\times[k_0,\infty)\), we have that
    \((x,k)\mapsto \mathcal{T}^k_x\) and \((x,k)\mapsto (\alpha^k)_x\)
    define smooth maps into \(TX\) and \(T^{*}X\) respectively.  The
    bundle \(\mathcal{V}^k\) depends smoothly on \(k\) in the sense of
    Definition~\ref{def:BundleConvergence}.
  \item[(5)] Let
    \(\ell^2(\C)=\{(a_j)_{j\in\N}\colon
    \sum_{j=1}^\infty|a_j|^2<\infty\}\).  Given \(N\geq 1\),
    \(1\leq m^-\leq m^+\) we denote by
    \(\operatorname{Pr}_{m^-,m^+}\colon \ell^2(\C)\to
    \C^{N+m^+-m^-+1}\) the projection
    \[(a_j)_{j\in \N}\mapsto
      (a_1,\ldots,a_N,a_{N+m^-},a_{N+m^-+1},\ldots,a_{N+m^+}).\] From
    the proof of Theorem~\ref{thm:MainThmIntroGeneral} (see
    Remark~\ref{rmk:RemarksProofGeneral}) it follows that there exist
    \(N\in \N\), a family of maps
    \(\{\hat{F}_k\colon X\times[0,c)\to \ell^2(\C)\}_{k\geq k_0}\) and
    a family \(\{(m^-_k,m^+_k)\}_{k\geq k_0}\) such that for each
    \(k\geq k_0\) we have that almost all entries of \(\hat{F}_k\) are
    zero, \(F_k=\operatorname{Pr}_{m^-,m^+}\circ \hat{F}_k\) and
    \(|F_k|^2=|\hat{F}_k|^2\). Furthermore, for any entry
    \((\hat{F}_k)_j\), \(j\in \N\), of \(\hat{F}_k\) we have that
    \((x,s,k)\mapsto (\hat{F}_k)_j(x,s)\) is smooth on
    \(X\times [0,c)\times [k_0,\infty)\). Roughly speaking, we have
    that the maps \(F_k\) and \(G_k\) in
    Theorem~\ref{thm:MainThmIntroGeneral} are induced by maps which
    depend smoothly on \(k\).
  \item[(6)] For the Sasakian case, as we see later in
    Theorem~\ref{thm:MainThmIntro}, it follows from the construction
    that \(\mathcal{T}=\mathcal{T}^k\) for all \(k\geq k_0\) and
    \((G_k)_*\mathcal{T}^k=\mathcal{T}_{\beta(k)}\circ G_k\). Hence,
    from (iv) in Theorem~\ref{thm:MainThmIntroGeneral} we find that
    \(G_k\) can be seen as an almost Sasakian embedding.
  \item[(7)] The techniques used in the proof of
    Theorem~\ref{thm:MainThmIntroGeneral} heavily rely on expansion
    results for the semi-classical functional calculus for Toeplitz
    operators on strictly pseudoconvex domains. This kind of calculus
    for abstract embeddable compact strictly pseudoconvex CR manifolds
    was first introduced by
    Herrmann--Hsiao-Marinescu--Shen~\cite{HHMS23} (see
    Herrmann--Hsiao--Li~\cite{HHL20} for the Sasakian case) and then
    considered for strictly pseudoconvex domains by
    Hsiao--Marinescu~\cite{HM23}. For the proof of
    Theorem~\ref{thm:MainThmIntroGeneral} we will combine the result
    of Hsiao--Marinescu~\cite{HM23} with the embedding machinery
    developed in Herrmann--Hsiao--Marinescu--Shen~\cite{HHMS23,
      HHMS24}.  We refer to Section~\ref{sec:IdeaProof} for more
    details.
  \end{itemize}

\end{remark}
  
From Theorem~\ref{thm:MainThmIntroGeneral} we obtain the following
corollary.
\begin{corollary}\label{cor:RealAnalyticPseudohermitianApproximation}
  Let \((X,T^{1,0}X,\alpha)\) be a compact pseudohermitian manifold
  such that \((X,T^{1,0}X)\) is CR embeddable into the complex
  Euclidean space. Then there exists a sequence of pseudohermitian
  structures \(\{(\mathcal{V}^k,\alpha^k)\}_{k\in \N}\) on \(X\) which
  approximate \((T^{1,0}X,\alpha)\) in \(\mathscr{C}^\infty\)-topology
  such that for each \(k\in \N\) we have that
  \((X,\mathcal{V}^k,\alpha^k)\) is isomorphic to a real analytic
  pseudohermitian submanifold of a sphere.
\end{corollary}
\subsection{Approximation of Strictly Pseudoconvex Domains}
As already mentioned above, Forstneri\v{c}~\cite{Fors86} and
L{\o}w~\cite{Low85} showed independently that any smoothly bounded
relatively compact strictly pseudoconvex domain in a Stein manifold
can be properly and holomorphically embedded into some finite
dimensional unit ball such that the embedding is continuous up to the
boundary. In general, those embeddings cannot be chosen to extend
smoothly up to the boundary due to the following result of
Forstneri\v{c}~\cite{Fors86}.
\begin{theorem}[cf. \cite{Fors86}*{Theorem 1.1}]
  Let \(M\) be a bounded domain in \(\C^n\), \(n\geq 2\), with
  \(\mathscr{C}^s\)-boundary, \(s\geq 1\). In every neighborhood of
  \(M\) in the \(\mathscr{C}^s\)-topology on domains there exists a
  domain \(D\) with smooth real-analytic boundary such that no proper
  holomorphic map \(F\colon D\to \mathbb{B}^N\), \(N>n\), extends
  smoothly to \(\overline{D}\).
\end{theorem}
In contrast to the result of Forstneri\v{c} we obtain from our methods
the following.
\begin{theorem}\label{thm:ContrastToForstneric}
  Let \(Y\) be a Stein manifold of dimension \(\dim_\C Y\geq 2\) and
  \(s\in \N\cup\{\infty\}\), \(s\geq 2\). Let \(M\subset \subset Y\)
  be a relatively compact strictly pseudoconvex domain with boundary
  of class \(\mathscr{C}^s\). In every neighborhood of \(M\) in the
  \(\mathscr{C}^s\)-topology for domains there exists a strictly
  pseudoconvex domain \(D\subset \subset M\) with real analytic
  boundary and a real analytic embedding
  \(F\colon \overline{D}\to \C^N\) for some \(N\in \N\) which is
  holomorphic on \(D\) such that \(F\) restricts to a proper
  holomorphic map of \(D\) into the unit ball of \(\C^N\).  In
  particular, the restriction of \(F\) to the boundary \(bD\) of \(D\)
  defines a real analytic CR embedding of \(bD\) into the unit sphere.
\end{theorem}

\begin{remark}{\color{white}{.}}
  \begin{itemize}
  \item [(1)] Theorem~\ref{thm:ContrastToForstneric} follows from the
    more general result
    Theorem~\ref{thm:ProofOnExhaustionBySphericalDom} (see also
    Corollary~\ref{cor:ContrastToForstneric}).
  \item [(2)] Replacing '\(\mathscr{C}^s\)-topology' with
    '\(\mathscr{C}^0\)-topology' in
    Theorem~\ref{thm:ContrastToForstneric}, the statement can be
    deduced from a result due to Forstneri\v{c} (see
    \cite{Fors86}*{Corollary 1.4}).
  \end{itemize}
\end{remark}

\subsection{Approximation of Sasakian Structures}
We will now consider CR embeddings into spheres in the context of
Sasakian manifolds.  Let \((X,T^{1,0}X,\mathcal{T})\) be a compact
Sasakian manifold of dimension \(\dim_\R X=2n+1\), \(n\geq 1\). In
particular, we have that \((X,T^{1,0}X)\) is a (connected) compact
strictly pseudoconvex CR manifold of hypersurface type and
\(\mathcal{T}\) is a Reeb vector field for the contact structure given
by \(\operatorname{Re}T^{1,0}X\) such that flow of \(\mathcal{T}\)
preserves the CR structure \(T^{1,0}X\) and the Levi form associated
to \(\mathcal{T}\) is positive definite (see
Section~\ref{sec:Setup}). The study of Sasakian manifolds plays an
important role in Kähler geometry as it can be seen as a CR analog of
a Kähler manifold in complex geometry. It is well known (see
Baouendi--Rothschild--Treves~\cite{BRT85}) that a Sasakian manifold is
locally CR equivalent to a real hypersurface in the complex Euclidean
space. The question whether such a Sasakian manifold can be realized
as a CR submanifold of a complex Euclidean space was first positively
answered by Lempert~\cite{Lem92} (see also
Marinescu--Yeganefar~\cite{MY07}).  Recently, it was shown by
Loi--Placini~\cite{LP22} that \((T^{1,0}X,\mathcal{T})\) can be
approximated in any \(\mathscr{C}^m\)-norm by quasi-regular Sasakian
structure which are globally embeddable into a sphere. More precisely,
they show that for any \(m\in \N\) there exists a sequence of quasi-regular
Sasakian structures
\(\{(\mathcal{V}^k,\mathcal{T}^k)\}_{k\in \N}\) on \(X\) converging to
\((T^{1,0}X,\mathcal{T})\) in \(\mathscr{C}^m\)-topology such that
each \((\mathcal{V}^k,\mathcal{T}^k)\) is induced by a weighted sphere
(see Definition~\ref{def:SasakianStructureSimpleDeformation}). From
Remark~~\ref{rmk:ThmSasakianManifoldMain}~(1) it follows that their
result immediately yields a result about the approximation of
pseudohermitian structures which are Sasakian.  We note that the
spheres \(\mathbb{S}^{2N-1}\), \(N\geq 2\), together with the strictly
pseudoconvex CR structure coming from \(\C^N\) carries plenty of
different Sasakian structures (see
Remark~\ref{rmk:SasakianStructuresOnSpheres}).  A Sasakian structure
is quasi-regular if all Reeb orbits are compact.  We will now define
the meaning of weighted sphere.

\begin{definition}\label{def:SasakianStructureSimpleDeformation}
  Given a Sasakian manifold \((X,T^{1,0}X,\mathcal{T})\), \(N\in\N\)
  and a vector \(\beta\in\R_+^N\) we say that the Sasakian structure
  is induced by a \(\beta\)-weighted sphere if there exists a CR
  embedding \(F\colon X\to \C^N\) such that
  \(F(X)\subset \mathbb{S}^{2N-1}\) and
  \(F_*\mathcal{T}=\mathcal{T}_\beta\circ F\) where \(\mathcal{T}_\beta\) is
  as in Example~\ref{ex:PseudoHermSpheres}.  We say that the Sasakian
  structure is induced by a weighted sphere if the Sasakian structure
  is induced by a \(\beta\)-weighted sphere for some \(N\in\N\) and
  \(\beta\in \R_+^N\).
\end{definition}
Let \((X,T^{1,0}X,\mathcal{T})\) be a Sasakian manifold.  Given a
smooth function \(\varphi\colon X\to \R\) with
\(\mathcal{T}\varphi\equiv 0\) we define
\begin{eqnarray}\label{eq:DefDeformedSasakianStructure}
  \mathcal{V}(\varphi)=\{Z-iZ(\varphi)\mathcal{T}\mid Z\in T^{1,0}X\}.
\end{eqnarray}
It turns out (see Section~\ref{sec:Setup}) that
\((X,\mathcal{V}(\varphi),\mathcal{T})\) is again a Sasakian manifold
provided that the \(\mathscr{C}^1\)-norm of \(d\varphi\) is
sufficiently small.  Our main theorem for Sasakian manifolds is the
following.
\begin{theorem}\label{thm:MainThmIntro}
  Let \((X,T^{1,0}X,\mathcal{T})\) be a compact Sasakian manifold of
  dimension \(2n+1\), \(n\geq 1\), and let \(C>0\) be a constant.
  There exist \(k_0>0\), a sequence of smooth functions
  \(\{\varphi_k\}_{k\in \N}\subset \mathscr{C}^\infty(X,\R)\) with
  \(\mathcal{T}\varphi_k\equiv 0\) and
  \(\|\varphi_k\|_{\mathscr{C}^m(X)}=O(k^{-2})\) for any \(m\in \N_0\),
  a sequence of positive integers \(\{N_k\}_{k\in\N}\) with
  \(N_k=O(k^{n+1})\) and a sequence of weight vectors
  \(\{\beta(k)\}_{k\in\N}\), \(\beta(k)\in\R_+^{N_k}\), with
  \(\beta(k)_j\in[C,k]\), \(1\leq j\leq N_k\), such that for each
  \(k\geq k_0\) we have that \((X,\mathcal{V}(\varphi_k),\mathcal{T})\)
  is a Sasakian manifold where
  \((\mathcal{V}(\varphi_k),\mathcal{T})\) is induced by a
  \(\beta(k)\)-weighted sphere.
\end{theorem}
\begin{remark}\label{rmk:ThmSasakianManifoldMain}{\color{white}{.}}
  \begin{itemize}
  \item [(1)] Given a Sasakian manifold \((X,T^{1,0}X,\mathcal{T})\)
    then the 
    form \(\alpha\) uniquely defined by
    \(\mathcal{T}\lrcorner\alpha=1\) and
    \(\ker\alpha=\operatorname{Re}T^{1,0}X\) defines a pseudohermitian
    structure \((T^{1,0}X,\alpha)\) on \(X\) such that \(\mathcal{T}\)
    is the Reeb vector field associated to \(\alpha\). Furthermore,
    given \(N\in \N\), \(\beta\in \R^N_+\) and a CR embedding
    \(F\colon X\to \C^N\) with \(F(X)\subset \mathbb{S}^{2N-1}\) and
    \(F_*\mathcal{T}=\mathcal{T}_\beta\circ F\) it automatically
    follows that \(F^*\alpha_\beta=\alpha\). Here
    \(\mathcal{T}_\beta\) and \(\alpha_\beta\) are as in
    Example~\ref{ex:PseudoHermSpheres}. Hence
    Theorem~\ref{thm:MainThmIntro} can be also seen as a Sasakian
    version of Theorem~\ref{thm:MainThmIntroGeneral}.
  \item[(2)] We note that it follows from the proof that the sequence
    \(\{\varphi\}_{k\in\N}\) in Theorem~\ref{thm:MainThmIntro} can
    also replaced by a family \(\{\varphi_k\}_{k\geq k_0}\) which is
    smooth in \(k\) similar as in
    Theorem~\ref{thm:MainThmIntroGeneral} (see
    Theorem~\ref{thm:ConstructionOfVarphik}).
  \item[(3)] Define \(Y=X\times \R\),
    \(T^{1,0}Y=T^{1,0}X\oplus\C \{\frac{\partial}{\partial
      s}-i\mathcal{T}\}\) and extend \(\mathcal{T}\) trivially to
    \(Y\). Then \((Y,T^{1,0}Y)\) is a complex manifold and
    \(X\ni x\mapsto (x,0)\in M\) defines a CR embedding of \(X\) into
    \(Y\). Given a smooth function \(\varphi\colon X\to \R\) with
    \(\mathcal{T}\varphi\equiv 0\) put
    \(X_\varphi=\{(x,s)\in Y\mid s=\varphi(x)\}\) and
    \(T^{1,0}X_\varphi:=\C TX_\varphi\cap T^{1,0}Y\). Provided that
    the \(\mathscr{C}^1\)-norm of \(d\varphi\) is sufficiently small
    one can show (see Section~\ref{sec:Setup}) that
    \((X_{\varphi},T^{1,0}X_\varphi,\mathcal{T})\) is a Sasakian
    manifold which is isomorphic to
    \((X,\mathcal{V}(\varphi),\mathcal{T})\). In the proof of
    Theorem~\ref{thm:MainThmIntro} we construct a sequence
    \(\{F_k\colon Y\to \C^{N_k}\}_{k\in \N}\) which restrict to
    holomorphic embeddings of a neighborhood \(U\subset Y\) around
    \(X\times\{0\}\) such that
    \(F_k(X_{\varphi_k})\subset\mathbb{S}^{2N_k-1}\) and
    \((F_k)_*\mathcal{T}=\mathcal{T}_{\beta(k)}\circ F_k\) on
    \(X_{\varphi_k}\) for all \(k\in \N\). One can show that each
    \(X_{\varphi_k}\) is a real analytic submanifold of \(Y\).  See
    Section~\ref{sec:ProofOfMainThmSasakian} for more details.
  \item[(4)] In the Sasakian case, given a smooth CR embedding
    \(F\colon X\to \C^{N}\) with \(F(X)\subset \mathbb{S}^{2N-1}\) and
    \(F_*\mathcal{T}=\mathcal{T}_\beta\circ F\) for some \(N\in \N\) and
    \(\beta\in \R_+^{N}\) it follows automatically (e.g.~from the construction in (3))
     that \(F(X)\) is real analytic (see
    Lemma~\ref{lem:SasakianSubmanifoldsOfSpheresAreRealAnalytic}).
  \item[(5)] The techniques used in the proof of
    Theorem~\ref{thm:MainThmIntro} rely heavily on expansion results
    for the semi-classical functional calculus on Sasakian manifolds
    established by Herrmann--Hsiao--Li~\cite{HHL20} (see
    also~\cite{HHL22}) in combination with applications to CR
    embeddings developed by
    Herrmann--Hsiao--Marinescu--Shen~\cite{HHMS23}. We refer to
    Section~\ref{sec:IdeaProof} for more details.
  \end{itemize}
	  
\end{remark}
Since \((T^{1,0}X,\mathcal{T})\) can be approximated by quasi-regular
Sasakian structures of the form \((T^{1,0}X,\mathcal{T}^k)\) in
\(\mathscr{C}^\infty\)-topology (see Rukimbira~\cite{Ru95}) we recover
as a consequence the result of Loi--Placini.
\begin{corollary}[cf. \cite{LP22}*{Theorem 1}]\label{cor:Loi-Placini}
  Let \((X,T^{1,0}X,\mathcal{T})\) be a compact Sasakian manifold and
  \(q\in \N\). There exists a sequence
  \(\{(\mathcal{V}^k,\mathcal{T}^k)\}_{k\in \N}\) of quasi-regular
  Sasakian structures on \(X\) converging to
  \((T^{1,0}X,\mathcal{T})\) in \(\mathscr{C}^q\)-norm as
  \(k\to\infty\) such that for each \(k\in \N\) we have that
  \((\mathcal{V}^k,\mathcal{T}^k)\) is induced by a weighted sphere.
\end{corollary}
Comparing Theorem~\ref{thm:MainThmIntro} with the result of
Loi--Placini we see that in Theorem~\ref{thm:MainThmIntro} we do not
need to change the Reeb vector field \(\mathcal{T}\) which in
principle shows that one can preserve more symmetries during the
approximation. This is due to the fact that our method relies on
semi-classical Szeg\H{o} kernel asymptotics for strictly pseudoconvex
CR manifolds. The method of Loi--Placini uses the Bergman kernel
expansion for cyclic complex orbifolds due to
Ross--Thomas~\cite{RT11-1, RT11-2} (see also
Dai--Liu--Ma~\cite{DLM12}) which makes it necessary to switch to
quasi-regular Sasakian manifolds during the approximation.  In addition, from the \(\mathscr{C}^\infty\)-smoothness of the semi-classical Szeg\H{o} kernel asymptotics we obtain a family of Sasakian structures converging to the original one in \(\mathscr{C}^\infty\)-topology.  Moreover,
Theorem~\ref{thm:MainThmIntro} provides an explicit analytic
description of the Sasakian structures which approximate the original
one. This observation leads to some insights about deformations of
Sasakian structures which we are going to explain now.

As already mentioned above, it is known (see
Forstneri\v{c}~\cite{Fors86}) that general strongly pseudoconvex CR
manifolds are 'almost never' CR embeddable into finite dimensional
spheres. Following the idea of Forstneri\v{c} we show that a similar
statement also holds for the space of certain Sasakian
deformations. More precisely, we consider deformations of the form
\(\mathcal{V}(\varphi)\) with \(\varphi\in\mathscr{C}^\infty(X,\R)\)
such that \(\mathcal{T}\varphi\equiv 0\) and the
\(\mathscr{C}^1\)-norm of \(d\varphi\) is sufficiently small. Since
\[\mathcal{V}(\varphi)=\mathcal{V}(\psi)\,\,\Leftrightarrow
  \varphi-\psi\equiv \text{ const.}\] we should cut out all constant
functions. Therefore, consider
\[\mathcal{M}=\{\varphi\in \mathscr{C}^\infty(X,\R)\colon
  \mathcal{T}\varphi\equiv 0, \,\mathcal{V}(\varphi) \text{ is
    strictly pseudoconvex}\}/\R\cdot1\!\!1\] where
\(1\!\! 1\colon X\to\R\), \(1\!\! 1(x)=1\), denotes the constant one
function and let \(\mathcal{S}\) be the subset defined by
\begin{eqnarray*}
  \mathcal{S}&=&\{[\varphi]\in \mathcal{M}\colon (\mathcal{V}(\varphi),\mathcal{T})\text{ is induced by a weighted sphere}\}.
\end{eqnarray*}
We equip \(\mathcal{M}\) with the quotient topology induced by the
\(\mathscr{C}^\infty\)-topology on
\(\mathscr{C}^\infty(X,\R)\). Adapting the ideas of Forstneri\v{c} to
the Sasakian case we obtain the following.
\begin{theorem}\label{thm:ForstnericForSasakian}
  Let \((X,T^{1,0}X,\mathcal{T})\) be a compact Sasakian manifold and
  \(\mathcal{S},\mathcal{M}\) as above. We have that \(\mathcal{S}\)
  is a set of first category in \(\mathcal{M}\).
\end{theorem}
Combining Theorem~\ref{thm:ForstnericForSasakian} with
Theorem~\ref{thm:MainThmIntro} we obtain the following corollary.
\begin{corollary}\label{cor:ForstnericForSasakian}
  Let \((X,T^{1,0}X,\mathcal{T})\) be a compact Sasakian manifold and
  \(\mathcal{S},\mathcal{M}\) as above. Then \(\mathcal{S}\) is dense
  and of first category in \(\mathcal{M}\).
\end{corollary}
We note that for the Sasakian case the conclusion of
Theorem~\ref{thm:ForstnericForSasakian} is less surprising compared to
the result of Forstneri\v{c}~\cite{Fors86} for the general case since
any Sasakian submanifold of a sphere has to be automatically real
analytic (see
Lemma~\ref{lem:SasakianSubmanifoldsOfSpheresAreRealAnalytic}). We
still provide a proof for this statement since its method leads to
structural insights about the set of Sasakian deformations which are
induced by weighted spheres as the following result will show.

For \(N\in \N\) fixed, consider the set
\[\mathcal{S}_N=\{[\varphi]\in \mathcal{M}\colon
  (\mathcal{V}(\varphi),\mathcal{T})\text{ is induced by a
    \(\beta\)-weighted sphere for some \(\beta\in\R_+^N\)}\}\] and for
any \(\beta\in \R^{N}_+\) define
\[\mathcal{S}_\beta=\{[\varphi]\in \mathcal{S}_N\colon
  (\mathcal{V}(\varphi),\mathcal{T})\text{ is induced by a
    \(\beta\)-weighted sphere} \}.\] One has that
\(\mathcal{S}=\bigcup_{N=1}^\infty\mathcal{S}_N\). For the proof of
Theorem~\ref{thm:ForstnericForSasakian} we will basically show that
for any \(N\in\N\) we have that \(\mathcal{S}_N\) is nowhere dense in
\(\mathcal{M}\) (see Lemma~\ref{lem:NonDensitySN}). Since
\(\mathcal{S}_\beta\subset \mathcal{S}_N\) it follows that the same
holds true for \(\mathcal{S}_\beta\).  However, it turns out that
\(\mathcal{S}_\beta\) has no isolated points. We have the following.
\begin{theorem}\label{thm:NoIsolatedPoint}
  Fix \(N\in \N\) and \(\beta\in \R_+^N\). Let
  \((X,T^{1,0}X,\mathcal{T})\) be a compact Sasakian manifold and
  \(\mathcal{S}_\beta,\mathcal{M}\) as above. We have that
  \(\mathcal{S}_\beta\) is nowhere dense in
  \(\mathcal{M}\). Furthermore, if \(\mathcal{S}_\beta\neq\emptyset\)
  we have that \(\mathcal{S}_\beta\subset \mathcal{M}\) has no
  isolated points. In particular, for any
  \([\varphi]\in\mathcal{S}_\beta\) there exist \(\varepsilon>0\) and
  a continuous non-constant
  \(\gamma\colon [0,\varepsilon)\to \mathcal{M}\) with
  \(\gamma(0)=\varphi\) and
  \(\gamma([0,\varepsilon))\subset \mathcal{S}_\beta\).
\end{theorem}
It might happen that given an element
\([\varphi]\in\mathcal{S}_\beta\) we find other elements
\([\psi]\in\mathcal{S}_\beta\) that are close to \([\varphi]\) with
\([\psi]\neq [\varphi]\) defining Sasakian structures
\((\mathcal{V}(\psi),\mathcal{T})\) which are isomorphic to
\((X,\mathcal{V}(\varphi),\mathcal{T})\) as Sasakian manifolds. Such
an issue could appear as the following example will show.
\begin{example}\label{ex:IsomorphicSasakianStructures}
  Consider the three-sphere
  \(X:=\mathbb{S}^3=\{(z,w)\in \C^2\mid |z|^2+|w|^2=1\}\subset \C^2\)
  with CR structure \(T^{1,0}X=\C TX\cap T^{1,0}\C^2\). Then choosing
  \[\mathcal{T}=i\left(z\frac{\partial}{\partial
        z}-\overline{z}\frac{\partial}{\partial
        \overline{z}}\right)+i\left( w\frac{\partial}{\partial
        w}-\overline{w}\frac{\partial}{\partial \overline{w}}\right)\]
  on \(X\) leads to a compact Sasakian manifold
  \((X,T^{1,0}X,\mathcal{T})\). We call \((T^{1,0}X,\mathcal{T})\) the
  standard Sasakian structure on \(X\). Taking
  \(F\colon X\to \mathbb{S}^{3}\), \(F(z,w)=(z,w)\) shows that
  \((T^{1,0}X,\mathcal{T})\) is induced by a \((1,1)\)-weighted
  sphere.  Hence, we find \([0]\in \mathcal{S}_{(1,1)}\). Now, for
  \(\varepsilon>0\) consider
  \(\varphi_\varepsilon(z,w)=-\log\sqrt{1+\varepsilon |z|^2}\). Then
  the map \(F_\varepsilon\colon X\to \C^2\),
  \(F_\varepsilon=(e^{\varphi_\epsilon(z,w)}(\sqrt{1+\varepsilon})z,e^{\varphi_\epsilon(z,w)}w)\)
  is CR with respect to the CR structure
  \(\mathcal{V}(\varphi_\varepsilon)\), satisfies
  \[|F_\varepsilon(z,w)|^2=(1+\varepsilon|z|^2)^{-1}\left((1+\varepsilon)|z|^2+|w|^2\right)=1,\,\,\,(z,w)\in
    X\] and
  \((F_\varepsilon)_*\mathcal{T}=\mathcal{T}_{(1,1)}\circ
  F_\varepsilon\). It follows that
  \([\varphi_\varepsilon]\in \mathcal{S}_{(1,1)}\) for all
  sufficiently small \(\varepsilon>0\) and
  \([\varphi_\varepsilon]\neq [0]\). Since in addition we have that
  \(F_\varepsilon\colon X\to \mathbb{S}^{3}\) has to be one to one for
  all sufficiently small \(\varepsilon>0\) we have that
  \(F_\varepsilon\) defines an isomorphism between Sasakian manifolds
  when \(\varepsilon >0\) is small. As a consequence we obtain that
  \(\{(\mathcal{V}(\varphi_\varepsilon),\mathcal{T})\}_{0<\varepsilon<<1}\)
  is a family of Sasakian structures on \(X\) which are induced by a
  \((1,1)\)-weighted sphere and are all isomorphic to the standard one
  but not equal.
\end{example}
Example~\ref{ex:IsomorphicSasakianStructures} leads to the following
question.
\begin{problem}
  Fix \(N\in \N\) and \(\beta\in \R_+^N\). Let
  \((X,T^{1,0}X,\mathcal{T})\) be a compact Sasakian manifold and
  \(\mathcal{S}_\beta,\mathcal{M}\) as above. Assume that \(N\in \N\)
  is sufficiently large and that there is a CR embedding
  \(F=(f_1,\ldots,f_N)\colon X\to \C^N\) with
  \(F(X)\subset \mathbb{S}^{2N-1}\) and
  \(F_*\mathcal{T}=\mathcal{T}_\beta\circ F\) such that
  \(f_j\not\equiv 0\) for \(1\leq j\leq N\). Is there a sequence
  \({[\varphi_k]}_{k\in\N}\in \mathcal{S}_\beta\) which converges to
  \([0]\) in \(\mathcal{M}\) such that \((X,T^{1,0}X,\mathcal{T})\)
  and \((X,\mathcal{V}(\varphi_k),\mathcal{T})\) are not isomorphic as
  Sasakian manifolds for any \(k\in \N\)?
\end{problem}

\subsection{Local Approximation Results}\label{sub:further}
In this section we note that we obtain certain local results from
our global results. Indeed,  let us 
assume that $(X,T^{1,0}X)$ is a not necessarily
compact, strictly pseudoconvex CR manifold of real
dimension $2n+1$ as before, which is locally CR embeddable
in the sense that for every $p\in X$ there exist a neighborhood
$U$ of $p$ in $X$ and a CR map $H\colon U \to \mathbb{C}^{n+1}$
defining a CR diffeomorphism onto its image (note that since
we are in the local setting here, the dimension of the target space
can be chosen to be $n+1$ without losing generality). We then have
the following results.

\begin{corollary}\label{cor:LocalEmbedding01Intro} Let $(X, T^{1,0} X)$ be a locally embeddable
	strictly pseudoconvex CR manifold. Then every $p\in X$
	has a neighborhood $U$ with the following property: There
	exists a sequence of strictly pseudoconvex CR structures
	\(\{\mathcal{V}^k\}_{k\in \N}\) on \(U\) which approximate
	\(T^{1,0}X\big|_U\) in \(\mathscr{C}^\infty\)-topology such that for each
	\(k\in \N\) there is \(N_k\in \N\) and a smooth embedding
	\(F_k\colon U\to \C^{N_k}\) which is CR with respect to
	\(\mathcal{V}^k\) and satisfies
	\(F_k(U)\subset \mathbb{S}^{2N_k-1}\). 
\end{corollary}
The result follows in a straightforward way from Theorem~\ref{thm:MainThmIntroGeneralTeaser}.   We first
choose a neighborhood $V$ of $p$ and coordinates $(z,w)\in \mathbb{C}^n \times \mathbb{C}$, such
that the embedding functions $(z(x), u(x) + iv(x))$ map $V$ onto a smooth strictly convex hypersurface
$M\subset \mathbb{C}^{n+1}$. After possibly shrinking $V$ again, we can assume that $M$ is
part of a smooth strictly convex hypersurface \(S\) bounding a bounded domain (see Amar~\cite{Amar84} and McNeal~\cite{McNeal92}) and apply the global theorem.

\begin{corollary}\label{cor:LocalEmbedding02Intro} Let $(X, T^{1,0} X,\alpha)$ be a pseudohermitian manifold such that \((X,T^{1,0}X)\) is locally embeddable. Then every $p\in X$
	has a neighborhood $U$ with the following property: There
	exists a sequence of pseudohermitian structures
	\(\{(\mathcal{V}^k,\alpha^k)\}_{k\in \N}\) on \(U\) which approximate
	\((T^{1,0}X,\alpha)\) in \(\mathscr{C}^\infty\)-topology on \(U\) such that for each
	\(k\in \N\) we have that \((U,\mathcal{V}^k,\alpha^k)\) is isomorphic to a real analytic pseudohermitian submanifold of a sphere. 
\end{corollary}
The result follows from Corollary~\ref{cor:RealAnalyticPseudohermitianApproximation} modifying the argument in the proof of Corollary~\ref{cor:LocalEmbedding01Intro}. 
Denote by \(T^{1,0}S\) the CR structure on \(S\) induced by \(\C^{n+1}\) and choose a contact form \(\gamma\) for \(\operatorname{Re}T^{1,0}S\) such that the respective Levi form is positive definite. Identifying \(V\) with \(M\) via the local CR embedding we can assume that  \(\alpha\) is a contact form for \(\operatorname{Re}T^{1,0}S\big|_M\). It is not hard to see (assuming that \(M\) is connected) that there is a positive smooth function \(f\colon M\to \R_+\) such that \(f\alpha=\gamma\) on \(M\). Hence after shrinking \(M\) we can modify \(\gamma\) such that \(\alpha=\gamma\) holds on \(M\) and apply  Corollary~\ref{cor:RealAnalyticPseudohermitianApproximation} to \((S,T^{1,0}S,\gamma)\).

\begin{corollary}\label{cor:LocalEmbedding03Intro} Let $(X, T^{1,0} X,\mathcal{T})$ be a Sasakian manifold. Then every $p\in X$
	has a neighborhood $U$ with the following property: There
	exists a sequence of Sasakian structures
	\(\{(\mathcal{V}^k,\mathcal{T}|_U)\}_{k\in \N}\) on \(U\) which approximate
	\((T^{1,0}X,\mathcal{T})\) in \(\mathscr{C}^\infty\)-topology on \(U\) such that for each
	\(k\in \N\) we have that \((\mathcal{V}^k,\mathcal{T}|_U)\) is induced by a weighted sphere. 
\end{corollary}
The result follows from Theorem~\ref{thm:MainThmIntro} using that Sasakian structures have 'good' local coordinates in the sense of complex geometry and the fact that the circle bundle of a positive holomorphic line bundle over a compact complex manifold together with the canonical \(S^1\)-action on the fiber defines a compact Sasakian manifold. We provide a proof for Corollary~\ref{cor:LocalEmbedding03Intro} in Section~\ref{sec:ProofCorollarySasakainLocal}.

\subsection{Organization of the Paper}
The paper is organized as follows. After sketching the idea of the
proof for Theorem~\ref{thm:MainThmIntroGeneral} and
Theorem~\ref{thm:MainThmIntro} in Section~\ref{sec:IdeaProof} we
recall some basic definitions and fundamental facts from complex
analysis and CR geometry in Section~\ref{sec:Preliminaries}. In
Section~\ref{sec:ProofOfMainThmSasakian} we prove
Theorem~\ref{thm:MainThmIntro}. Section~\ref{sec:NONWeightedSpheres}
and Section~\ref{sec:DeformationWeightedSphere} contain the proofs of
Theorem~\ref{thm:ForstnericForSasakian} and
Theorem~\ref{thm:NoIsolatedPoint}, respectively. In
Section~\ref{sec:ProofOfApproximationGeneral} we prove
Theorem~\ref{thm:MainThmIntroGeneral} and
Theorem~\ref{thm:ContrastToForstneric}.

\section{Idea of the Proofs}\label{sec:IdeaProof}
\subsection{Bergman Kernels, Szeg\H{o} Kernels and Asymptotic
  Expansions}
The proof of the approximation results
Theorem~\ref{thm:MainThmIntroGeneral} and
Theorem~\ref{thm:MainThmIntro} rely on the semi-classical functional
calculus for Toeplitz operators on CR manifolds
\cite{HHL20,HHL22,HHMS23} and on compact manifolds with boundary
\cite{HM23} in combination with their applications to embedding
problems \cite{HHL22,HHMS23, HHMS24}. The semi-classical functional
calculus for Toeplitz operators can be seen as a CR analog of Bergman
kernel expansion. The Bergman kernel for one complex variable was
introduced by Stefan Bergman~\cite{Berg33}. Due to the famous and
fundamental works of Hörmander~\cite{Ho65}, Diederich~\cite{Di70},
Kerzman~\cite{KER71}, Fefferman~\cite{Fef74}, Boutet de
Monvel--Sjöstrand~\cite{BS75} and others (see
Greene--Kim--Krantz~\cite{GKK11} and Ohsawa~\cite{Ohsawa2018} for more
details) it became an important subject in several complex variables
and geometry. An interesting phenomenon is that considering the
Bergman kernel \(B_k\) for the tensor power \((L^k,h^h)\),
\(k\in \N\), of some positive holomorphic line bundle \(L\) with
Hermitian metric \(h\) over a compact complex manifold \(M\),
\(\dim_\C M=n\), it follows from the results by Tian~\cite{Ti90},
Ruan~\cite{Ruan98}, Catlin~\cite{Catlin1999}, Zelditch~\cite{Zel98}
(see also Berman--Berndtsson--Sjöstrand~\cite{BBS08},
Dai--Liu--Ma~\cite{DLM06} and Ross--Thomas~\cite{RT11-1}) that the
integral kernel \(B_k(z,w)\) of \(B_k\) admits a full asymptotic
expansion in \(k\) for \(k\to \infty\) (we refer to the book of
Ma--Marinescu~\cite{MM07} for a complete reference). For example, on
the diagonal in \(M\times M\), we find smooth functions
\(\{b_j\}_{j\in \N_0}\subset \mathscr{C}^\infty(M)\) such that
\[B_k(z,z)\sim k^{n}b_0(z)+k^{n-1}b_1(z)+\ldots,\,\,\,\text{ in
  }\mathscr{C}^\infty\text{-topology for }k\to \infty.\] Furthermore,
the functions \(b_j\) encode the local geometry of \((M,L,h)\) (see
for example Xu~\cite{XH12} and Herrmann~\cite{He18} for explicit
formulas of the \(b_j\)'s).  Beside numerous of applications in
complex geometry, random geometry and mathematical physics
(see~\cite{MM07} and the references therein) the Bergman kernel
asymptotic expansion turns out to be a useful tool for studying
holomorphic embeddings of complex manifolds (see for example
Donaldson~\cite{Do01,Do05,Do09}). The method for deriving embedding
results from those kernel asymptotics is due to Bouche~\cite{Bch96}
and was applied to a various kind of set-ups by
Shiffman--Zelditch~\cite{SZ02}, Ma--Marinescu~\cite{MM07}, Hsiao, Li
and others \cite{Hs15,HLM21,HHL22}.  In CR geometry the analog of the
Bergman kernel is the Szeg\H{o} kernel. Since in general it is singular
on the diagonal (see~\cite{BS75}) it is much harder to derive
geometric applications from it (see for example
Hsiao--Marinescu~\cite{HM14,HM17} and Hsiao--Shen~\cite{HS20}). In
order to obtain smooth objects one can consider group actions on the
manifold and the Szeg\H{o} kernel for subspaces of equivariant
functions (see for example Herrmann--Hsiao--Li~\cite{HHL16,HHL22},
Paoletti--Galasso~\cite{GP18}, Shen~\cite{Sh19},
Hsiao--Li--Marinescu~\cite{HLM21},
Fritsch--Herrmann--Hsiao~\cite{FHH22}). Another approach is to study
the functional calculus of Toeplitz operators introducing a
semi-classical parameter. For example, let \((X,T^{1,0}X, \alpha)\) be
a compact, embeddable, pseudohermitian manifold. More precisely, we
have that \((X,T^{1,0}X)\) is a compact, strictly pseudoconvex CR
manifold which is CR embeddable into the complex Euclidean space and
\(\alpha\) is a contact form for \(\operatorname{Re}T^{1,0}X\) such
that its associated Levi form is positive definite. Let
\(\mathcal{T}\) denote the Reeb vector field with respect to
\(\alpha\) and consider the CR Toeplitz operator
\[T:=\Pi (-i\mathcal{T})\Pi\colon \mathscr{C}^\infty(X)\to
  \mathscr{C}^\infty(X) \] where \(\Pi\colon L^2(X)\to H_b^0(X)\)
denotes the Szeg\H{o} projection associated to an \(L^{2}\)-inner
product arising from the volume form \(dV:=\alpha\wedge (d\alpha)^n\)
on \(X\). It follows from Kohn~\cite{Koh63,Koh64} that \(\Pi\)
preserves the space of smooth functions and hence that \(T\) is well
defined. Moreover, CR Toeplitz operators in a much wider sense where
intensively studied by Boutet de Monvel--Guillemin~\cite{BG81}. As
part of their results (see also~\cite{HHMS23}) it follows that \(T\)
has a self-adjoint extension to a densely defined operator on
\(L^2(X)\) and admits a 'good' spectral decomposition. Hence, for any
smooth compactly supported function
\(\eta\in \mathscr{C}_c^\infty(\R)\),  the functional
calculus \(\eta(T)\) possesses a smooth integral kernel
\(\eta(T)(x,y)\). In order to introduce a semi-classical parameter
\(k>0\) into the method we fix \(\eta\in \mathscr{C}_c^\infty(\R_+)\)
and put \(\eta_k(\lambda)=\eta(k^{-1}\lambda)\), \(k,\lambda>0\). The
semi-classical functional calculus of \(T\) is then given by the
integral kernel \(\eta_k(T)(x,y)\) of \(\eta_k(T)\) for \(k>0\). It
was shown by Herrmann--Hsiao--Marinescu--Shen~\cite{HHMS23} that
\(\eta_k(x,y)\) admits a smooth full asymptotic expansion for
\(k\to \infty\) (see also the earlier works of
Herrmann--Hsiao--Li~\cite{HHL20,HHL22} for rigid CR manifolds). For
instance, on the diagonal in \(X\times X\), it follows from the
results in~\cite{HHMS23} that we find smooth functions
\(\{a_j\}_{j\in \N_0}\subset \mathscr{C}^\infty(X)\) such that
\[\eta_k(T)(x,x)\sim k^{n+1}a_0(x)+k^{n}a_1(x)+\ldots,\,\,\,\text{ in
  }\mathscr{C}^\infty\text{-topology for }k\to \infty\] where \(a_0\)
is an explicitly given positive constant.  Moreover, from the results
by Chang--Herrmann--Hsiao~\cite{CHH24} it follows that there exist
smooth functions \(\{b_j\}_{j\in \N_0}\subset \mathscr{C}^\infty(X)\)
independent of \(\eta\) such that
\(a_j=a_j^\eta=b_j\int_\R\eta(t)t^{n-j}dt\) for all \(j\in \N_0\) and
\(b_j(x)\) is related to the local geometry of
\((X,T^{1,0}X,\alpha,dV)\) near \(x\in X\). Using the asymptotic
expansion of the semi-classical functional calculus, first geometric
applications were provided in~\cite{HHMS23,HHMS24} with respect to
spectral measures, CR embeddings and the equidistribution of zeros of
random CR functions. Furthermore, the semi-classical functional
calculus for Toeplitz operators on smoothly bounded strictly
pseudoconvex domains was considered by Hsiao--Marinescu in
\cite{HM23}. One of the key observation in their work is that the
asymptotics in the domain case can be described by the semi-classical
functional calculus of a certain kind of generalized CR Toeplitz operator
on the boundary. This observation in combination with the results and
methods in~\cite{HHL20,HHMS23,HM23,HHMS24} will be used to prove the
approximation results Theorem~\ref{thm:MainThmIntroGeneral} and
Theorem~\ref{thm:MainThmIntro} in this paper.
 
\subsection{The Sasakian Case} 
We start by explaining the method for the Sasakian case
(Theorem~\ref{thm:MainThmIntro}) since it is less technical and the
idea can be stated more clearly.  Let \((X,T^{1,0}X,\mathcal{T})\) be
a compact Sasakian manifold with \(\dim_\R X=2n+1\), \(n\geq 1\).  We
denote by \(0<\lambda_1\leq\lambda_2\leq\ldots\) the positive
eigenvalues of the operator \(-i\mathcal{T}\) acting on CR functions
counting multiplicity and let \(\{f_j\}_{j\in \N}\) be an orthonormal
system of corresponding eigenfunctions. Here we use an \(L^2\)-inner
product on \(X\) induced by a volume form \(dV\) which is a constant
multiple of \(\alpha\wedge (d\alpha)^n\) where \(\alpha\) is the
contact form associated to \(\mathcal{T}\). Given a function
\(\eta\in \mathscr{C}_c^\infty(\R_+)\) put \(\eta_k(t)=\eta(k^{-1}t)\),
\(t,k>0\) and consider the integral kernel of the semi-classical
functional calculus for the Toeplitz operator \(-i\mathcal{T}\) acting
on CR functions given by
\[\eta_k(-i\mathcal{T})(x,y)=\sum_{j=1}^\infty
  \eta_k(\lambda_j)f_j(x)\overline{f_j(y)},\,\,(x,y)\in X\times X,\,k>0.\] It follows
from~\cite{HHMS23} (see also~\cite{HHL22}) that this operator can be
expressed as a semi-classical Fourier integral operator up to an
\(O(k^{-\infty})\). Furthermore, it was shown in~\cite{HHL20} (see
also~\cite{HHMS23}) that \(\eta_k(-i\mathcal{T})(x,y)\) admits a full
asymptotic expansion on the diagonal, that is,
\begin{eqnarray}\label{eq:FunctionalCalculusOnDiagonal}
  \eta_k(-i\mathcal{T})(x,x)\sim a_0(x)k^{n+1}+a_1(x)k^{n}+\ldots
\end{eqnarray}
where \(a_j\), \(j\geq 0\), are smooth functions on \(X\) and \(a_0\)
is a positive constant (see also~\cite{CHH24} for more details on the
\(a_j\)).  Define \(Y=X\times \R\),
\(T^{1,0}Y=T^{1,0}X\oplus\C \{\frac{\partial}{\partial
  s}-i\mathcal{T}\}\) and extend \(\mathcal{T}\) trivially to
\(Y\). Then \((Y,T^{1,0}Y)\) is a complex manifold and
\(X\ni x\mapsto (x,0)\in Y\) defines a CR embedding of \(X\) into
\(Y\). Furthermore, we have that for any \(j\in \N\) the function
\(\tilde{f}_j\colon Y\to \C\),
\(\tilde{f}_j(x,s)=e^{\lambda_js}f_j(x)\) is a holomorphic extension
of \(f_j\).  It is a fact that for any \(j_0\in \N_0\) there exists
\(N\in \N\) such that
\(x\mapsto(f_{j_0+1}(x),\ldots, f_{j_0+N}(x))\) defines a CR
embedding of \(X\) into \(\C^{N}\). For simplicity we take
\(j_0=0\). It follows that the holomorphic extension
\[G\colon Y\to \C^{N},\,\,\,G(x,s)= (e^{\lambda_{1}s}f_{1}(x),\ldots,
  e^{\lambda_{N}s}f_{N}(x)),\] defines a holomorphic embedding of an
open neighborhood \(X\times(-c,c)\subset Y\) around \(X\times\{0\}\)
for some \(c>0\). Now we take a function
\(\chi\in \mathscr{C}^\infty_c(\R_+)\) with \(\chi\not\equiv 0\),
\(\supp\chi\subset (0,1)\). Put \(k_0=\lambda_1\). For each large
enough \(k>k_0\) we consider the holomorphic map
\[H_k\colon Y\to
  \C^{N_k},\,\,\,\,H_k(x,s)=(\chi(k^{-1}\lambda_1)e^{\lambda_{1}s}f_{1}(x),\ldots,\chi(k^{-1}\lambda_{N_k})
  e^{\lambda_{N_k}s}f_{N_k}(x))\] where
\(N_k=\#\{j\in \N\mid\lambda_j\leq k\}\). Then consider the map
\[F_k\colon Y\to
  \C^{N+N_k},\,\,\,\,F_k(x,s)=\left(e^{-k}G(x,s),\frac{H_k(x,s)}{\sqrt{k^{n+1}C}}\right)\]
where \(C>0\) is a constant independent of \(k\). Now from the
properties of \(G\) it follows that for each \(k\) the map \(F_k\)
defines a holomorphic embedding of \(X\times (-c,c)\) into
\(\C^{N+N_k}\) while at the same time the factor \(e^{-k}\) in front
of \(G\) ensures that the contribution of \(G\) to \(F_k\) is an
\(O(k^{-\infty})\). More precisely, we find
\[|F_k(x,s)|^2=e^{-2k}|G(x,s)|^2+
  k^{-n-1}C(\chi,n)|H_k(x,s)|^2=k^{-n-1}C(\chi,n)|H_k(x,s)|^2+O(k^{-\infty})\]
and
\[|H_k(x,0)|^2=\sum_{j=1}^\infty
  |\chi(k^{-1}\lambda_j)|^2|f_j(x)|^2=\eta_k(-i\mathcal{T})(x,x)\]
with \(\eta=|\chi|^2\). It follows that choosing \(C=a_0\) where
\(a_0>0\) is the constant in~\eqref{eq:FunctionalCalculusOnDiagonal}
we obtain \(|F_k(x,0)|^2=1+O(k^{-1})\) and hence that \(F_k\) maps
\(X\times\{0\}\) close to the unit sphere in \(\C^{N+N_k}\) when \(k\)
becomes large. Fix \(k>k_0\). Since we are interested in embeddings
into the unit sphere we will now identify the points \((x,s)\in Y\)
such that \(F_k(x,s)\) lies in the unit sphere, that is, we have to
solve the equation \(|F_k(x,s)|^2=1\). We write
\begin{eqnarray}\label{eq:NormSquareMapIdeaProofSasakian}
  |F_k(x,s)|^2=e^{-2k}\sum_{j=1}^Ne^{2\lambda_js}|f_j(x)|^2+\frac{1}{k^{n+1}a_0}\sum_{j=1}^{N_k}|\chi(k^{-1}\lambda_j)|^2e^{2\lambda_js}|f_j(x)|^2
\end{eqnarray}
and observe from \(\lambda_j>0\) that for each \(x\in X\) we have that
\(s\mapsto F_k(x,s)\) is a strictly increasing function on \(\R\) with
\(\lim_{s\to \infty}F_k(x,s)=\infty\) and
\(\lim_{s\to -\infty}F_k(x,s)=0\). As a conclusion we find that for
each \(x\in X\) and \(k>k_0\) there is a unique \(s=s(x,k)\) which
solves the equation \(|F_k(x,s)|^2=1\). Furthermore, one can show
using~\eqref{eq:NormSquareMapIdeaProofSasakian} and
\(|F_k(x,0)|^2=1+O(k^{-1})\) that \(|s(x,k)|\leq C_1k^{-2}\) for some
constant \(C_1>0\) and all sufficiently large \(k\) (see
Proposition~\ref{pro:PropertiesOfFk}).  Then, for each \(k>k_0\), we define a function \(\varphi_k\colon X\to \R\),
\(\varphi_k(x)=s(x,k)\). From the implicit function theorem it follows
that \(\varphi_k\) is smooth and invariant under the flow of
\(\mathcal{T}\) when \(k\) is large enough. Furthermore, using the
results on the semi-classical functional calculus for Toeplitz
operators in the Sasakian case~\cite{HHL22,HHMS23} and some
modifications (see Lemma~\ref{lem:TaylorExpansionFunctionalCalculus})
we can control all \(\mathscr{C}^m\)-norms of \(\varphi_k\) when \(k\)
becomes large (see Theorem~\ref{thm:ConstructionOfVarphik}). Put
\(X_{k}=\{(x,s)\in Y\colon s=\varphi_k(x)\}\). Then \(X_k\) is a
smooth real hypersurface in \(Y\) and hence \((X_k,T^{1,0}X_k)\) with
\(T^{1,0}X_k=\C TX_k\cap T^{1,0}Y\) is a CR submanifold of \(Y\).
Furthermore, since \(X_k\) is contained in \(X\times (-c,c)\) for
\(k\) large enough we find that \(F_k\) restricts to a smooth CR
embedding \(\hat{F}_k\colon X_k\to\C^{N+N_k}\) and \(\hat{F}_k(X_k)\)
is contained in the unit sphere by construction. It follows that
\((X_k,T^{1,0}X_k)\) is strictly pseudoconvex (see
Lemma~\ref{lem:CRsubmanifoldsOfSpheresAreStrictlyPseudoconvex}). Since
\(\mathcal{T}\varphi_k\equiv 0\) we conclude that
\((X_k,T^{1,0}X_k,\mathcal{T})\) is Sasakian and
\((\hat{F}_k)_*\mathcal{T}=\mathcal{T}_{\beta(k)}\circ\hat{F}_k\) with
\(\mathcal{T}_{\beta(k)}\) as in Example~\ref{ex:PseudoHermSpheres}
for some weight vector \(\beta(k)\in \R^{N+N_k}_+\). Finally, the
projection \((x,s)\mapsto x\) defines an isomorphism between the
Sasakian manifolds \((X_k,T^{1,0}X_k,\mathcal{T})\) and
\((X,\mathcal{V}(\varphi_k),\mathcal{T})\) (see
Lemma~\ref{lem:IsoGraphDeformation}) with \(\mathcal{V}(\varphi_k)\)
given by~\eqref{eq:DefDeformedSasakianStructure}. The conclusion of
Theorem~\ref{thm:MainThmIntro} follows.
 
 \subsection{The Pseudohermitian Case}
 Now let \((X,T^{1,0}X,\alpha)\), \(\dim_\R X=2n+1\), \(n\geq 1\), be
 a compact pseudohermitian manifold which is embeddable into the
 complex Euclidean space. More precisely, \((X,T^{1,0}X)\) is a
 compact, CR embeddable, strictly pseudoconvex CR manifold and
 \(\alpha\) is a contact form for \(\operatorname{Re}T^{1,0}X\) such
 that its associated Levi form is positive definite. Denote by
 \(\mathcal{T}\) the Reeb vector field associated to
 \(\alpha\). Without loss of generality (see
 Lemma~\ref{lem:CharacterizationEmbeddable}) we can assume that \(X\)
 bounds a relatively compact, smoothly bounded, strictly pseudoconvex
 domain \(M\) in a complex manifold \(M'\) with \(\dim_\C M'=n+1\). We
 choose a Hermitian metric on \(M'\) and a Toeplitz operator \(T_R\)
 on \(\overline{M}\) which reflects the pseudohermitian structure on
 \(X\) (see \eqref{eq:propHermMetric1}, \eqref{eq:propHermMetric2} and
 \eqref{eq:DefToeplitzOperatorGeneral}).  We denote by
 \(0<\lambda_1\leq\lambda_2\leq\ldots\) the positive eigenvalues of
 the operator \(T_R\) acting on smooth functions on \(\overline{M}\)
 counting multiplicity and let \(\{f_j\}_{j\in \N}\) be an orthonormal
 system of corresponding eigenfunctions.  Given a function
 \(\eta\in \mathscr{C}^\infty(\R_+)\) put \(\eta_k(t)=\eta(k^{-1}t)\),
 \(t,k>0\) and consider the integral kernel of the semi-classical
 functional calculus for the Toeplitz operator \(T_R\) by
 \[\eta_k(T_R)\colon \overline{M}\times \overline{M}\to
   \C,\,\,\,\eta_k(T_R)(z,w)=\sum_{j=1}^\infty
   \eta_k(\lambda_j)f_j(z)\overline{f_j(w)},\,\,k>0.\] It follows
 from~\cite{HM23} that this operator can be expressed as a
 semi-classical Fourier integral operator up to an
 \(O(k^{-\infty})\). Furthermore, a key point for the proof of the
 result in~\cite{HM23} is the observation that
 \(\eta_k(T_R)|_{X\times X}\), (\(X=bM\)), has a similar asymptotic
 description as the semi-classical functional calculus for Toeplitz
 operators on \(X\) as in~\cite{HHMS23}. As a consequence we find that
 on \(X=bM\) we have
 \begin{eqnarray}\label{eq:ExpansionDomainCaseIdea}
   \eta_k(T_R)(x,x)\sim a_0(x)k^{n+2}+a_1(x)k^{n+1}+\ldots,\,\,\,\,x\in X
 \end{eqnarray}
 and \(a_0\) is a positive constant.  Let
 \(\tilde{\mathcal{G}}\colon X\to \C^N\) be a CR embedding of \(X\)
 and denote its holomorphic extension by
 \(\mathcal{G}\colon \overline{M}\to \C^N\). We find that
 \(\mathcal{G}\) is an embedding of a small open neighborhood of \(X\)
 in \(\overline{M}\). Put \(k_0=\lambda_1\). Given
 \(\chi\in \mathscr{C}^\infty_c(\R_+)\), \(\supp \chi \subset (0,1)\),
 \(\chi\not\equiv 0\) we define
 \[\mathcal{H}_k\colon \overline{M}\to \C^{N_k},
   \mathcal{H}_k(z)=(\chi(k^{-1}\lambda_1)f_{1}(z),\ldots,\chi(k^{-1}\lambda_{N_k})
   f_{N_k}(z))\] for \(k>k_0\) with
 \(N_k=\#\{j\in \N\mid\lambda_j\leq k\}\) and consider the map
 \begin{eqnarray}
   \mathcal{F}_k\colon \overline{M}\to \C^{N+N_k},\,\,\, \mathcal{F}_k=\left(e^{-k}\mathcal{G},\frac{\mathcal{H}_k}{\sqrt{k^{n+2}C}}\right)
 \end{eqnarray}
 with \(C=a_0>0\) where \(a_0\) is the constant
 in~\ref{eq:ExpansionDomainCaseIdea} for \(\eta=|\chi|^2\). We note
 that for some technical reasons we actually do not start with the
 eigenfunction \(f_1\) in the definition of \(\mathcal{H}_k\)
 (see~\eqref{eq:DefHkGeneral}). But in order to explain the idea we
 should keep the notation simple. Now let
 \(E\colon X\times [0,c)\to V\) be a diffeomorphism onto a relatively
 open neighborhood \(V\subset\overline{M}\) around \(X=bM\) such that
 \(E|_{X\times\{0\}}\) is the inclusion \(X\to \overline{M}\). Then
 \(E\) induces a complex structure on \(Y=X\times [0,c)\) such that
 \(Y\) is a complex manifold with boundary \(bY=X\). Then consider
 \(\tilde{F}_k\colon Y\to \C^{N+N_k}\),
 \(\tilde{F}_k(x,s)=\mathcal{F}_k(E(x,s))\). From the properties of
 \(\mathcal{G}\) we find that by shrinking \(c>0\) we have that
 \(\tilde{F}_k\) is a holomorphic embedding of \(Y\) into
 \(\C^{N+N_k}\) for all \(k>k_0\). Similar as in the Sasakian case we
 would like to identify for each \(k\) the points \((x,s)\in Y\) which
 are mapped into the unit sphere under \(\tilde{F}_k\), that is, we
 would like to find solutions of the equation
 \(|\tilde{F}_k(x,s)|^2=1\). This leads to a problem since
 \(\tilde{F}_k\) is only defined for \(s\geq 0\) (In the Sasakian case
 it was defined for all \(s\in \R\)). In order to ensure that the
 desired solutions \((x,s)\) satisfy \(s\geq 0\) we scale
 \(\tilde{F}_k\) by a \(k\)-dependent factor, that is, we define
 \[F_k\colon Y\to
   \C^{N+N_k},\,\,\,F_k=\frac{1}{\sqrt{1-d/k}}\tilde{F}_k\] where
 \(d>0\) is a large enough positive constant and \(k_0\) has to be
 increased depending on \(d\). Hence, choosing \(d\) and \(k_0\) well,
 we find using~\eqref{eq:ExpansionDomainCaseIdea} with
 \(\eta=|\chi|^2\) that \(|F_k(x,0)|^2>1\) for all \(x\in X\) and
 \(k>k_0\). Recall that in the Sasakian case we had a quite explicit
 description for the dependency of \(|F_k(x,s)|^2\) on \(s\). In the
 general case we use the result in~\cite{HM23} on the asymptotic
 expansion of \(\eta_k(T_R)\) with \(\eta=|\chi|^2\) and find that
 \begin{eqnarray}\label{eq:FunctionalCalculusDomainIdea}
   |F_k(x,s)|^2=\int_\R e^{-tk\tau(x,s)}(a_0(x,s,t)+a_1(x,s,t,k))dt
 \end{eqnarray}
 for some functions \(\tau,a_0,a_1\) satisfying some properties (see
 Lemma~\ref{lem:ConsequenceOfHM}). After rescaling the diffeomorphism
 \(E\) (see Lemma~\ref{lem:ConsequenceOfHMGoodE}) we can ensure that
 \(\tau\) is good enough in order to show that for any \(x\in X\) and
 \(k\geq k_0\) there is a unique \(s=s(x,k)\geq 0\) with
 \(|F_k(x,s)|^2=1\) (see
 Proposition~\ref{pro:PropertiesFkGeneral}). Furthermore, we find
 \(0< s(x,k)\leq C_1k^2\) for some constant \(C_1\). Then, as in the
 Sasakian case, define \(\varphi_k\colon X\to(0,c)\),
 \(\varphi_k(x)=s(x,k)\) and put
 \(X_k=\{(x,s)\in Y\colon s=\varphi_k(x)\}\),
 \(T^{1,0}X_k=\C TX_k\cap T^{1,0}Y\). Furthermore, denote by
 \(\iota_k\colon X_k\to Y\) the
 inclusion. From~\eqref{eq:FunctionalCalculusDomainIdea} we can verify
 crucial properties of the asymptotic behavior of \(\varphi_k\) in
 \(k\) (see Lemma~\ref{lem:IntegralLemma} and
 Theorem~\ref{thm:ConstructionPhiGeneral}). We should mention that
 this part of the proof is more technical than for the Sasakian case
 since we have to use~\eqref{eq:FunctionalCalculusDomainIdea} instead
 of the explicit
 description~\eqref{eq:NormSquareMapIdeaProofSasakian}. Since \(X_k\)
 is a real hypersurface in \(Y\) we have that \((X_k,T^{1,0}X_k)\) is
 a CR manifold and \(F_k\) restricts to a CR embedding of \(X_k\) into
 \(\C^{N+N_k}\) such that \(F_k(X_k)\subset \mathbb{S}^{2(N+N_k)-1}\)
 by construction. It follows (see
 Lemma~\ref{lem:CRsubmanifoldsOfSpheresAreStrictlyPseudoconvex}) that
 \((X_k,T^{1,0}X_k,\gamma_k)\) with
 \(\gamma_k=\iota_k^*F_k^*\alpha_{\beta(k)}\) is a pseudohermitian
 manifold where \(\beta(k)\in \R^{N+N_k}_+\) is chosen with respect to
 the map \(F_k\) (see~\eqref{eq:DefBetakGeneral}). Since \(F_k(X_k)\)
 is the transversal intersection of a complex submanifold in
 \(\C^{N+N_k}\) with \(\mathbb{S}^{2(N+N_k)-1}\) we find that 
 \((X_k,T^{1,0}X_k,\gamma_k)\) is isomorphic to a real analytic
 pseudohermitian submanifold of the sphere. Now define
 \(\mathcal{V}^k=\operatorname{pr}_*T^{1,0}X_k\),
 \(\alpha^k=\kappa_k^*\gamma_k\) and \(G_k=F_k\circ \kappa_k\). Here, 
 \(\operatorname{pr}\colon Y\to X\), \((x,s)\mapsto x\) denotes the
 projection and \(\kappa_k\colon X\to X_k\) is defined by
 \(\kappa_k(x)=(x,\varphi_k(x))\). It follows that
 \((X,\mathcal{V}^k,\alpha^k)\) is a pseudohermitian manifold which is
 isomorphic to \((X_k,T^{1,0}X_k,\gamma_k)\) and that \(G_k\) is a CR
 embedding of \(X\) with respect to the CR structure \(\mathcal{V}^k\)
 such that \(G_k(X)\subset \mathbb{S}^{2(N+N_k)-1}\) and
 \(G_k^*\alpha_{\beta(k)}=\alpha^k\). Let \(\mathcal{T}^k\) denote the
 Reeb vector field associated to \(\alpha^k\). Now, following the
 ideas and methods developed in~\cite{HHMS23,HHMS24}, we can show that
 \(\alpha_k\to \alpha\), \(\mathcal{T}^k\to \mathcal{T}\) in
 \(\mathscr{C}^\infty\)-topology as \(k\to \infty\) and describe the
 pushforward of \(\mathcal{T}^k\) under the map \(G_k\). We note that
 in the Sasakian case we get this part for free since the embedding maps are
 equivariant by construction.  Also, in the Sasakian case, the
 convergence of \(\mathcal{V}^k\) to \(T^{1,0}X\) is quite obvious
 since the complex structure on \(Y\) is explicitly given and hence
 \(\mathcal{V}^k\) is explicitly described in terms of \(\varphi_k\)
 and \(T^{1,0}X\) (see the definition of \(\mathcal{V}(\varphi)\)
 in~\eqref{eq:DefDeformedSasakianStructure}). In the pseudohermitian
 case we only have that there is a complex structure \(T^{1,0}Y\) on
 \(Y\) which induces the CR structure \(T^{1,0}X\) on
 \(X\simeq X\times\{0\}\). It turns out by the properties of
 \(\varphi_k\) that this is still enough to prove that
 \(\mathcal{V}^k\to T^{1,0}X\) in \(\mathscr{C}^\infty\)-topology as
 \(k\to \infty\) (see Definition~\ref{def:BundleConvergence} and
 Lemma~\ref{lem:CRStructureConvergenceGeneral}). The conclusion of
 Theorem~\ref{thm:MainThmIntroGeneral} follows.

 \section{Preliminaries}\label{sec:Preliminaries}
 \subsection{Notations}
 We use the following notations throughout this article. $\mathbb{Z}$
 is the set of integers, $\mathbb N=\{1,2,3,\ldots\}$ is the set of
 natural numbers, $\mathbb N_0=\mathbb N\bigcup\{0\}$,  $\mathbb R$
 is the set of real numbers and \(\C\) is the set of complex numbers. Also,
 $\overline{\mathbb R}_+=\{x\in\mathbb R:x\geq0\}$,
 $\R^*:=\R\setminus\{0\}$,
 $\R_+:=\overline{\mathbb R}_+\setminus\{0\}$.  We write
 $\alpha=(\alpha_1,\ldots,\alpha_n)\in\mathbb N^n_0$ if
 $\alpha_j\in\mathbb N_0$, $j=1,\ldots,n$ and put
 \(|\alpha|=\sum_{j=1}^na_j\). For $x=(x_1,\ldots,x_n)\in\mathbb R^n$,
 we write
 \[
   \begin{split}
     &x^\alpha=x_1^{\alpha_1}\ldots x^{\alpha_n}_n,\\
     & \partial_{x_j}=\frac{\partial}{\partial x_j}\,,\quad
       d_x^\alpha=\partial^\alpha_x=\partial^{\alpha_1}_{x_1}\ldots\partial^{\alpha_n}_{x_n}
       =\frac{\partial^{|\alpha|}}{\partial x^\alpha}\,.       
   \end{split}
 \]
 For $j, s\in\mathbb Z$, set $\delta_{js}=1$ if $j=s$, $\delta_{js}=0$
 if $j\neq s$.

 Let $U$ be an open set in $\mathbb{R}^{n_1}$, $V$ be an open set in
 $\mathbb{R}^{n_2}$ and \(s\in \N_0\cup\{\infty\}\). Let
 $\mathscr{C}^s(U)$ (\(\mathscr{C}^s(U,\R)\)) and $\mathscr{C}^s_c(V)$
 ($\mathscr{C}^s_c(V,\R)$) be the space of \(\mathscr{C}^s\)-smooth
 complex valued (real valued) functions on $U$ and the space of smooth
 complex valued (real valued) functions with compact support in $V$ ,
 respectively; $\mathscr{D}'(U)$ and $\mathscr{E}'(V)$ be the space of
 distributions on $U$ and the space of distributions with compact
 support in $V$, respectively. By saying smooth we always mean
 \(\mathscr{C}^\infty\)-smooth.

 Let $F:\cC^\infty_c(V)\to\mathscr{D}'(U)$ be a continuous operator
 and let $F(x,y)\in\mathscr{D}'(U\times V)$ be the distribution kernel
 of $F$. In this work, we will identify $F$ with $F(x,y)$.  We say
 that $F$ is a smoothing operator if $F(x,y)\in\cC^\infty(U\times V)$.
 For two continuous linear operators
 $A,B:\cC^\infty_c(V)\to\mathscr{D}'(U)$, we write $A\equiv B$ (on
 $U\times V$) or $A(x,y)\equiv B(x,y)$ (on $U\times V$) if $A-B$ is a
 smoothing operator, where $A(x,y),B(x,y)\in\mathscr{D}'(U\times V)$
 are the distribution kernels of $A$ and $B$, respectively.

 If $U=V$, for a smooth function
 $f(x,y)\in\mathscr{C}^\infty(U\times U)$, we write
 $f(x,y)=O\left(|x-y|^\infty\right)$ if for all multi-indices
 $\alpha,\beta\in\N_0^{n_1}$,
 $\left(\partial_x^\alpha\partial_y^\beta f\right)(x,x)=0$, for all
 $x\in U$.

 Let us introduce some notion in microlocal analysis used in this
 paper. Let $D\subset\R^{2n+1}$ be an open set. For any
 $m\in\mathbb R$, $S^m_{1,0}(D\times D\times\mathbb{R}_+)$ is the
 space of all $s(x,y,t)\in\cC^\infty(D\times D\times\mathbb{R}_+)$
 such that for all compact sets $K\Subset D\times D$, all
 $\alpha, \beta\in\mathbb N^{2n+1}_0$ and $\gamma\in\mathbb N_0$,
 there is a constant $C_{K,\alpha,\beta,\gamma}>0$ satisfying the
 estimate
 \begin{equation}
   \left|\partial^\alpha_x\partial^\beta_y\partial^\gamma_t a(x,y,t)\right|\leq 
   C_{K,\alpha,\beta,\gamma}(1+|t|)^{m-|\gamma|},\ \ 
   \mbox{for all $(x,y,t)\in K\times\mathbb R_+$, $|t|\geq1$}.
 \end{equation}
 We put
 \[
   S^{-\infty}(D\times D\times\mathbb{R}_+) :=\bigcap_{m\in\mathbb
     R}S^m_{1,0}(D\times D\times\mathbb{R}_+).\] Let
 $s_j\in S^{m_j}_{1,0}(D\times D\times\mathbb{R}_+)$, $j=0,1,2,\cdots$
 with $m_j\rightarrow-\infty$ as $j\rightarrow+\infty$.  By the
 argument of Borel construction, there always exists
 $s\in S^{m_0}_{1,0}(D\times D\times\mathbb{R}_+)$ unique modulo
 $S^{-\infty}$ such that
 \begin{equation}
   s-\sum^{\ell-1}_{j=0}s_j\in S^{m_\ell}_{1,0}(D\times D\times\mathbb{R}_+)
 \end{equation}
 for all $\ell=1,2,\cdots$. If $s$ and $s_j$ have the properties
 above, we write
 \begin{equation}
   s(x,y,t)\sim\sum^{+\infty}_{j=0}s_j(x,y,t)~\text{in}~ 
   S^{m_0}_{1,0}(D\times D\times\mathbb{R}_+).
 \end{equation}
 Also, we use the notation
 \begin{equation}
   s(x, y, t)\in S^{m}_{{\rm cl\,}}(D\times D\times\mathbb{R}_+)
 \end{equation}
 if $s(x, y, t)\in S^{m}_{1,0}(D\times D\times\mathbb{R}_+)$ and we
 can find $s_j(x, y)\in\cC^\infty(D\times D)$, $j\in\N_0$, such that
 \begin{equation}
   s(x, y, t)\sim\sum^{+\infty}_{j=0}s_j(x, y)t^{m-j}\text{ in }S^{m}_{1, 0}
   (D\times D\times\mathbb{R}_+).
 \end{equation}

 Let $M, M_1$ be smooth paracompact manifolds. We can define
 $S^m_{1,0}(M\times M_1\times\mathbb R_+)$,
 $S^m_{{\rm cl\,}}(M\times M_1\times\mathbb R_+)$ and asymptotic sums
 in the standard way. Given a smooth vector bundle \(E\to M\) over
 \(E\) we denote the space of smooth sections on \(M\) with values in
 \(E\) by \(\mathscr{C}^\infty(M,E)\). The space of those sections
 with compact support is denoted by \(\mathscr{C}_c^\infty(M,E)\). For
 any \(q\in\N_0\) we denote the space of smooth \(q\)-forms on \(M\)
 by \(\Omega^q(M)\) and denote space of those forms with compact
 support by \(\Omega_c^q(M)\). Recall the definition of the exterior
 differential \(d\colon \Omega^q(M)\to \Omega^{q+1}(M)\).

 We recall some notations of semi-classical analysis.  Let $U$ be an
 open set in $\mathbb{R}^{n_1}$ and let $V$ be an open set in
 $\mathbb{R}^{n_2}$.

\begin{definition}\label{D:knegl}
  We say a $k$-dependent continuous linear operator
  $F_k:\mathscr{C}^\infty_c(V)\to\mathscr{D}'(U)$ is $k$-negligible if
  for all $k$ large enough, $F_k$ is a smoothing operator, and for any
  compact set $K$ in $U\times V$, for all multi-index
  $\alpha\in\N_0^{n_1}$, $\beta\in\N_0^{n_2}$ and $N\in\mathbb{N}_0$,
  there exists a constant $C_{K,\alpha,\beta,N}>0$ such that
  \begin{equation*}
    \left|\partial^\alpha_x\partial^\beta_y F_k(x,y)\right|\leq 
    C_{K,\alpha,\beta,N} k^{-N},
  \end{equation*}
  for all $x,y\in K$ and $k$ large enough.  We write
  $F_k(x,y)=O(k^{-\infty})$ (on $U\times V$) or $F_k=O(k^{-\infty})$
  (on $U\times V$) if $F_k$ is $k$-negligible.
\end{definition}

Next, we recall semi-classical symbol spaces. Let $W$ be an open set
in $\mathbb{R}^{N}$, we define the space
\begin{equation}
  S(1;W):=\{a\in\mathscr{C}^\infty(W):\sup_{x\in W}|\partial^\alpha_x a(x)|<\infty~\text{for all}~\alpha\in\mathbb{N}_0^N\}.
\end{equation}
Consider the space $S^0_{\operatorname{loc}}(1;W)$ containing all
smooth functions $a(x,k)$ with real parameter $k$ such that for all
multi-index $\alpha\in\mathbb{N}_0^N$, any cut-off function
$\chi\in\mathscr{C}^\infty_c(W)$, we have
\begin{equation}
  \sup_{\substack{k\in\mathbb{R}\\k\geq 1}}\sup_{x\in W}|\partial^\alpha_x(\chi(x)a(x,k))|<\infty.
\end{equation}
For general $m\in\mathbb{R}$, we can also consider
\begin{equation}
  S^m_{\operatorname{loc}}(1;W):=\{a(x,k):k^{-m}a(x,k)\in\ S^0_{\operatorname{loc}}(1;W)\}.  
\end{equation}
In other words, $S^m_{\operatorname{loc}}(1;W)$ takes all the smooth
function $a(x,k)$ with parameter $k\in\mathbb{R}$ satisfying the
following estimate. For any compact set $K\Subset W$, any multi-index
$\alpha\in\mathbb N^n_0$, there is a constant $C_{K,\alpha}>0$
independent of $k$ such that
\begin{equation}
  |\partial^\alpha_x (a(x,k))|\leq C_{K,\alpha} k^m,~\text{for all}~x\in K,~k\geq1. 
\end{equation}
For a sequence of $a_j\in S^{m_j}_{\operatorname{loc}}(1;W)$ with
$m_j$ decreasing, $m_j\to-\infty$, and
$a\in S^{m_0}_{\operatorname{loc}}(1;W)$. We say
\begin{equation}
  a(x,k)\sim\sum_{j=0}^\infty a_j(x,k)~\text{in}~S^{m_0}_{\operatorname{loc}}(1;W)
\end{equation}
if for all $l\in\mathbb{N}$, we have
\begin{equation}
  a-\sum_{j=0}^{l-1} a_j\in S^{m_l}_{\operatorname{loc}}(1;W).
\end{equation}
In fact, for all sequences $a_j$ above, there always exists an element
$a$ as the asymptotic sum, which is unique up to the elements in
$S^{-\infty}_{\operatorname{loc}}(1;W):=\cap_{m}
S^m_{\operatorname{loc}}(1;W)$.

\begin{definition}
  Let $W$ be an open set in $\mathbb{R}^{N}$, \(M\subset W\) a subset,
  \(s\in \N_0\cup\{\infty\}\), \(m\in\Z\) and \(k_0\geq 1\). Given a
  family \(\{f_k\}_{k\geq k_0}\subset \mathscr{C}^s(W)\) of
  \(\mathscr{C}^s\)-smooth functions on \(W\) we say \(f_k=O(k^{m})\)
  in \(\mathscr{C}^s\)-topology on \(M\) (or in
  \(\mathscr{C}^s(M)\)-topology) if for any compact set \(K\subset M\)
  and any \(\alpha\in \N_0^{N}\), \(|\alpha|\leq s \), there exists a
  constant \(C_{\alpha,K}\) such that
  \begin{equation}
    |\partial^\alpha_x f_k(x)|\leq C_{K,\alpha} k^m,~\text{for all}~x\in K,~k\geq k_0. 
  \end{equation}
  If \(M=W\) we just say \(f_k=O(k^{m})\) in
  \(\mathscr{C}^s\)-topology.  We say \(f_k=O(k^{m})\) on \(M\) if
  \(f_k=O(k^{m})\) in \(\mathscr{C}^0\)-topology on
  \(M\). Furthermore, we say \(f_k=O(k^{-\infty})\) in
  \(\mathscr{C}^s\)-topology on \(M\) if \(f_k=O(k^{m})\) in
  \(\mathscr{C}^s\)-topology on \(M\) for all \(m\in \Z\).
\end{definition}
\begin{remark}
  Given \(m\in \Z\) and a family of smooth functions
  \(\{f_k\}_{k\geq 1}\subset \mathscr{C}^\infty(W)\) we have
  \(f_k=O(k^{m})\) in \(\mathscr{C}^\infty\)-topology if and only if
  \(((x,k)\mapsto f_k(x))\in S^m_{\operatorname{loc}}(1,W)\).
\end{remark}
The detail of discussion here about microlocal analysis or
semi-classical analysis can be found in \cites{DS99, GS94} for
example. All the notations introduced above can be generalized to the
case on paracompact manifolds and paracompact manifolds with boundary.

Given a complex manifold \(Y\) we denote its complex structure by
\(T^{1,0}Y\) and the space of holomorphic functions on \(Y\) by
\(\mathcal{O}(Y)\). Furthermore, if \(Y\) is a complex manifold with
boundary we denote by \(\mathcal{O}^\infty(Y)\) the smooth functions
on \(Y\) which are holomorphic on the set of interior points of
\(Y\). For any \(p,q\in\N_0\) we denote the space of smooth
\((p,q)\)-forms by
\(\Omega^{p,q}(Y):=\mathscr{C}^\infty(Y,\Lambda^{p,q}\C T^*Y)\). We
further denote the space of those forms with compact support by
\(\Omega^{p,q}_c(Y)\).  Here, \(\Lambda^{p,q}\C T^*Y\) is the bundle
of alternating multilinear forms on \(\C TY\) of type
\((p,q)\). Recall that on \(Y\) we have the decomposition for the
exterior differential \(d=\partial+\overline{\partial}\) with
\[\partial\colon \Omega^{p,q}(Y)\to
  \Omega^{p+1,q}(Y),\,\,\,\overline{\partial}\colon \Omega^{p,q}(Y)\to
  \Omega^{p,q+1}(Y).\]

\subsection{Pseudohermitian Manifolds}\label{sec:PseudohermitianCRManifolds}
Let $X$ be a ($\mathscr{C}^\infty$-)smooth connected orientable
manifold of real dimension $2n+1,~n\geq 1$. A subbundle \(T^{1,0}X\)
of the complexified tangent bundle \(\C TX\) is called (codimension
one) Cauchy--Riemann (CR for short) structure if
\begin{enumerate}
\item (codimension one) $\dim_{\mathbb{C}}T^{1,0}_{p}X=n$ for any
  $p\in X$.
\item (non-degenerate) $T^{1,0}_p X\cap T^{0,1}_p X=\{0\}$ for any
  $p\in X$, where $T^{0,1}_p X:=\overline{T^{1,0}_p X}$.
\item (integrable) For $V_1, V_2\in \mathscr{C}^{\infty}(X,T^{1,0}X)$,
  then $[V_1,V_2]\in\mathscr{C}^{\infty}(X,T^{1,0}X)$, where
  $[\cdot,\cdot]$ stands for the Lie bracket between vector fields.
\end{enumerate}
Given a CR structure \(T^{1,0}X\) on \(X\) we call the pair
$(X,T^{1,0}X)$ a (codimension one) Cauchy--Riemann (CR for short)
manifold with CR structure \(T^{1,0}X\).  Given a distribution
\(f\in \mathcal{D}'(X)\) we say that \(f\) is CR if
\(\overline{Z}f= 0\) for all smooth sections
\(Z\in \vbunsec{X,T^{1,0}X}\) in the sense of distributions. A smooth
CR function on \((X,T^{1,0}X)\) is a function
\(f\in \mathscr{C}^\infty(X)\) such that \(f\) defines a CR
distribution.
\begin{example}
  Let \((Y,T^{1,0}Y)\) be a complex manifold with \(\dim_\C Y=N\) and
  let \(X\subset Y\) be a smooth real submanifold of \(Y\) which is
  connected and orientable with \(\dim_\R X=2n+1\), \(n\geq 1\). Put
  \(T^{1,0}X:=\C TX\cap T^{1,0}Y\). If \(T^{1,0}X\) is a complex rank
  \(n\) subbundle of \(\C TX\) we have that \(T^{1,0}X\) is a CR
  structure on \(X\) and \(X\) together with this particular CR
  structure \(T^{1,0}X\) is called a CR submanifold of
  \(Y\). Furthermore, given a holomorphic function \(f\) in an open
  neighborhood around \(X\) in \(Y\) we have that \(f|_X\) is a smooth
  CR function on \(X\).
\end{example}
\begin{definition}
  Let \((Y,T^{1,0}Y)\) be a complex manifold with \(\dim_\C Y=N\) and
  let \((X,T^{1,0}X)\) be a CR manifold with \(\dim_\R X=2n+1\),
  \(n\geq 1\). A map \(F\colon X\to Y\) is called a CR embedding if
  \(F\) is a smooth embedding and
  \(F_*T^{1,0}X=\C T F(X)\cap T^{1,0}Y\).
\end{definition}
In particular, we have that given a CR embedding \(F\colon X\to Y\) it follows that \(F(X)\) is a CR submanifold of \(Y\) such that the induced
CR structure on \(F(X)\) coincides with the push forward of the
original CR structure \(T^{1,0}X\) on \(X\) under the map \(F\). The
following is quite useful.
\begin{lemma}
  Let \((X,T^{1,0}X)\) be a CR manifold with \(\dim_\R X=2n+1\),
  \(n\geq 1\) and \(N\in \N\). A smooth embedding
  \(F=(f_1,\ldots,f_N)\colon X\to \C^N\) is a CR embedding if and only
  if \(f_j\), \(1\leq j\leq N\) are smooth CR functions.
\end{lemma}
\begin{proof}
  See Chen--Shaw~\cite{CS01}*{Lemma 12.1.2}
\end{proof}

The Levi distribution $HX$ of the CR manifold $X$ is the real part of
$T^{1,0}X \oplus T^{0,1}X$, i.e. \(HX=\operatorname{Re}T^{1,0}X\),
i.e. the unique subbundle $HX$ of $TX$ such that
\begin{align}\label{eq:2.5}
  \C HX=T^{1,0}X \oplus T^{0,1}X.
\end{align} 
Let $J:HX\to HX$ be the (almost) complex structure given by
$J(u+\ol u)=iu-i\ol u$, for every $u\in T^{1,0}X$.  If we extend $J$
complex linearly to $\C HX$ we have
$T^{1,0}X \, = \, \left\{ V \in \C HX \,:\, \, JV \, = \, iV
\right\}$.  Thus the CR structure $T^{1,0}X$ is determined by the Levi
distribution and $J$. We note that \(J\) induces an orientation on
\(HX\). Since \(TX\) is orientable by assumption it follows that
\(TX/HX\) is orientable. Hence there exists a non-vanishing real
$1$-form \(\alpha\in \Omega^1(X)\) such that \(\ker \alpha=HX\).
\begin{definition}
  A (non-vanishing) real valued $1$-form \(\alpha\in \Omega^1(X)\)
  with \(\ker\alpha=HX=\operatorname{Re}T^{1,0}X\) is called a
  characteristic $1$-form for \((X,T^{1,0}X)\).
\end{definition}
\begin{definition}\label{Def:LeviFormCRGeneral}
  Let \(\alpha\in \Omega^1(X)\) be characteristic $1$-form for
  \((X,T^{1,0}X)\).  The Levi form
  $\mathcal{L}_x=\mathcal{L}^{\alpha}_x$ of $X$ at $x\in X$ associated
  to $\alpha$ is the symmetric bilinear map
  \begin{equation}\label{eq:2.12}
    \mathcal{L}_x:H_xX\times H_x X\to\R,\quad \mathcal{L}_x(u,v)
    =\frac{d\alpha}{2}(u,Jv),\quad \text{ for } u, v\in H_xX.
  \end{equation}  
  It induces a Hermitian form
  \begin{equation}\label{eq:2.12b}
    \mathcal{L}_x:T^{1,0}_xX\times T^{1,0}_xX\to\C,
    \quad \mathcal{L}_x(U,V)=\frac{d\alpha}{2i}(U, \ol V) ,
    \quad \text{ for } U, V\in T^{1,0}_xX.
  \end{equation}  
\end{definition}
 
 \begin{definition}
   A CR structure \(T^{1,0}X\) is said to be strictly pseudoconvex if
   there exists a characteristic $1$-form $\alpha$ for
   \((X,T^{1,0}X)\) such that for every $x\in X$ the Levi form
   $\mathcal{L}^{\alpha}_x$ is positive definite. We then call
   $(X,T^{1,0}X)$ a strictly pseudoconvex CR manifold.
 \end{definition}
 \begin{lemma}
   Let $(X,T^{1,0}X)$ be strictly pseudoconvex. Then
   \(HX=\operatorname{Re}T^{1,0}X\) is a contact structure. Any
   contact form for \(HX\) is a characteristic $1$-form for
   \((X,T^{1,0}X)\) and vice versa. Furthermore, if \(\alpha\) is a
   characteristic $1$-form for \((X,T^{1,0}X)\) we have either
   $\mathcal{L}^{\alpha}_x$ is positive definite at any point
   \(x\in X\) or $\mathcal{L}^{\alpha}_x$ is negative definite at any
   point \(x\in X\).
 \end{lemma}
 \begin{proof}
   From the definition we find a characteristic $1$-form
   $\tilde{\alpha}$ such that for every $x\in X$ the Levi form
   $\mathcal{L}^{\tilde{\alpha}}_x$ is positive definite. As a
   consequence we have
   $\tilde{\alpha}\wedge(d\tilde{\alpha})^n\neq 0$, where
   $\dim_\R X=2n+1$. It follows that \(HX\) is a contact structure and
   \(\tilde{\alpha}\) is a contact form. Given another characteristic
   $1$-form \(\alpha\) we find \(\alpha=f\tilde{\alpha}\) for some
   non-vanishing real valued smooth function
   \(f:X\to \R\setminus\{0\}\). The same holds true if \(\alpha\) is
   any contact form for \(HX\). Recalling that \(X\) is connected,
   this verifies the second part of the statement.
 \end{proof}

\begin{definition}
  Let $(X,T^{1,0}X)$ be a strictly pseudoconvex CR manifold and
  \(\alpha\in \Omega^1(X)\) a contact form for the contact structure
  \(\operatorname{Re}T^{1,0}X\) such that the associated Levi form is
  positive definite everywhere. We call \((T^{1,0}X,\alpha)\) a
  pseudohermitian structure on \(X\) and the triple
  \((X,T^{1,0}X,\alpha)\) is called a pseudohermitian manifold.
\end{definition}

\begin{definition}\label{D:Reebfield}
  Let \((X,T^{1,0}X,\alpha)\) be a pseudohermitian manifold. The Reeb
  vector field $\mathcal{T}=\mathcal{T}^{(\alpha)}$ associated with
  the contact form $\alpha$ is the vector field uniquely defined by
  the equations ${\mathcal{T}}\lrcorner\alpha=1$,
  ${\mathcal{T}}\lrcorner d\alpha=0$ (cf.\
  Geiges~\cite{Gei08}*{Lemma/Definition 1.1.9}). Here, \(\lrcorner\)
  denotes the contraction between vector fields and forms.
\end{definition}
We refer to Example~\ref{ex:PseudoHermSpheres} for examples of
pseudohermitian manifolds and Reeb vector fields.
\begin{remark}
  Let \((X,T^{1,0}X,\alpha)\) be a pseudohermitian manifold and let
  \(\mathcal{T}\) be the Reeb vector field associated to
  \(\alpha\). We find that \(TX=HX\oplus\R\mathcal{T}\). Denote by
  \(J\colon HX\to HX\) the (almost) complex structure associated to
  \(T^{1,0}X\) and by \(\operatorname{pr}\colon TX\to HX\) the
  projection with respect to the decomposition
  \(TX=HX\oplus\R\mathcal{T}\). Then the pseudohermitian structure
  \((T^{1,0}X,\alpha)\) induces a Riemannian metric
  \(g^{T^{1,0}X,\alpha}\) on \(X\) defined by
  \[g^{T^{1,0}X,\alpha}(u,v)=\frac{d\alpha}{2}\left(\operatorname{pr}u,J\operatorname{pr}v\right)+\alpha(u)\alpha(v),
    \,\,u,v\in T_pX,\,p\in X.\]
\end{remark}
\subsection{Strictly Pseudoconvex Domains in Complex
  Manifolds}\label{ssec:domains}
An important class of strictly pseudoconvex CR manifolds are
boundaries of strictly pseudoconvex domains.
\begin{definition}
  Let \((Y,T^{1,0}Y)\) be a complex manifold and let \(D\subset Y\) be
  a domain with boundary of class \(\mathscr{C}^2\). A boundary point
  \(p\in bD\) is called a strictly pseudoconvex boundary point if
  there is a neighborhood \(U\subset Y\) around \(p\) and a function
  \(\rho\colon U\to \R\) with \(D\cap U=\{z\in U\colon \rho(z)<0\}\),
  \(d\rho\neq 0\) at \(p\) and
  \(\partial\overline{\partial}\rho (Z,\overline{Z})>0\) for all
  \(Z\in \C T_pbD\cap T^{1,0}_pY\setminus\{0\}\). The domain \(D\) is
  called strictly pseudoconvex if any boundary point of \(D\) is
  strictly pseudoconvex
\end{definition}
\begin{example}\label{ex:BoundaryOfstrictlyPseudoconvexDomain}
  Let \((Y,T^{1,0}Y)\) be a complex manifold and let \(D\subset Y\) be
  a smoothly bounded strictly pseudoconvex domain with connected
  boundary. Let \(\rho \colon Y\to \R\) be a defining function for
  \(D\), that is, \(D=\{z\in Y\colon \rho(z)<0\}\) and \(d\rho\neq 0\)
  at any point \(p\in bD\). Put \(X=bD\),
  \(T^{1,0}X:=\C TbD\cap T^{1,0}Y\) and
  \(\alpha=-i\iota^*(\partial-\overline{\partial})\rho\) where
  \(\iota\colon X\to Y\) denotes the inclusion map. We have that
  \((X,T^{1,0}X)\) is a CR submanifold of \(Y\) and
  \((X,T^{1,0}X,\alpha)\) is a pseudohermitian manifold.
\end{example}

The following lemma shows that compact strictly pseudoconvex CR
manifolds which are embeddable into the complex Euclidean space can
be realized as in
Example~\ref{ex:BoundaryOfstrictlyPseudoconvexDomain} and can be seen as a consequence of 
the works of Boutet de Monvel~\cite{Bou75}, Boutet de
Monvel--Sj\"ostrand \cite{BS75}, Harvey--Lawson~\cite{HL75},
Burns~\cite{Bu:77}, Kohn~\cite{Koh85,Koh86}, Heunemann~\cite{Heu86},
Ohsawa~\cite{O84}, see also Grauert~\cite{Gr58,Gr94},
Catlin~\cite{MR1128581}, and Hill--Nacinovich~\cite{MR1289628}
(cf.~\cite{MM07,HM17}). Below, we will give a brief explanation of its proof.
\begin{lemma}\label{lem:CharacterizationEmbeddable}
  Let \((X,T^{1,0}X)\) be a compact strictly pseudoconvex CR
  manifold. The following are equivalent.
  \begin{itemize}
  \item [(1)]\(X\) is CR embeddable into the complex Euclidean space.
  \item [(2)]There exist a complex manifold \(Y\) and a smoothly
    bounded strictly pseudoconvex domain \(D\subset\subset Y\) such
    that \(X\) is CR isomorphic to the boundary \(bD\) of \(D\).
  \end{itemize}
\end{lemma}
\begin{proof}
  \underline{\((2)\Rightarrow (1)\):} It follows from Kohn~\cite{Koh85,Koh86} that the Kohn Laplacian  (see~\eqref{eq:DefKohnLaplacian}) on \(bD\) has closed range. Then from Boutet de Monvel~\cite{Bou75} (see also Boutet de Monvel--Sj\"ostrand~\cite{BS75}) we obtain (1).\\
  \underline{\((1)\Rightarrow (2)\):} It follows from 
  Harvey--Lawson~\cite{HL75}, that \(X\) bounds an analytic variety
  with only isolated singularities (see also~\cite{LY98}, \cite{HL00}). More precisely, there exists a
  compact smoothly bounded analytic variety \(V\) with boundary \(bV\)
   where the singular set
  \(\operatorname{Sing}(V)\) consists of finitely many points lying in
  \(\operatorname{Int}(V)\) and \(X\) is CR isomorphic to
  \(bV\). Here, \(\operatorname{Int}(V)\) denotes the set of interior 
  points of \(V=\operatorname{Int}(V)\cup bV\). We can blow up each singularity (see
  e.g.~\cite{MM07}*{Chapter~2.1}) and obtain a compact complex
  manifold \(M=\operatorname{Int}(M)\cup bM\) where the boundary
  \(bM\) is still CR isomorphic to \(X\). Hence, \(M\) is a compact
  complex manifold with strictly pseudoconvex boundary. Then it
  follows from Heunemann~\cite{Heu86} and Ohsawa~\cite{O84} that there
  exist a complex manifold \(Y\), a smoothly bounded strictly
  pseudoconvex domain \(D\subset \subset Y\) and a diffeomorphism
  \(F\colon M\to \overline{D}\) which is a biholomorphism between
  \(\operatorname{Int}(M)\) and \(D\) with \(F(bM)=bD \).
\end{proof}

The following lemma summarize some important results for strictly
pseudoconvex domains in Stein manifolds.
\begin{lemma}\label{lem:GlobalDefiningFunction}
  Let \((Y,T^{1,0}Y)\) be a Stein manifold and let
  \(D\subset\subset Y\) be a relatively compact domain with boundary
  of class \(\mathscr{C}^s\), \(s\in \N\cup\{\infty\}\), \(s\geq
  2\). The following are equivalent:
  \begin{itemize}
  \item[(1)] \(D\) is strictly pseudoconvex.
  \item[(2)] There exists a \(\mathscr{C}^s\)-smooth defining function
    for \(D\) which is strictly plurisubharmonic in a neighborhood of
    \(\overline{D}\).
  \end{itemize}
\end{lemma}
\begin{proof}
  See for example Fornæss--Stensønes~\cite{FoSt87} and
  Harz--Shcherbina--Tomassini~\cite{HST21}.
\end{proof}
We have the following.
\begin{lemma}\label{lem:ApproximationBySmoothDomains}
  Let \((Y,T^{1,0}Y)\) be a Stein manifold and let
  \(D\subset\subset Y\) be a relatively compact strictly pseudoconvex
  domain with boundary of class \(\mathscr{C}^s\), \(s\geq 2\). In any
  neighborhood of \(D\) in the \(\mathscr{C}^s\)-topology for domains
  there exists a smoothly bounded strictly pseudoconvex domain
  \(M\subset\subset D\).
\end{lemma}
\begin{proof}
  Choose any Hermitian metric on \(Y\). From
  Lemma~\ref{lem:GlobalDefiningFunction} we find that there exist an
  open neighborhood \(U'\subset\subset Y\) with
  \(\overline{D}\subset U'\) and a function
  \(\rho\in \mathscr{C}^s(U',\R)\) with
  \(D=\{z\in U'\colon \rho(z)<0\}\), \((d\rho)_p\neq 0\) for all \(p\)
  in an open neighborhood of \(bD\) and
  \(\partial\overline{\partial}\rho(Z,\overline{Z})\geq 2\) for all
  \(Z\in T_p^{1,0}Y\), \(|Z|=1\), \(p\in U'\). Choose an open
  neighborhood \(U\subset\subset U'\) with
  \(\overline{D}\subset U\) and let
  \(V\subset\subset V'\subset\subset U\) be open neighborhoods of
  \(bD\) with \((d\rho)_p\neq 0\) for all \(p\in V'\). Furthermore,
  there is \(A>0\) such that \(|\rho|\geq2A\) on \(U\setminus V\). We
  find a sequence
  \(\{\tilde{\rho}_k\}_{k\in \N}\subset \mathscr{C}^\infty(U',\R)\)
  such that \(\tilde{\rho}_k\to \rho\) in \(\mathscr{C}^s\)-norm on
  \(\overline{U}\) as \(k\to\infty\). For \(k\in \N\) put
  \(c_k=\min_{p\in U}(\tilde{\rho}_k(p)-\rho(p))\) and define
  \(\rho_k(p):=\tilde{\rho}_k(p)-c_k+1/k\). Since
  \(\lim_{k\to\infty}c_k+1/k=0\) we find \(\rho_k\to \rho\) in
  \(\mathscr{C}^s\)-norm on \(U\) as \(k\to\infty\). Then there exists
  \(k_0\in \N\) such that for all \(k\geq k_0\) we have
  \(\partial\overline{\partial}\rho_k(Z,\overline{Z})\geq 1\) for all
  \(Z\in T_p^{1,0}Y\), \(|Z|=1\), \(p\in \overline{U}\),
  \((d\rho_k)_p\neq 0\) for all \(p\in V\) and \(|\rho_k(p)|\geq A\)
  for all \(p\in U\setminus V\). Since \(\rho+1/k\leq\rho_k\) on \(U\)
  we find \(D_k:=\{z\in U\colon \rho_k(z)<0\}\subset\subset D\). For
  any \(k\geq k_0\) it follows from the construction that
  \(bD_k\subset V\) and hence that \(\rho_k\colon U\to\R\) is a
  strictly plurisubharmonic defining function for \(D_k\).  Then
  \(\rho_k\to \rho\) in \(\mathscr{C}^s\)-norm as \(k\to\infty\)
  yields the statement.
\end{proof}
\subsection{Rigid CR Manifolds}\label{sec:Setup}
Let \((X,T^{1,0}X)\) be a CR manifold with \(\dim_\R X=2n+1\),
\(n\geq 1\) and let \(\mathcal{T}\in \vbunsec{X,TX}\) be a real vector
field which is transversal and CR, meaning that \(\mathcal{T}\)
satisfies the following two properties:
\begin{itemize}
\item
  \(\displaystyle \C TX=T^{1,0}X\oplus \overline{T^{1,0}X}\oplus \C
  \mathcal{T}
  \,\,\,\,\,\,\,\,\,\,\,\,\,\,\,\,\,\,\,\,\,\,\,\,\,\,\,\,\text{
    (transversality)}\),
\item
  \(\displaystyle [\mathcal{T},\vbunsec{X,T^{1,0}X}]\subset
  \vbunsec{X,T^{1,0}X}\,\,\,\,\,\,\,\) (CR condition).
\end{itemize}

Given a smooth function \(\varphi\colon X\to \R\) with
\(\mathcal{T}\varphi\equiv 0\) we define
\[\mathcal{V}(\varphi)=\{Z-iZ(\varphi)\mathcal{T}\mid Z\in
  T^{1,0}X\}.\]
\begin{lemma}
  Given a smooth function \(\varphi\colon X\to \R\) with
  \(\mathcal{T}\varphi\equiv 0\) we have that
  \((X,\mathcal{V}(\varphi))\) is a codimension one CR manifold and
  \(\mathcal{T}\) is a transversal CR vector field.
\end{lemma}
\begin{proof}
  Given \(p\in X\) we notice that the map
  \(T^{1,0}_pX\ni Z\mapsto Z-iZ(\varphi)\mathcal{T}\in \C TX\) is
  linear and injective and defines a vector bundle morphism. It
  follows, that \(\mathcal{V}(\varphi)\) is a complex vector bundle of
  rank \(n\). Given
  \(V\in \mathcal{V}(\varphi)_p\cap
  \overline{\mathcal{V}(\varphi)_p}\) we find
  \(Z_1,Z_2\in T^{1,0}_pX\) such that
  \(Z_1-iZ_1(\varphi)\mathcal{T}=V=\overline{Z_2}+i\overline{Z_2}(\varphi)\mathcal{T}\). It
  follows that \(Z_1=Z_2=0\) and hence \(V=0\). Let
  \(V_1,V_2\in \vbunsec{X,\mathcal{V}(\varphi)}\) be two vector
  fields. We find \(Z_j\in \vbunsec{X,T^{1,0}X}\), \(j=1,2\), such
  that \(V_j=Z_j-iZ_j(\varphi)\mathcal{T}\) and compute
  \begin{eqnarray*}
    [V_1,V_2]&=& [Z_1,Z_2]-i[Z_1,Z_2(\varphi)\mathcal{T}]+i[Z_2,Z_1(\varphi)\mathcal{T}]-[Z_1(\varphi)\mathcal{T},Z_2(\varphi)\mathcal{T}]\\
             &=& [Z_1,Z_2]-i[Z_1,Z_2](\varphi)\mathcal{T} -iZ_2(\varphi)[Z_1,\mathcal{T}]+iZ_1(\varphi)[Z_2,\mathcal{T}]\\
             &&-Z_1(\varphi)\mathcal{T}(Z_2(\varphi))\mathcal{T}+Z_2(\varphi)\mathcal{T}(Z_1(\varphi))\mathcal{T}\\
             &=&[Z_1,Z_2]-i[Z_1,Z_2](\varphi)\mathcal{T}\\
             &&-iZ_2(\varphi)([Z_1,\mathcal{T}]-i[Z_1,\mathcal{T}](\varphi)\mathcal{T})+iZ_1(\varphi)([Z_2,\mathcal{T}]-i[Z_2,\mathcal{T}](\varphi)\mathcal{T})\\
             &=&[Z_1,Z_2]+iZ_2(\varphi)[\mathcal{T},Z_1]-iZ_1(\varphi)[\mathcal{T},Z_2]\\
             &&-i([Z_1,Z_2]+iZ_2(\varphi)[\mathcal{T},Z_1]-iZ_1(\varphi)[\mathcal{T},Z_2])(\varphi)\mathcal{T}.
  \end{eqnarray*}
  where we used \(\mathcal{T}\varphi\equiv 0\). Since
  \([\mathcal{T},Z_j]\in \vbunsec{X,T^{1,0}X}\) for \(j=1,2\) we find
  \([V_1,V_2]\in \vbunsec{X,\mathcal{V}(\varphi)}\).
\end{proof}
\begin{lemma}\label{lem:ExistenceOfOneFormSasakian}
  Let \((X,T^{1,0}X)\) be a codimension one CR manifold and let
  \(\mathcal{T}\in \vbunsec{X,TX}\) be real vector field which is
  transversal and CR. There exists a non-vanishing real one form
  \(\alpha^{T^{1,0}X,\mathcal{T}}\in \Omega^{1}(X)\) which is uniquely
  defined by the equations
  \[\alpha^{T^{1,0}X,\mathcal{T}}(\mathcal{T})\equiv1\text{ and } \ker
    \alpha^{T^{1,0}X,\mathcal{T}}=\operatorname{Re}\left(T^{1,0}X\right).\]
\end{lemma}
\begin{proof}
  The statement follows immediately from the transversality of the
  vector field \(\mathcal{T}\).
\end{proof}
\begin{definition}\label{def:SasakianManifold}
  Let \((X,T^{1,0}X)\) be a codimension one CR manifold and let
  \(\mathcal{T}\in \vbunsec{X,TX}\) be a real vector field. We call the
  triple \((X,T^{1,0}X,\mathcal{T})\) a Sasakian manifold if
  \(\mathcal{T}\) is transversal and CR with
  \[\forall Z\in T^{1,0}X\setminus\{0\}\colon
    \,\,\,\frac{1}{2i}d\alpha^{T^{1,0}X,\mathcal{T}}(Z,\overline{Z})>0\]
\end{definition}
\begin{remark}
  Given a Sasakian manifold \((X,T^{1,0}X,\mathcal{T})\) we have that
  \((X,T^{1,0}X)\) is a strictly pseudoconvex CR
  manifold. Furthermore, \(\operatorname{Re}T^{1,0}X\) defines a
  contact structure with contact form
  \(\alpha^{T^{1,0}X,\mathcal{T}}\) and respective Reeb vector field
  \(\mathcal{T}\).
\end{remark}
\begin{lemma}\label{lem:LeviFormSasakianDeformation}
  Let \((X,T^{1,0}X)\) be a codimension one CR manifold and let
  \(\mathcal{T}\in \vbunsec{X,TX}\) be real vector field which is
  transversal and CR.  Denote by
  \(\alpha:=\alpha^{T^{1,0}X,\mathcal{T}}\) the one form defined in
  Lemma~\ref{lem:ExistenceOfOneFormSasakian}. Given a function
  \(\varphi\in \mathscr{C}^\infty(X,\R)\) with
  \(\mathcal{T}\varphi\equiv 0\) put
  \(\alpha^\varphi:=\alpha^{\mathcal{V}(\varphi),\mathcal{T}}\) and
  \(\gamma^\varphi=\alpha^\varphi-\alpha\in \Omega^1(X)\). Then
  \(\gamma^\varphi(\mathcal{T})\equiv0\) and for any vector field
  \(Z\in \vbunsec{X,T^{1,0}X}\) one has
  \begin{itemize}
  \item [(i)] \(\gamma^\varphi(Z)=i(d\varphi) (Z)\),
    \(\gamma^\varphi(\overline{Z})=-i(d\varphi) (\overline{Z})\)
  \item[(ii)]
    \(\frac{1}{2i}d\alpha^\varphi(Z-iZ(\varphi)\mathcal{T},\overline{Z-iZ(\varphi)\mathcal{T}})=\frac{1}{2i}d\alpha(Z,\overline{Z})-\frac{1}{2i}\gamma^\varphi([Z,\overline{Z}])-\operatorname{Re}\left(Z(\overline{Z}\varphi)\right).\)
  \end{itemize}
\end{lemma}
\begin{proof}
  We have
  \(\gamma^\varphi(\mathcal{T})=\alpha^\varphi(\mathcal{T})-\alpha(\mathcal{T})=1-1=0
  \). Furthermore, given \(Z\in T^{1,0}X\) we find
  \(Z-iZ(\varphi)\mathcal{T}\in\mathcal{V}(\varphi)\) and hence
  \[0=\alpha^\varphi(Z-iZ(\varphi)\mathcal{T})=-iZ(\varphi)+\gamma^\varphi(Z).\]
  Since \(\gamma^\varphi\) is real we obtain (i).  Given
  \(Z\in \vbunsec{X,T^{1,0}X}\) we compute
  \begin{eqnarray*}
    [Z-iZ(\varphi)\mathcal{T},\overline{Z-iZ(\varphi)\mathcal{T}}]&=& [Z,\overline{Z}]+2i\operatorname{Re}\left(Z\overline{Z}\varphi\mathcal{T}+\overline{Z}(\varphi)[Z,\mathcal{T}]\right)\\
                                                                  &&+2i\operatorname{Im}\left(
                                                                     \overline{Z}(\varphi)[Z,\mathcal{T}](\varphi)\mathcal{T}\right).
  \end{eqnarray*}
  Since \([Z,\mathcal{T}]\in \vbunsec{X,T^{1,0}X}\) we have
  \(\alpha([Z,\mathcal{T}])=0\) and
  \(\gamma^\varphi[Z,\mathcal{T}]=i[Z,\mathcal{T}](\varphi)\). We
  conclude that
  \begin{eqnarray*}
    \alpha^{\varphi}([Z-iZ(\varphi)\mathcal{T},\overline{Z-iZ(\varphi)\mathcal{T}}])&=& \alpha([Z,\overline{Z}]) +\gamma^\varphi([Z,\overline{Z}])+2i\operatorname{Re}\left(Z\overline{Z}\varphi\right)
  \end{eqnarray*}
  Using Cartan's formula the claim follows.
\end{proof}
\begin{remark}\label{rmk:StrictlyPseudoconvexIsOpen}
  It follows from Lemma~\ref{lem:LeviFormSasakianDeformation} that
  given a compact Sasakian manifold \((X,T^{1,0}X,\mathcal{T})\) and a
  function \(\varphi\in \mathscr{C}^\infty(X,\R)\) with
  \(\mathcal{T}\varphi\equiv 0\) we have that
  \((X,\mathcal{V}(\varphi),\mathcal{T})\) is again a Sasakian
  manifold provided that the \(\mathscr{C}^1\)-norm of \(d\varphi\) is
  sufficiently small.
\end{remark}
Let \((X,T^{1,0}X)\) be a CR manifold and let
\(\mathcal{T}\in \vbunsec{X,TX}\) be real vector field which is
transversal and CR.  We define \(Y:=X\times \R\) and
\(T^{1,0}Y=T^{1,0}X\oplus\C \{\frac{\partial}{\partial
  s}-i\mathcal{T}\}\). Then \((Y,T^{1,0}Y)\) is a complex manifold and
\(X\ni x\mapsto (x,0)\in Y\) defines a CR embedding of \(X\) into
\(Y\). Hence \(X\) can be considered as a CR submanifold of \(Y\). Let
\(\varphi\colon X\to \R\) be a smooth function with
\(\mathcal{T}\varphi\equiv 0\). Consider
\(X_\varphi=\{(x,\varphi(x))\mid x\in X\}=\{(x,s)\in Y\mid
s=\varphi(x)\}\subset Y\) and
\(T^{1,0}X_\varphi:=\C TX_\varphi\cap T^{1,0}Y\). Since
\(\mathcal{T}\varphi\equiv 0\) we find that
\(\mathcal{T}\in \vbunsec{X_\varphi,TX_\varphi}\) defines a
transversal CR vector field.
\begin{lemma}\label{lem:IsoGraphDeformation}
  The map \(\operatorname{Pr}\colon Y\to X\),
  \(\operatorname{Pr}(x,s)=x\), defines a CR isomorphism between
  \((X_\varphi,T^{1,0}X_\varphi)\) and \((X,\mathcal{V}(\varphi))\)
  with \(Pr_*\mathcal{T}=\mathcal{T}\).
\end{lemma}
\begin{proof}
  Since \(X_\varphi\) is a graph over \(X\) in \(Y=X\times \R\) we
  observe that \(\operatorname{Pr}\mid_{X_\varphi}\) is a
  diffeomorphism between \(X_\varphi\) and \(X\) with
  \(\operatorname{Pr}_*\mathcal{T}=\mathcal{T}\). Then it is enough to
  show that
  \(\operatorname{Pr}_*T^{1,0}X_\varphi\subset
  \mathcal{V}(\varphi)\). Let \(\tilde{p}\in X_\varphi\) and
  \(\tilde{Z}\in T^{1,0}_{\tilde{p}}X_\varphi\) be arbitrary. We can
  write \(\tilde{p}=(p,\varphi(p))\) for some \(p\in X\) and
  \(\tilde{Z}=Z+c(\frac{\partial}{\partial s}-i\mathcal{T})\) for some
  \(Z\in T_p^{1,0}X\), \(c\in \C\).  Since
  \(\tilde{Z}\in\C TX_\varphi\) and \(\mathcal{T}\varphi\equiv 0\) we
  find \(c=Z(\varphi)\). It follows that
  \(\operatorname{Pr}_*\tilde{Z}=Z-iZ(\varphi)\mathcal{T}\in
  \mathcal{V}(\varphi)_p\).
\end{proof}
\begin{lemma}\label{lem:HolomorphicFunctionsOnCylinder}
  Given a function \(f\in H_b^0(X)\cap\mathscr{C}^\infty(X)\) with
  \(\mathcal{T}f=i\lambda f\) for some \(\lambda\in \R\). We have that
  \(\tilde{f}\colon Y\to\C\), \(\tilde{f}(x,s)=e^{\lambda s}f(x)\)
  defines a holomorphic function \(\tilde{f}\in\mathcal{O}(Y)\) on
  \(Y\).
\end{lemma}
\begin{proof}
  Since \(\tilde{f}\) is smooth we just need to show that
  \(\overline{\tilde{Z}}\tilde{f}=0\) for all
  \(\tilde{Z}\in T^{1,0}Y\). Given \(\tilde{Z}\in T^{1,0}_{(x,s)}Y\)
  we can write
  \(\tilde{Z}=Z+c(\frac{\partial}{\partial s}-i\mathcal{T})\) for some
  \(Z\in T^{1,0}_xX\) and \(c\in \C\). Since
  \((\overline{Z}\tilde{f})(x,s)=e^{\lambda s}(\overline{Z}f)(x)=0\)
  and
  \(((\frac{\partial}{\partial
    s}+i\mathcal{T})\tilde{f})(x,s)=e^{\lambda s}(\lambda
  f(x)+i(\mathcal{T}f)(x))=0\) the claim follows.
\end{proof}
\subsection{CR Submanifolds of the
  Sphere}\label{sec:SubmanifoldsSphere}
\begin{lemma}\label{lem:CRsubmanifoldsOfSpheresAreStrictlyPseudoconvex}
  Let \((X,T^{1,0}X)\) be a CR submanifold of \(\C^N\) with
  \(X\subset \mathbb{S}^{2N-1}\). For any \(\beta\in \R_+^{N}\) we
  have that \((X,T^{1,0}X,\iota^*\alpha_{\beta})\) is a
  pseudohermitian submanifold of \(\mathbb{S}^{2N-1}\) (see
  Definition~\ref{def:PseudoHermSubManiSphere}). Here,
  \(\iota\colon X\to \C^N\) denotes the inclusion and \(\alpha_\beta\)
  is as in Example~\ref{ex:PseudoHermSpheres}.
\end{lemma}
\begin{proof}
  Since \((X,T^{1,0}X)\) is a CR submanifold of \(\C^N\) we have
  \(T^{1,0}X=\C TX\cap T^{1,0}\C^N\). Since
  \(TX\subset T\mathbb{S}^{2N-1}\) and
  \(T^{1,0}\mathbb{S}^{2N-1}=\C T\mathbb{S}^{2N-1}\cap T^{1,0}\C^N\)
  we have \(T^{1,0}X=\C TX\cap T^{1,0}\mathbb{S}^{2N-1}\). Let
  \(\gamma\in \Omega^{1}(X)\) be a characteristic $1$-form for
  \((X,T^{1,0}X)\). Since \(\alpha_\beta(Z+\overline{Z})\) for all
  \(Z\in T^{1,0}\mathbb{S}^{2N-1}\) we find
  \(T^{1,0}X\subset \ker\iota^*\alpha_\beta\). Hence there exists a
  smooth function \(f\colon X\to \R\) such that
  \(\iota^*\alpha_\beta=f\gamma\). Given
  \(Z\in T_p^{1,0}X\setminus\{0\}\subset
  T_p^{1,0}\mathbb{S}^{2N-1}\setminus\{0\}\), \(p\in X\), we find
  \[0<\frac{d\alpha_\beta}{2i}(Z,\overline{Z})=\frac{df\wedge
      \gamma}{2i}(Z,\overline{Z})+\frac{f(p)}{2i}d\gamma(Z,\overline{Z})=\frac{f(p)}{2i}d\gamma(Z,\overline{Z}).\]
  It follows that \(f(p)\neq 0\) for any \(p\in X\) and hence that
  \(\iota^*\alpha_\beta\) is a characteristic $1$-form for
  \((X,T^{1,0}X)\) such that the associated Levi form is positive
  definite everywhere.
\end{proof}

\begin{lemma}\label{lem:SasakianSubmanifoldsOfSpheresAreRealAnalytic}
  Let \((X,T^{1,0}X,\mathcal{T})\) be a Sasakian manifold and
  \(F\colon X\to \C^N\) be a CR embedding with
  \(F_*\mathcal{T}=\mathcal{T}_\beta\circ F\) for some
  \(\beta\in \R_+^N\) and \(\mathcal{T}_\beta\) as in
  Example~\ref{ex:PseudoHermSpheres}. If
  \(F(X)\subset \mathbb{S}^{2N-1}\) we have that \(F(X)\) is real
  analytic.
\end{lemma}
\begin{proof}
  Consider the complex manifold \((M,T^{1,0}M)\) with \(M=X\times \R\)
  and
  \(T^{1,0}M=T^{1,0}X\oplus\C \{\frac{\partial}{\partial
    s}-i\mathcal{T}\}\). Write \(F=(f_1,\ldots,f_N)\). Since
  \(F_*\mathcal{T}=\mathcal{T}_\beta\circ F\) we find
  \(\mathcal{T}f_j=i\beta_jf_j\), \(1\leq j\leq N\). Hence, by
  Lemma~\ref{lem:HolomorphicFunctionsOnCylinder}, we have that the map
  \(\tilde{F}\colon M\to\C^N\),
  \(\tilde{F}(x,s)=(e^{\beta_1 s}f_1(x),\ldots,e^{\beta_N s}f_N(x))\)
  is holomorphic with \(\tilde{F}(x,0)=F(x)\) for all \(x\in
  X\). Since the differential of \(F\) is injective and
  \(\overline{\partial}\tilde{F}=0\) we have that the differential of
  \(\tilde{F}\) is injective near \(X\times\{0\}\). Since \(F\) is
  injective, it follows that for any point \(x_0\in X\) there is an
  open neighborhood \(U\subset M\) around \((x_0,0)\) such that
  \(\tilde{F}\) restricts to a holomorphic embedding of \(U\) into
  \(\C^N\). We find
  \[\frac{\partial}{\partial
      s}\tilde{F}=(x,0)\sum_{j=1}^{N}\beta_j\operatorname{Re}\left(f_j(x)\frac{\partial}{\partial
        z_j}\right).\] As a consequence, \(F(X)\) is locally the
  transversal intersection of the complex manifold \(\tilde{F}(U)\)
  and \(\mathbb{S}^{2N-1}\). It follows that \(F(X)\) is real
  analytic.
	 
\end{proof}
\begin{remark}
  As a consequence of
  Lemma~\ref{lem:SasakianSubmanifoldsOfSpheresAreRealAnalytic} any
  (smooth) Sasakian submanifold of a \(\beta\)-weighted sphere is real
  analytic.
\end{remark}

\subsection{Microlocal Analysis of the Szeg\H{o} Kernel}\label{sec:CRmanifoldsMicroLocal}
Let \((X,T^{1,0}X,\alpha)\) be a compact pseudohermitian manifold and
let \(\mathcal{T}\) denote the Reeb vector field associated to
\(\alpha\).  We choose a Hermitian metric $\langle\cdot|\cdot\rangle$
on \(\C TX\) so that $\langle\mathcal{T}|\mathcal{T}\rangle=1$ on $X$, \(T^{1,0}X\perp T^{0,1}X\), \(\mathcal{T}\perp T^{1,0}X\oplus T^{0,1}X\) and \(\langle V| W\rangle\in \R\) for \(V,W\in T_xX\), \(x\in X\).  With respect to the given
Hermitian metric $\langle\cdot|\cdot\rangle$, we consider the
orthogonal projection
\begin{equation}
  \pi^{(0,q)}:\Lambda^q\mathbb{C}T^*X\to T^{*0,q}X.
\end{equation}
The tangential Cauchy--Riemann operator is defined to be
\begin{equation}
  \label{tangential Cauchy Riemann operator}
  \overline{\partial}_b:=\pi^{(0,q+1)}\circ d:\Omega^{0,q}(X)\to\Omega^{0,q+1}(X).
\end{equation}
By Cartan's formula, we can check that
\begin{equation}
  \overline{\partial}_b^2=0.  
\end{equation}
We work with two volume forms on $X$.
\begin{itemize}
\item[(\emph{i})] A given smooth positive $(2n+1)$-form $dV(x)$.
\item[(\emph{ii})] The $(2n+1)$-form
  $dV_{\alpha}:=\frac{2^{-n}}{n!}\alpha\wedge\left(d\alpha\right)^n$
  given by the contact form $\xi$.
\end{itemize}
Take the $L^2$-inner product $(\cdot|\cdot)$ on $\Omega^{0,q}(X)$
induced by $dV$ and $\langle\cdot|\cdot\rangle$ via
\begin{equation}
  (f|g):=\int_X\langle f|g\rangle dV,~f,g\in\Omega^{0,q}(X).
\end{equation}
We denote by $L^2_{(0,q)}(X)$ the completion of $\Omega^{0,q}(X)$ with
respect to $(\cdot|\cdot)$, and we write
$L^2(X):=L^2_{(0,0)}(X)$. Furthermore, for \(q=0\) we will also use
the notation \((\cdot,\cdot):=(\cdot|\cdot)\) to denote the
\(L^2\)-inner product.  Let $\|\cdot\|$ denote the corresponding
$L^2$-norm on $L^2_{(0,q)}(X)$.  Let $\overline{\partial}_b^*$ be the
formal adjoint of $\overline{\partial}_b$ with respect to
$(\cdot|\cdot)$, and the Kohn Laplacian is defined by
\begin{equation}\label{eq:DefKohnLaplacian}
  \Box^{(0)}_{b}:=\overline{\partial}_b^*\overline{\partial}_b:\mathscr{C}^\infty(X)\to\mathscr{C}^\infty(X),
\end{equation}
which is not an elliptic operator because the principal symbol
$\sigma_{\Box^{(0)}_{b}}(x,\eta)\in\mathscr{C}^\infty(T^*X)$ has a
non-empty characteristic set
$\{(x,\eta)\in
T^*X:\eta=\lambda\alpha(x),~\lambda\in\mathbb{R}^*\}$. For later use,
we denote the symplectic cone
\begin{equation}
  \Sigma:=\{(x,\eta)\in T^*X:\eta=\lambda\alpha(x),~\lambda>0\}.
\end{equation}
In fact, $\Box^{(0)}_{b}$ can even not be hypoelliptic, that is,
$\Box^{(0)}_{b}u\in\mathscr{C}^\infty(X)$ cannot guarantee
$u\in\mathscr{C}^\infty(X)$. We extend $\overline{\partial}_b$ to the
$L^2$-space by
\begin{equation}\label{e-gue201223yydu}
  \overline{\partial}_b: \Dom\overline{\partial}_b\subset L^2(X)\to 
  L^2_{(0,1)}(X),\quad
  \Dom\overline{\partial}_b:=\{u\in L^2(X):\overline{\partial}_bu
  \in L^2_{(0,1)}(X)\}.
\end{equation} 
Let
$\overline{\partial}^*_{b,H}: \Dom\overline{\partial}^*_{b,H}\subset
L^2_{(0,1)}(X)\to L^2(X)$ be the $L^2$-adjoint of
$\overline{\partial}_{b}$. Let
$\Box^{(0)}_b: \Dom\Box^{(0)}_b\subset L^2(X)\to L^2(X)$, where
\[\Dom\Box^{(0)}_b=\{u\in L^2(X): u\in\Dom\overline{\partial}_b,
  \overline{\partial}_bu\in\Dom\overline{\partial}^*_{b,H}\}\] and
$\Box^{(0)}_bu=\overline{\partial}^*_{b,H}\,\overline{\partial}_{b}u$,
$u\in \Dom\Box^{(0)}_b$.

\begin{definition}
  The space of $L^2$ CR functions is given by
  \begin{equation}
    H_b^{0}(X):=\ker\Box^{(0)}_{b}:=
    \left\{u\in\Dom\Box^{(0)}_{b}:
      \Box^{(0)}_{b} u=0~\text{in the sense of distributions}\right\}.
  \end{equation}
\end{definition}

\begin{definition}
  The orthogonal projection
  \begin{equation}
    \Pi: L^2(X)\to\ker\Box^{(0)}_{b}
  \end{equation}
  is called the Szeg\H{o} projection on $X$, and we call its Schwartz
  kernel $\Pi(x,y)\in\mathscr{D}'(X\times X)$ the Szeg\H{o} kernel on
  $X$.
\end{definition}

In our case, we can understand the singularities of $\Pi$ by
microlocal analysis as in \cite{BS75}, see also \cites{Hs10,HM17}.

\begin{theorem}[\cite{Hs10}*{Part I Theorem 1.2}]
  \label{Boutet-Sjoestrand theorem}
  Let $(X,T^{1,0}X)$ be an orientable compact strictly pseudoconvex
  Cauchy--Riemann manifold with volume form \(dV\) such that the Kohn
  Laplacian on $X$ has closed range in $L^2(X)$. For any coordinate
  chart $(D,x)$ in $X$, we have
  \begin{equation}
    \label{eq:Szego FIO}
    \Pi(x,y)\equiv\int_0^{+\infty} e^{it\varphi(x,y)}s^\varphi(x,y,t)dt,
  \end{equation}
  where
  \begin{equation}\label{eq:varphi(x,y)}
    \begin{split}
      &\varphi(x,y)\in \mathscr{C}^\infty(D\times D),\\
      &\mbox{${\rm Im\,}\varphi(x,y)\geq 0$},\\
      &\mbox{$\varphi(x,y)=0$ if and only if $y=x$},\\
      &\mbox{$d_x\varphi(x,x)=-d_y\varphi(x,x)=\xi(x)$},
    \end{split}
  \end{equation}
  and
  \begin{equation}
    \label{s(x,y,t)}
    \begin{split}
      &s^\varphi(x,y,t)\in S^n_{1,0}(D\times D\times\mathbb R_+),\\
      &s^\varphi(x,y,t)\sim\sum^{+\infty}_{j=0}s^\varphi_j(x,y)t^{n-j}~\text{in}~S^n_{1,0}(D\times D\times\mathbb R_+),\\
      &s^\varphi_j(x,y)\in\mathscr{C}^\infty(D\times D),~j=0,1,2,\ldots,\\
      &s^\varphi_0(x,x)=\frac{1}{2\pi^{n+1}}\frac{dV_{\xi}}{dV}(x)\neq 0.
    \end{split}
  \end{equation}
\end{theorem}
In the local picture, near any point $p\in X$, let $\{W_j\}_{j=1}^n$
be an orthonormal frame of $T^{1,0}X$ in a neighborhood of $p$ such
that $\mathcal{L}_p(W_j,\overline{W}_s)=\delta_{js}\mu_j$, for some
$\mu_j>0$ and $j,s\in\{1,\ldots,n\}$.  Define
\begin{equation}
  \det\mathcal{L}_p:=\Pi_{j=1}^n\mu_j(p).
\end{equation} 
Then
\begin{equation}
  dV_\xi(x)=\det\mathcal{L}_x\sqrt{\det(g)}\,dx.
\end{equation}
where $g=(g_{jk})_{j,k=1}^{2n+1}$,
$g_{jk}=\langle\frac{\partial}{\partial x_j}| \frac{\partial}{\partial
  x_k}\rangle$.

\section{The Sasakian Case}\label{sec:SasakianCase}
\subsection{Proof of Theorem~\ref{thm:MainThmIntro}}\label{sec:ProofOfMainThmSasakian}
In this section we are going to prove Theorem~\ref{thm:MainThmIntro}.
Let \((X,T^{1,0}X)\) be a CR manifold and let \(\mathcal{T}\in \vbunsec{X,TX}\) be real vector field which is transversal and CR.
We define \(Y:=X\times \R\) and \(T^{1,0}Y=T^{1,0}X\oplus\C \{\frac{\partial}{\partial s}-i\mathcal{T}\}\). Then \((Y,T^{1,0}Y)\) is a complex manifold and \(X\ni x\mapsto (x,0)\in Y\) defines a CR embedding of \(X\) into \(Y\). Hence \(X\) can be considered as a CR submanifold of \(Y\). Let \(\varphi\colon X\to \R\) be a smooth function with \(\mathcal{T}\varphi\equiv 0\). Consider \(X_\varphi=\{(x,\varphi(x))\mid x\in X\}=\{(x,s)\in Y\mid s=\varphi(x)\}\subset Y\) and \(T^{1,0}X_\varphi:=\C TX_\varphi\cap T^{1,0}Y\). Since \(\mathcal{T}\varphi\equiv 0\) we find that \(\mathcal{T}\in \vbunsec{X_\varphi,TX_\varphi}\) defines a transversal CR vector field. Recall that by Lemma~\ref{lem:IsoGraphDeformation} we have that \((X_\varphi,T^{1,0}X_\varphi, \mathcal{T})\) and \((X,\mathcal{V}(\varphi),\mathcal{T})\) are isomorphic. Furthermore, it follows from Lemma~\ref{lem:HolomorphicFunctionsOnCylinder} that given a function \(f\in H_b^0(X)\cap\mathscr{C}^\infty(X)\) with \(\mathcal{T}f=i\lambda f\) for some \(\lambda\in \R\) we have that \(\tilde{f}\colon Y\to\C\), \(\tilde{f}(x,s)=e^{\lambda s}f(x)\) defines a holomorphic function \(\tilde{f}\in\mathcal{O}(Y)\) on \(Y\).

From now on we assume that \((X,T^{1,0}X,\mathcal{T})\) is a compact Sasakian manifold with \(\dim_\R X=2n+1\), \(n\geq 1\).
Let \(\alpha:=\alpha^{T^{1,0}X,\mathcal{T}}\in \Omega^1(X)\) be the smooth real one form satisfying
\[\alpha(\mathcal{T})=1 \text{ and } \C\ker \alpha=T^{1,0}X\oplus \overline{T^{1,0}X}\]
and let \(dV\) denote the volume form on \(X\) given by 
\[dV=\frac{1}{(2\pi)^{n+1}n!}\alpha\wedge (d\alpha)^n.\] Then \(dV\) is invariant under the flow of \(\mathcal{T}\). We denote by \(L^2(X)\) the space of \(L^2\)-functions on \(X\) with inner product given by \((f,g)_X=\int_Xf\overline{g}dV\),  \(f,g\in L^2(X)\).
Let \(0<\lambda_1\leq \lambda_2\leq\ldots\) denote the positive eigenvalues of \(-i\mathcal{T}\) acting on the space of CR functions and let \(\{f_j\}_{j\in\N}\) be a respective orthonormal system of smooth CR eigenfunctions that is \(f_j\in H^0_b(X)\cap \mathscr{C}^\infty\), \(-i\mathcal{T}f_j=\lambda_jf_j\), \((f_j,f_j)_X=1\) and \((f_j,f_\ell)_X=0\) for all \(j,\ell\in \N\) with \(j\neq \ell\).
Fix \(0<\delta_1<\delta_2<1\), choose a function \(\chi\in\mathscr{C}^\infty_c((\delta_1,\delta_2))\) with \(\chi\not\equiv 0\) and define \(\chi_k(t):=\chi(k^{-1}t)\) for \(t,k>0\).
For \(k>0\) define a CR map \(\mathcal{F}_k\colon X\to \C^{N_k}\)
\begin{eqnarray}\label{eq:MapFkOnX}
	\mathcal{F}_k(x)=(\chi_k(\lambda_1)f_1(x),\ldots,\chi_k(\lambda_{N_k})f_{N_k}(x))
\end{eqnarray}
with \(N_k=\# \{j\in \N\mid \lambda_j\leq k\}\).
From the results of Herrmann--Hsiao--Li~\cite{HHL20} and Herrmann--Hsiao--Marinescu--Shen~\cite{HHMS23} we obtain the following.
\begin{theorem}[cf. \cite{HHMS23},\cite{HHL20}]\label{thm:HHMSEmbedding}
	Let \((X,T^{1,0}X,\mathcal{T})\) be a compact Sasakian manifold with \(\dim_\R X=2n+1\), \(n\geq 1\). Given \(0<\delta_1<\delta_2<1\) and a function \(\chi\in\mathscr{C}^\infty_c((\delta_1,\delta_2))\) with \(\chi\not\equiv 0\) let \(\mathcal{F}_k\) be the map defined by \eqref{eq:MapFkOnX}. There exists \(k_0>0\) such that \(\mathcal{F}_k\) is a CR embedding for all \(k>k_0\). Furthermore, the function \(x\mapsto|\mathcal{F}_k(x)|^2\) has a smooth asymptotic expansion. More precisely, there exist smooth real functions \(a_1,a_2,\ldots\in \mathscr{C}^\infty(X)\) such that for any \(m,N\in \N\) there exists a constant \(C_{m,N}>0\) with  
	\[\left\||\mathcal{F}_k|^2-\sum_{j=0}^Na_jk^{n+1-j}\right\|_{\mathscr{C}^m(X)}\leq C_{m,N}k^{n-N}.\]
	In addition, we have that \(a_0\) is a constant function, given by \(a_0\equiv \int_0^1|\chi(t)|^2t^ndt\).
\end{theorem}
We note that since \((X,T^{1,0}X,\mathcal{T})\) is Sasakian the expansion part in Theorem~\ref{thm:HHMSEmbedding} follows from~\cite{HHL20} while the embedding part follows from~\cite{HHMS23}. 
Given \(c>0\) we put \(Y_c:=\{(x.s)\in Y\mid |s|<c\}\) which is an open neighborhood of \(X\times\{0\}\) in \(Y\).
\begin{lemma}\label{lem:ExistenceOfG}
	Let \((X,T^{1,0}X.\mathcal{T})\) be a compact Sasakian manifold with \(\dim_\R X=2n+1\), \(n\geq 1\) and let \(\{\lambda_j\}_{j\in\N}\) and \(\{f_j\}_{j\in \N}\) be as above. Define \(\tilde{f}_j(x,s):=e^{\lambda_js}f_j(x)\) for \(j\in \N\). For any \(m_0\in \N\) there exist \(c>0\), \(N\in\N\) such that 
	\[G\colon Y\to \C^{N-m_0+1},\,\, G(x,s)=(\tilde{f}_{m_0}(x,s),\tilde{f}_{m_0+1}(x,s),\ldots,\tilde{f}_N(x,s))\] defines a holomorphic embedding of \(Y_c\) into \(\C^{N-m_0+1}\setminus\{0\}\).
\end{lemma}
\begin{proof}
	Choose \(k>0\) large enough such that \(\mathcal{F}_k\) in Theorem~\ref{thm:HHMSEmbedding} is an embedding of \(X\) into some \(\C^{N}\setminus\{0\}\) and that \(k\delta_1>m_0\). It follows that  \(x\mapsto G(x,0)\) defines an embedding of \(X\) into \(\C^{N-m_0+1}\setminus\{0\}\). Since, in addition,  \(G\) is holomorphic it follows that \(dG\) is injective at any point \((x,0)\in Y\). Hence there exists \(\tilde{c}>0\) such that \(G|_{Y_{\tilde{c}}}\) is a holomorphic injective immersion. It follows that \(G\) defines a holomorphic embedding of \(Y_c\)  into \(\C^{N-m_0+1}\setminus\{0\}\) for some \(0<c<\tilde{c}\). 
\end{proof}
Let \(c\), \(N\) and \(G\) be as in Lemma~\ref{lem:ExistenceOfG} for some \(m_0\in\N\). To simplify the notation we assume without loss of generality that \(m_0=1\). For \(k>0\) define a holomorphic map 
\begin{eqnarray}\label{eq:MapFkOnM}
	F_k\colon Y\to \C^{N+N_k},\,\,\,\, F_k(x,s)=(e^{-k}G(x,s),H_k(x,s))
\end{eqnarray}
where \(H_k\colon Y\to \C^{N_k}\) is defined by
\[H_k(x,s)=C_\chi k^{\frac{-n-1}{2}}(\chi_k(\lambda_1)\tilde{f}_1(x,s),\ldots,\chi_k(\lambda_{N_k})\tilde{f}_{N_k}(x,s))\]
with  \(\{\tilde{f_j}\}_{j\in\N}\) as in Lemma~\ref{lem:ExistenceOfG} and \(C_\chi:=(\int_0^1|\chi(t)|^2t^ndt)^{-\frac{1}{2}}\).
In the following proposition we verify some properties of the map \(F_k\).
\begin{proposition}\label{pro:PropertiesOfFk}
	With the assumptions and definitions above we have:
	\begin{itemize}
		\item[(i)] For any \(k>0\) the map \(F_k\) defines a holomorphic embedding of \(Y_c\) into \(\C^{N+N_k}\setminus\{0\}\).  Furthermore, we have \(x\mapsto|F_k(x,0)|^2=1+O(k^{-1})\) for \(k\to\infty\) in \(\mathscr{C}^\infty\)-topology on \(X\). 
		\item [(ii)] For any \(k>0\) and \(x\in X\) we have that \( s\mapsto |F_k(x,s)|^2\) defines a positive, strictly increasing function on \(\R\) with
		\[\lim_{s\to -\infty}|F(x,s)|^2=0\text{ and } \lim_{s\to \infty}|F(x,s)|^2=\infty.\]
		\item[(iii)] There exist constants \(C,k_0>0\) such that for any \(k\geq k_0\), \(x\in X\) and \(s\in \R\) with \(|F_k(x,s)|=1\) we have \(|s|\leq Ck^{-2}\).
		\item[(iv)] For any \(C>0\) there exist constants \(A,k_0>0\) such that for any \(k\geq k_0\), \(x\in X\) and \(s\in \R\) with \(|s|\leq Ck^{-1}\) we have
		\[ 	\frac{\partial}{\partial s}|F_k(x,s)|^2\geq k/A.\]
		\item[(v)] For \(k>0\), \(s\in \R\) and \(\ell\in \N_0\) consider the smooth function
		\[r_{k,s,\ell}\colon X\to \R,\,\,\,\,r_{k,s,\ell}(x)=k^{-\ell}\left(\frac{\partial}{\partial s}\right)^\ell|F_k(x,s)|^2-b_\ell\]   with \(b_\ell=(\int_0^1|\chi(t)|^2t^{n+\ell}dt)(\int_0^1|\chi(t)|^2t^ndt)^{-1}\). Fix \(C>0\). There exists \(k_0>0\) such that for any \(\ell\in \N_0\) and \(m\in\N_0\) there is a constant \(C_{\ell,m}\) such that for all \(k\geq k_0\) and \(|s|\leq Ck^{-2}\) one has \(\|r_{k,s,\ell}\|_{\mathscr{C}^m(X)}\leq C_{m,\ell}k^{-1}\).  	 
	\end{itemize}
\end{proposition}
\begin{proof}\(\,\)\\
	\underline{Claim (i):} Since \(F_k=(e^{-k}G,H_k)\) is holomorphic and \(G\) defines an embedding of \(Y_c\) into \(\C^N\setminus\{0\}\), we find that \(F_k\) defines a holomorphic embedding of \(Y_c\) into \(\C^{N+N_k}\setminus\{0\}\) for any \(k>0\). Since \[|F_k(x,0)|^2=e^{-2k}|G(x,0)|^2+C^2_\chi k^{-n-1}|\mathcal{F}_k(x)|^2\]
	with \(\mathcal{F}_k\) as in \eqref{eq:MapFkOnX}, Claim~(i) follows immediately from Theorem~\ref{thm:HHMSEmbedding}. \\
	\underline{Claim (ii):} We write
	\begin{eqnarray}\label{eq:CalculationModulusF_k}
		|F_k(x,s)|^2&=&e^{-2k}\sum_{j=1}^N|\tilde{f}_j(x,s)|^2+C^2_\chi k^{-n-1}\sum_{j=1}^{N_k}|\chi_k(\lambda_j)|^2|\tilde{f}_j(x,s)|^2 \nonumber\\ 
		&=&e^{-2k}\sum_{j=1}^N|f_j(x)|^2e^{2\lambda_j s}+C^2_\chi k^{-n-1}\sum_{j=1}^{N_k}|\chi_k(\lambda_j)|^2|f_j(x)|^2e^{2\lambda_j s}.
	\end{eqnarray}   
	Let \(x\in X\) and \(k>0\) be arbitrary. Since \(G(x,0)\neq 0\) we find \(|\tilde{f}_{j_0}(x,0)|=|f_{j_0}(x)|\neq 0\) for some \(1\leq j_0\leq N\). Since \(\lambda_j>0\) and \(|f_j(x)|^2\geq 0\) for all \(j\in \N\) Claim (ii) follows from the properties of the exponential function and \eqref{eq:CalculationModulusF_k}.\\
	\underline{Claim (iii):} By Lemma~\ref{lem:ExistenceOfG} and Theorem~\ref{thm:HHMSEmbedding} we can choose constants \(0<A_1<1\), \(A_2>0\),  \(\tilde{k}_1>1\) such that \(|G(x,0)|^2>A_1\), \(|H_k(x,0)|^2>1-A_2/k\), \(1-A_2k^{-1}>0\) and \(A_2-kA_1e^{-2k}>0\) 
	holds for all \(x\in X\) and all \(k\geq \tilde{k}_1\). Since \(\lambda_j>0\), \(j\in \N\), and since \(\chi_k(\lambda_j)\neq 0\) implies \(\lambda_j\geq k\delta_1\) we have for \(s\geq 0\), \(x\in X\) and \(k\geq \tilde{k}_1\) that
	\begin{eqnarray*}
		|F_k(x,s)|^2\geq e^{-2k}|G(x,0)|^2+|H_k(x,0)|^2e^{2k\delta_1s}\geq A_1e^{-2k}+e^{2k\delta_1s}(1-A_2k^{-1}). 
	\end{eqnarray*}
	It follows that for any \(x\in X\), any \(k\geq \tilde{k}_1\) and any \(s\geq 0\) with \(|F_k(x,s)|^2=1\) we have \(1\geq A_1e^{-2k}+e^{2k\delta_1s}(1-A_2k^{-1})\) which leads to 
	\[2\delta_1 ks\leq\log\left(\frac{k-kA_1e^{-2k}}{k-A_2}\right)=\log\left(1+\frac{A_2-kA_1e^{-2k}}{k-A_2}\right)\leq\frac{C_2-kA_1e^{-2k}}{k-A_2}.\]
	Hence there is a constant \(C_1>0\) and \(k_1>\tilde{k}_1\) such that \(s\leq C_1k^{-2}\) holds for all \(x\in X\), \(k\geq k_1\) and \(s\geq 0\) with \(|F_k(x,s)|^2=1\). \\
	We can also choose constants \(B_1>1\), \(B_2>0\), \(\tilde{k}_2>0\) such that \(|G(x,0)|^2>B_1\), \(|H_k(x,0)|^2<1+B_2k^{-1}\),  and \(1/C_2-e^{-2k}>0\) holds for all \(x\in X\) and all \(k\geq k_2\). Since \(\lambda_j>0\), \(j\in \N\), and since \(\chi_k(\lambda_j)\neq 0\) implies \(\lambda_j\geq k\delta_1\) we have for \(s\leq 0\), \(x\in X\) and \(k\geq \tilde{k}_2\) that
	\begin{eqnarray*}
		|F_k(x,s)|^2\leq e^{-2k}|G(x,0)|^2+|H_k(x,0)|^2e^{2k\delta_1s}\leq B_1e^{-2k}+e^{2k\delta_1s}(1+B_2k^{-1}). 
	\end{eqnarray*}
	It follows that for any \(x\in X\), any \(k\geq \tilde{k}_2\) and any \(s\leq 0\) with \(|F_k(x,s)|^2=1\) we have \(1\leq B_1e^{-2k}+e^{2k\delta_1s}(1+B_2k^{-1})\) which leads to
	\begin{eqnarray*}
		2\delta_1 ks\geq\log\left(\frac{1-B_1e^{-2k}}{1+B_2k^{-1}}\right)&=&-\log\left(\frac{1+B_2k^{-1}}{1-B_1e^{-2k}}\right)\\
		&=&-\log\left(1+\frac{B_2k^{-1}+B_1e^{-2k}}{1-B_1e^{-2k}}\right)\geq-\frac{B_2k^{-1}+B_1e^{-2k}}{1-B_1e^{-2k}}.
	\end{eqnarray*} 
	Hence there is a constant \(C_2>0\) and \(k_2>\tilde{k}_2\) such that \(s\geq -C_2k^{-2}\) holds for all \(x\in X\), \(k\geq k_2\) and \(s\leq 0\) with \(|F_k(x,s)|^2=1\). \\
	Now, choosing \(k_0=\max\{k_1,k_2,1\}\) and \(C=\max\{C_1,C_2\}\), Claim (iii) follows.\\
	\underline{Claim (iv):} Let \(C>0\) be arbitrary. Choose \(k_0>1\) such that \(2^{-1}\leq|H_k(x,0)|^2\leq2\) holds for all \(k\geq k_0\) (see Theorem~\ref{thm:HHMSEmbedding}). Since  \(\chi_k(\lambda_j)\neq 0\) implies \(k\delta_1\leq\lambda_j\leq k\delta_2 \)  we find for  \(x\in X\), \(k\geq k_0\) and \(s\geq -C/k\) that
	\begin{eqnarray*}
		\frac{\partial}{\partial s}|F_k(x,s)|^2&=&2e^{-2k}\sum_{j=1}^N\lambda_j|f_j(x)|^2e^{2\lambda_j s}+2C^2_\chi k^{-n-1}\sum_{j=1}^{N_k}\lambda_j|\chi_k(\lambda_j)|^2|f_j(x)|^2e^{2\lambda_j s}\\
		&\geq& 2C^2_\chi k^{-n-1}\sum_{j=1}^{N_k}\lambda_j|\chi_k(\lambda_j)|^2|f_j(x)|^2e^{2\lambda_j s}\\
		&\geq& 2\delta_1 k e^{-2C\delta_2}C^2_\chi k^{-n-1}\sum_{j=1}^{N_k}|\chi_k(\lambda_j)|^2|f_j(x)|^2\\
		&\geq&2\delta_1 k e^{-2C\delta_2}|H_k(x,0)|^2\geq k\delta_1 e^{-2C\delta_2}
	\end{eqnarray*}  
	holds. 
	From the calculation above Claim (iv) follows.\\
	\underline{Claim (v):} Considering the functions \(\eta_\ell\colon \R\times\R\to\R\), \[ \eta_\ell(\tilde{s},t)=t^\ell e^{\tilde{s}t}|\chi(t)|^2\]
	and substituting \(\tilde{s}=ks\) Claim (v) follows from Lemma~\ref{lem:TaylorExpansionFunctionalCalculus} below.
\end{proof}

\begin{lemma}\label{lem:TaylorExpansionFunctionalCalculus}
	We use the notations and assumptions above. For any \(\tau\in\mathscr{C}^\infty(\R_+)\) denote by \(R^{\tau}\colon X\to \C\) the function defined by
	\[R^{\tau}(x)=\sum_{j=1}^\infty\tau(\lambda_j)|f_j(x)|^2.\]
	Given \(0<\delta_1<\delta_2\) and a smooth function \(\chi\colon (-1,1)\times\R\to \C \) with \(\chi(s,t)=0\) for \(|s|<1\) and \(t\in \R\setminus (-\delta_1,\delta_2)\) denote by \(\chi_{s,k}\colon  \R\to \R\), \(k>0\),  the function \(\chi_{s,k}(t):=\chi(s,tk^{-1})\).
	Fix \(C>0\). There exists \(k_0>0\) so that for any \(m\in \N\) there exists \(C_m>0\) such that
	\[\left\|k^{-n-1}R^{\chi_{s,k}}-\int_\R t^{n}\chi(0,t)dt\right\|_{\mathscr{C}^m(X)}\leq C_mk^{-1}\]
	holds for all \(k\geq k_0\) and \(|s|<Ck^{-1}\). 
\end{lemma}
\begin{proof}
	Choose an open interval \(I\subset\subset (-1,1)\) containing the origin. For \(k>0\) Denote by \(\chi_k^{(j)}\colon  I\times \R\to \C\), \(j\geq0\),  the function defined by \(\chi_k^{(j)}(s,t):=\left(\frac{\partial}{\partial s}\right)^j\chi(s,tk^{-1})\) and put \(\chi_k^{(j)}(t)=\chi_k^{(j)}(0,t)\), \(\chi_k(s,t)=\chi_k^{(0)}(s,t)\) . We have that for any \(j\in\N_0\) there is \(B_j>0\) such that \(\sup_{(s,t)\in I\times \R}|\chi_k^{(j)}(s,t)|^2\leq B_j\). Choose a smooth function \(\eta\colon \R\to [0,1]\) with \(\supp\eta\subset [2^{-1}\delta_1,2\delta_2]\) and \(\eta|_{[\delta_1,\delta_2]}\equiv 1\). Put \(\eta_k(t)=\eta(k^{-1}t)\) for \(t\in\R\), \(k>0\). By Taylor expansion we obtain that
	\[\left|\chi(s,t)-\sum_{j=0}^{N-1}\chi_1^{(j)}(0,t)\frac{s^j}{j!}\right|\leq B_N\frac{|s|^N}{N!}\]
	holds for all \(N\in\N\), \(s\in I\) and \(t\in \R\). Replacing \(t\) by \(tk^{-1}\) and using the assumptions on \(\chi\) leads to
	\[\left|\chi_{s,k}(t)-\sum_{j=0}^{N-1}\chi_k^{(j)}(t)\frac{s^j}{j!}\right|\leq B_N\eta_k(t)\frac{|s|^N}{N!}\]
	which implies
	\[\left|R^{\chi_{k,s}}(x)-\sum_{j=0}^{N-1}R^{\chi^{(j)}_k}(x)\frac{s^j}{j!}\right|\leq B_NR^{\eta_k}(x)\frac{|s|^N}{N!}.\]
	for all \(N\in\N\), \(s\in I\), \(x\in X\) and \(k>0\).
	Since \(\|R^{\eta_k}\|_{\mathscr{C}^0(X)}=O(k^{n+1})\) there is a constant \(\tilde{C}\) such that 
	\[\left\|R^{\chi_{k,s}}-\sum_{j=0}^{N-1}R^{\chi^{(j)}_k}\frac{s^j}{j!}\right\|_{\mathscr{C}^0(X)}\leq \tilde{C}C^NB_N\frac{k^{n+1-N}}{N!}\] holds for all \(N\in \N\) and \(s\in I\) with \(|s|\leq Ck^{-1}\).
	From Theorem~\ref{thm:HHMSEmbedding} we obtain \(\|R^{\chi_k^{(j)}}\|_{\mathscr{C}^m(X)}=O(k^{n+1})\). Furthermore, for any \(m\in\N\) there exists \(\tilde{C}_m>0\) such that \(\||f_j|^2\|_{\mathscr{C}^m(X)}\leq \tilde{C}_m(1+\lambda_j^m)\) holds for all \(j\in \N\). This leads to the a priori asymptotic
	\[\left\|R^{\chi_{k,s}}-\sum_{j=0}^{N-1}R^{\chi^{(j)}_k}\frac{s^j}{j!}\right\|_{\mathscr{C}^m(X)}=O(k^{n+1+m})\]
	uniformly in \(s\in I\) when \(k\to\infty\) for all \(m,N\in\N\). 
	From Hörmander's trick (see~\cite{Hoermander_2015}*{Lemma~7.7.2} and \cite{He18}*{Lemma~3.11}) we obtain that for any \(m,N\in\N\) there is a constant \(C_{m,N}>0\) such that
	\[\left\|R^{\chi_{k,s}}-\sum_{j=0}^{N-1}R^{\chi^{(j)}_k}\frac{s^j}{j!}\right\|_{\mathscr{C}^m(X)}\leq C_{m,N}k^{n+1-N}\]
	holds for all \(s\in I\) with \(|s|\leq Ck^{-1}\). Since \(k^{-n-1}R^{\chi^{(0)}_k}-\int_\R\chi(0,t)t^ndt=O(k^{-1})\) holds in \(\mathscr{C}^\infty\)-topology the claim follows.
	
\end{proof}
\begin{theorem}\label{thm:ConstructionOfVarphik}
	Let \((X,T^{1,0}X,\mathcal{T})\) be a compact Sasakian manifold with \(\dim_\R X=2n+1\), \(n\geq 1\), and let \(F_k\) be the map defined in \eqref{eq:MapFkOnM}. There exists \(k_0>0\) such that for any \(k\geq k_0\) there is a smooth function \(\varphi_k\colon X\to (-c,c)\) with \(\mathcal{T}\varphi_k\equiv 0\) uniquely defined by the equation \(|F_k(x,\varphi_k(x))|=1\), \(x\in X\). Furthermore, for any \(m\in \N_0\) there exists  a constant \(C_m>0\) such that for all \(k\geq k_0\) we have \(\|\varphi_k\|_{\mathscr{C}^m(X)}\leq C_mk^{-2}\). In addition, \((x,k)\mapsto\varphi_k(x)\) is smooth on \(X\times [k_0,\infty)\). 
\end{theorem}
\begin{proof}
	First, let \(k>1\) and \(x\in X\)  be arbitrary. From Proposition~\ref{pro:PropertiesOfFk}~(ii) we find that there exists an \(s=s(x,k)\) uniquely determined by the equation \(|F_k(x,s)|^2=1\). It follows from Proposition~\ref{pro:PropertiesOfFk}~(iii) that there exist \(C_1,k_1>0\) with \(|s(x,k)|\leq C_1k^{-2}\) for all \(x\in X\) and \(k\geq k_1\). Then, for some \(k_2>k_1\), we have that for any \(k\geq k_2\) the function \(\varphi_k\colon X\to (-c,c)\), \(\varphi_k(x)=s(x,k)\) is well defined. Using Proposition~\ref{pro:PropertiesOfFk}~(iv) we find \(C_3>0\) and \(k_3>k_2\) such that
	\[k/C_3\leq 	\frac{\partial}{\partial s}|F_k(x,s)|^2.\]
	holds for all \(k\geq k_3\) and \(|s|\leq C_1k^{-2}\). Since for fixed \(k\geq k_3\) we have that \(s=\varphi_k(x)\) is the unique solution of \(|F_k(x,s)|^2=1\) and we have by construction that \[\frac{\partial}{\partial s}|F_k(x,s)|^2|_{s=\varphi_k(x)}\geq\frac{k}{C_3}>0\]
	it follows from the implicit function theorem that \(\varphi_k\colon X\to (-c,c)\) is  smooth and since \((x,s,k)\mapsto |F_k(x,s)|^2\) is smooth that \((x,k)\mapsto \varphi_k(x)\) is smooth as well.\\
	We will now show that \(\mathcal{T}\varphi_k\equiv 0\) holds for all \(k\geq k_3\). For any \(j\in \N\) we have \(\mathcal{T}|f_j|^2=i\lambda_jf_j\overline{f_j}-i\lambda_jf_j\overline{f_j}=0\). Hence we conclude \(\mathcal{T}|F_k(x,s)|^2=0\) for any \(s\in \R\), \(x\in X\). Since \(x\mapsto |F_k(x,\varphi_k(x))|^2\) is constant and \(\varphi_k\) is smooth we find using the chain rule  and the definition of partial derivatives that
	\begin{eqnarray*}
		0&=&\mathcal{T} F_k(x,\varphi_k(x))= \mathcal{T}(\varphi_k)(x)\frac{\partial}{\partial s}|F_k(x,s)|^2|_{s=\varphi_k(x)}.
	\end{eqnarray*} 
	Since \(\frac{\partial}{\partial s}|F_k(x,s)|^2|_{s=\varphi_k(x)}\neq 0\) we find \(\mathcal{T}(\varphi_k)(x)=0\) for any \(x\in X\) and \(k>k_3\).\\
	It remains to show that for any \(m\in\N\) we have \(\|\varphi_k\|_{\mathscr{C}^m(X)}=O(k^{-2})\) for \(k\to \infty\). From our construction it follows that \(\|\varphi_k\|_{\mathscr{C}^0(X)}\leq C_1k^{-2}\). Furthermore, from the implicit function theorem we obtain for any \(x_0\in X\)
	\[d_x\varphi_k|_{x=x_0}=-\frac{d_x|F_k(x,s)|^2|_{(x,s)=(x_0,\varphi_k(x_0))}}{\frac{\partial}{\partial s}|F_k(x_0,s)|^2|_{s=\varphi_k(x_0)}}.\]
	Let \(p\in X\) be a point and let \((D,x_1,\ldots,x_{2n+1})\) be local coordinates around \(p\). For \(\alpha\in \N_0^{{2n+1}}\) put \(|\alpha|=\sum_{j=1}^{2n+1}\alpha_j\) and \(d^{\alpha}_x=\left(\frac{\partial}{\partial x_1}\right)^{\alpha_1}\ldots\left(\frac{\partial}{\partial x_{2n+1}}\right)^{\alpha_{2n+1}}\). 
	Furthermore, for \(\alpha\in \N_0^{2n+1}\), \(\ell\in\N_0\) and \(k>0\) define
	\[R_{k,\alpha,\ell}\colon D\to \R,\,\,\,\,R_{k,\alpha,\ell}(x)=\left(d_x^{\alpha}\left(\frac{\partial}{\partial s}\right)^\ell|F(x,s)|^2\right)|_{s=\varphi_k(x)}.\] 
	For \(m\in \N\) and \(\beta\in \N_0^{2n+1}\) we put
	\begin{eqnarray*}
		I(m,\beta)&:=&\left\{(\tau(1),\ldots,\tau(m))\in(\N_0^{2n+1})^m\colon \sum_{j=1}^m\tau(j)=\beta\right\},\\
		I'(m,\beta)&:=&\left\{\tau\in I(m,\beta)\colon |\tau(j)|\geq 1\,\forall 1\leq j\leq m\right\}.
	\end{eqnarray*}
	We note that from Proposition~\ref{pro:PropertiesOfFk}~(iv)  we have that \(x\mapsto 1/R_{k,0,1}\) defines a smooth function when \(k>0\) is large enough. For \(\beta\in \N_0^{2n+1}\) we compute
	\begin{eqnarray}\label{eq:Derivatives1overRk01}
		d_x^\beta\frac{1}{R_{k,0,1}}=\frac{1}{R_{k,0,1}^{|\beta|+1}}\sum_{\tau\in I(|\beta|,\beta)}c_\tau\prod_{\nu=1}^{|\beta|}d^{\tau(\nu)}_xR_{k,0,1}
	\end{eqnarray}
	with integer coefficients \(c_\tau\) independent of \(k\). 
	Furthermore, we find for \(\beta\in \N_0^{2n+1}\) that
	\begin{eqnarray}\label{eq:DerivativesRkalphal}
		d_x^\beta R_{k,\alpha,\ell}=R_{k,\alpha+\beta,\ell}+\sum_{\underset{|\gamma|<|\beta|}{\gamma\leq \beta}}\sum_{m=1}^{|\beta|-|\gamma|}R_{k,\alpha+\gamma,\ell+m}\sum_{\tau\in I'(m,\beta-\gamma)}c_{\gamma,m,\tau}\prod_{\nu=1}^m d_x^{\tau(\nu)}\varphi_k
	\end{eqnarray}
	holds with integer coefficients \(c_{\gamma,m,\tau}\) independent of \(k\). For \(1\leq j\leq 2n+1 \) define \(e(j)\in \N_0^{2n+1}\) by \(e(j)_j=1\) and \(e(j)_\ell=0\) for \(\ell\neq j\). Hence, we can write
	\begin{eqnarray}\label{eq:dvarphiInRks}
		\frac{\partial}{\partial x_j}\varphi_k=\frac{R_{k,e(j),0}}{R_{k,0,1}},\,\forall 1\leq j\leq 2n+1.
	\end{eqnarray} 
	We will show now via induction with respect to \(|\beta|\) that for all \(\beta\in \N^{2n+1}_0\) the following three statements hold true:
	\begin{eqnarray}
		d^\beta_x\frac{\partial}{\partial x_j}\varphi_k&=&O(k^{-2})\, \forall 1\leq j\leq 2n+1, \label{eq:induction1}\\
		d^\beta_x R_{k,0,\ell}&=&O(k^{\ell})\, \forall \ell\in \N_0,\label{eq:induction2}\\
		d^\beta_x R_{k,\alpha,\ell}&=&O(k^{-1+\ell})\, \forall \ell\in \N_0,\,\alpha\in \N^{2n+1}_0, |\alpha|\geq 1.\label{eq:induction3} 
	\end{eqnarray}
	From Proposition~\ref{pro:PropertiesOfFk} (iv) , (v) it follows that for any \(\alpha\in\N_0^{2n+1}\) with \(|\alpha|\geq 1\) and any \(\ell\in\N_0\) one has \(R_{k,0,\ell}=O(k^\ell)\), \(R_{k,\alpha,\ell}=O(k^{-1+\ell})\)  and \((R_{k,0,1})^{-1}=O(k^{-1})\). Then \eqref{eq:induction1}, \eqref{eq:induction2} and \eqref{eq:induction3} follow using~\eqref{eq:dvarphiInRks} in the case \(|\beta|=0\). Now assume that \eqref{eq:induction1}, \eqref{eq:induction2} and \eqref{eq:induction3} hold for all \(\beta\in \N^{2n+1}_0\) with \(|\beta|\leq N\) for some \(N\in\N_0\). Then, we choose \(\beta\in \N^{2n+1}_0\) with \(|\beta|=N+1\). Since for the \(\tau(\nu)\) on the right-hand side of~\eqref{eq:DerivativesRkalphal} we have \(1\leq|\tau(\nu)|\leq N+1\) it follows from the induction hypothesis that \eqref{eq:induction2} and \eqref{eq:induction3} hold true. Since 
	\[	d^\beta_x\frac{\partial}{\partial x_j}\varphi_k=\sum_{\alpha\leq \beta}\binom{\beta}{\alpha}d_x^{\beta-\alpha}R_{k,e(j),0}d_x^{\alpha}\frac{1}{R_{k,0,1}}\]  
	we obtain from~\eqref{eq:Derivatives1overRk01} using the induction hypothesis, \eqref{eq:induction2}, \eqref{eq:induction3} and \((R_{k,0,1})^{-1}=O(k^{-1})\) that \eqref{eq:induction1} holds. 
\end{proof}

\begin{proof}[\textbf{Proof of Theorem~\ref{thm:MainThmIntro}}]
	Let \(C_0>0\) be arbitrary and let \(F_k=(e^{-k}G,H_k)\colon Y\to \C^{N+N_k}\setminus\{0\}\) be the map defined in \eqref{eq:MapFkOnM}. Note that we can choose \(m_0\) for the map \(G\) arbitrary large, meaning that after removing the zero components from \(F_k\) we have for any \(k>0\) that the components of the map \(F_k\) arise from eigenfunctions of \(-i\mathcal{T}\) acting on \(H^0_b(X)\) for eigenvalues in an interval \([C_0,k]\). Furthermore, we have \(N+N_k=O(k^{n+1})\) (see \cite{HHL20}*{Theorem~1.3}). From now on  consider \(k\in\N\). Let \(\{\varphi_k\}_{k\geq k_0}\) be the sequence of functions given by Theorem~\ref{thm:ConstructionOfVarphik}. Put \(X_{\varphi_k}=\{(x,s)\in Y\colon s=\varphi_k(x)\}\subset Y_c\). Then \(X_{\varphi_k}\) is a CR submanifold of \(Y\) contained in \(Y_c\) with a CR transversal vector field given by \(\mathcal{T}\). Since \(F_k\) is holomorphic on \(Y\) we have that the function \(\nu_k\colon Y\to\R\), \(\nu_k(x,s)=|F_k(x,s)|^2-1\) is real analytic on \(Y\) with \(X_{\varphi_k}=\{(x,s)\in Y\colon \nu_k(x,s)=0\}\) by the construction of \(\varphi_k\). From Proposition~\ref{pro:PropertiesOfFk}~(iv) it follows that \(d\nu_k\neq 0\) in a neighborhood of \(X_{\varphi_k}\). Hence \(\nu_k\) is a real analytic defining function for \(X_{\varphi_k}\). It follows that \(X_{\varphi_k}\) is areal analytic submanifold of \(Y\).   Furthermore, the restriction \(\hat{F}_k\) of \(F_k\) to \(X_{\varphi_k}\) defines a CR embedding of \(X_{\varphi_k}\) into the unit sphere in \(\C^{N+N_k}\) by construction. 
	A direct calculation shows that \({\hat{F_k}}_{*}\mathcal{T}=\mathcal{T}_{\beta(k)}\circ \hat{F_k}\) for some \(\beta(k)\in \R_+^{N+N_k}\) where the entries of \(\beta(k)\) lie in the interval \([C_0,k]\). Then Lemma~\ref{lem:IsoGraphDeformation} yields a CR embedding \(\tilde{F}_k\) of \((X,\mathcal{V}(\varphi_k),\mathcal{T})\) into a sphere satisfying the required properties. 
	We note that since \(\tilde{F}_k\) yields a CR embedding of \((X,\mathcal{V}(\varphi_k),\mathcal{T})\) into a sphere it follows from Lemma~\ref{lem:CRsubmanifoldsOfSpheresAreStrictlyPseudoconvex} that  \((X,\mathcal{V}(\varphi_k),\mathcal{T})\) is a Sasakian manifold for any \(k\geq k_0\). 
\end{proof}

\subsection{Proof of Corollary~\ref{cor:LocalEmbedding03Intro}}\label{sec:ProofCorollarySasakainLocal}
	We first consider the following model. Let \(I\subset (-\pi,\pi)\) be an open interval and consider the CR manifold \((\mathbb{B}^n\times I,T^{1,0}\mathbb{B}^n)\) with transversal CR vector field \(\mathcal{T}:= \frac{\partial}{\partial t}\) where we use  coordinates \((z,t)=(z_1,\ldots,z_n,t)\) on \(\mathbb{B}^n\times I\) and \(\mathbb{B}^n\) denotes the unit ball in \(\C^n\). Given  a smooth function \(\varphi\) defined on an open neighborhood of  \(\overline{\mathbb{B}^n}\) consider the CR structure \(\mathcal{V}(\varphi)\) on \(\mathbb{B}^n\times I\) (see~\eqref{eq:DefDeformedSasakianStructure}). Here we  consider  \(\varphi\) as a function on \({\mathbb{B}^n}\times \R\) independent of \(t\) and write \(\varphi\) instead of \(\varphi|_{\mathbb{B}^n\times I}\).  

Under the assumption that \(-\varphi\) is strictly plurisubharmonic we have that \((\mathbb{B}^n\times I,\mathcal{V}(\varphi),\mathcal{T})\) is a Sasakian manifold. Furthermore, it follows from a result due to Baouendi--Rothschild--Treves~\cite{BRT85} that given any Sasakian manifold \((\tilde{X},T^{1,0}\tilde{X},\tilde{\mathcal{T}})\), \(\dim_\R \tilde{X}=2n+1\), \(n\geq 1\), and a point \(p\in \tilde{X}\) there is an open neighborhood \(U\subset \tilde{X}\) around \(p\) such that \((U,T^{1,0}\tilde{X}|_U,\tilde{\mathcal{T}}|_U)\) and \((\mathbb{B}^n\times I,\mathcal{V}(\varphi),\mathcal{T})\) are isomorphic as Sasakian manifolds for some specific choice of \(\varphi\) and \(I\). Hence, in order to prove the local approximation result for Sasakian manifolds (Corollary~\ref{cor:LocalEmbedding03Intro}) it is enough to consider  \((\mathbb{B}^n\times I,\mathcal{V}(\varphi),\mathcal{T})\) as above.
\begin{proof}[\textbf{Proof of Corollary~\ref{cor:LocalEmbedding03Intro}}]
 Let \(L=O(1)\rightarrow \C\mathbb{P}^n\) be the dual of the tautological line bundle over the complex projective space of dimension \(n\geq 1\). We will construct a metric \(h\) on \(L^m\) for some positive integer \(m\geq 1\) which has positive curvature and such that an open set \(U\) in the circle bundle of \((\C\mathbb{P}^n,L^m,h)\) is isomorphic to \((\mathbb{B}^n\times I,\mathcal{V}(\varphi),\mathcal{T})\)  as Sasakian manifold.  Let \(V=\{[z_0,\ldots,z_n]\in \C\mathbb{P}^n\colon z_0\neq 0\}\). Then \(V\) can be identified with \(\C^n\) via \(\C^n\ni z\mapsto [1,z_1,\ldots, z_n]\in V\). Assume that \(\varphi\) is defined on \(B(0,1+4\varepsilon)\) for some \(\varepsilon>0\) where \(B(0,r)\) denotes the open ball in \(\C^n\) centered in the origin  with radius \(r>0\). Let \(\eta\colon \C^n\to [0,1]\) be a smooth function with \(\supp \eta\subset B(0,1+4\varepsilon)\) such that \(\eta\equiv 1\) on \(B(0,1+3\varepsilon)\). Let \(\chi\colon \R \to \R\) be a smooth increasing and convex function such that \(\chi(t)= 0\) for \(t\leq \log(1+(1+\varepsilon)^2)\) and \(\chi(t)=t\) for \(t\geq \log(1+(1+2\varepsilon)^2)\). We have that \(z\mapsto \chi(\log(1+|z|^2))\) defines a smooth plurisubharmonic function on \(\C^n\) which is  strictly plurisubharmonic for \(|z|>1+2\varepsilon\). Since \(z\mapsto-\eta(z)\varphi(z)\) is smooth on \(\C^n\) and strictly plurisubharmonic for \(|z|<1+3\varepsilon\) it follows that for \(m\in \N\) large enough we have that \(\psi(z):=-\eta(z)\varphi(z)+m\chi(\log(1+|z|^2))\) defines a smooth strictly plurisubharmonic function on \(\C^n\) such that \(\psi(z)=-\varphi(z)\) for \(|z|< 1+\varepsilon\) and \(\psi(z)=m\log(1+|z|^2)\) for \(|z|> 1+4\varepsilon\). Since \(z\mapsto \log(1+|z|^2)\) is a Kähler potential for the Fubini-Study metric on \(V\) we find that \(\psi\) defines a positive metric \(h\) on \(L^m\). From the construction it follows that the circle bundle of \((\C\mathbb{P}^n,L^m,h)\) is a compact Sasakian manifold containing an open subset \(U\) which is isomorphic  to \((\mathbb{B}^n\times I,\mathcal{V}(\varphi),\mathcal{T})\)  as Sasakian manifold. Hence, we obtain the conclusion of Corollary~\ref{cor:LocalEmbedding03Intro} from Theorem~\ref{thm:MainThmIntro}.
\end{proof}

\subsection{Proof of Theorem~\ref{thm:ForstnericForSasakian}}\label{sec:NONWeightedSpheres}
In this section we are going to prove Theorem~\ref{thm:ForstnericForSasakian}. The main argument basically follows the idea of Forstneri\v{c}~\cite{Fors86}. However, it should be mentioned that in the Sasakian case the calculations are much more simpler compared to those appearing in the prove of the general result due to Forstneric. Furthermore, as already mentioned in the introduction, we note that for the Sasakian case the conclusion of
Theorem~\ref{thm:ForstnericForSasakian} is less surprising compared to
the result of Forstneri\v{c}~\cite{Fors86} for the general case since
any Sasakian submanifold of a sphere has to be automatically real
analytic (see Lemma~\ref{lem:SasakianSubmanifoldsOfSpheresAreRealAnalytic}). However, since we start with a smooth Sasakian manifold, it is not clear from the beginning how to distinguish between analyticity or non-analyticity in that context. So we provide a proof of Theorem~\ref{thm:ForstnericForSasakian} which do not use the fact that  Sasakian submanifolds of spheres are real analytic. Moreover, the proof will lead to structural insights about the set of Sasakian deformations which will be used to prove Theorem~\ref{thm:NoIsolatedPoint} in the next section.

Let \((X,T^{1,0}X,\mathcal{T})\) be a compact Sasakian manifold of dimension \(2n+1\), \(n\geq 1\). 
We fix a family of \(\mathscr{C}^m\)-norms on \(X\). For \(m\in\N\cup\{\infty\}\) put \(\mathscr{C}^m_\mathcal{T}(X,\R)=\{\varphi\in \mathscr{C}^m(X,\R)\colon \mathcal{T}\varphi\equiv 0\}\) and consider the following sets
\begin{eqnarray*}
	\mathcal{M}'&=&\{\varphi\in \mathscr{C}_\mathcal{T}^\infty(X,\R)\colon \mathcal{V}(\varphi) \text{ is strictly pseudoconvex}\}\\
	\mathcal{S}'&=&\{\varphi\in \mathcal{M}\colon (\mathcal{V}(\varphi),\mathcal{T})\text{ is induced by a weighted sphere}\}.
\end{eqnarray*}
For \(m\in\N_0\cup\{\infty\}\) we consider the \(\mathscr{C}^m\)-topology on \(\mathscr{C}^\infty_\mathcal{T}(X,\R)\) which is the subset topology coming from the \(\mathscr{C}^m\)-norm restricted to \(\mathscr{C}^\infty(X)\). Note that for \(m\geq 2\) we  have that \(\mathcal{M}'\) is open in \(\mathscr{C}^\infty_\mathcal{T}(X,\R)\) with respect to the \(\mathscr{C}^m\)-topology (see Remark~\ref{rmk:StrictlyPseudoconvexIsOpen}). We denote by \(0<\lambda_1< \lambda_2<\ldots\) the positive eigenvalues of \(-i\mathcal{T}\) acting on the space of CR functions and let \(V_j\), \(j\in\N\), be the eigenspace for \(\lambda_j\), that is, we have \(f\in H^0_b(X)\cap \mathscr{C}^\infty(X)\) with \(-i\mathcal{T}f=\lambda_jf\) if and only if \(f\in V_j\).

\begin{lemma}\label{lem:DescriptionOfBetaMaps}
	Given \(N\in\N\), \(\varphi\in \mathcal{M}\), \(\beta\in \R_+^N\) and a smooth map \(F\colon X\to \C^N \) which is CR with respect to the structure \(\mathcal{V}(\varphi)\) and such that \(F_*\mathcal{T}=\mathcal{T}_\beta\circ F\). Then there exist  \(1\leq j_1,\ldots,j_N\) and \(f_{j_\ell}\in V_{j_\ell}\), \(1\leq \ell \leq N\), such that 
	 \(F=(e^{\lambda_{j_1}\varphi}f_{j_1},\ldots,e^{\lambda_{j_N}\varphi}f_{j_N})\).
\end{lemma}
\begin{proof}
	We can write \(F=(F_1,\ldots,F_N)\) where \(F_\ell\), \(1\leq \ell\leq N\) are smooth CR functions on \(X\) with respect to the CR structure \(\mathcal{V}(\varphi)\). From \(F_*\mathcal{T}=\mathcal{T}_\beta\circ F\) it follows (see Example~\ref{ex:PseudoHermSpheres}) that \(-i\mathcal{T}F_\ell=\beta_\ell F_\ell\). For the case that \(F_\ell= 0\) is the zero function we can choose \(f_{j\ell}=0\). If \(F_\ell\neq0\) is not the zero function it follows that \(F_\ell\) is an eigenvector of \(-i\mathcal{T}\) for the eigenvalue \(\beta_\ell\). From Lemma~\ref{lem:IsoGraphDeformation} and Lemma~\ref{lem:HolomorphicFunctionsOnCylinder} it then follows that \(F_je^{-\beta_j\varphi}\) is a smooth CR function on \(X\) with respect to the CR structure \(T^{1,0}X\). Furthermore, one finds \(-i\mathcal{T}F_je^{-\beta_j\varphi}=\beta_jF_je^{-\beta_j\varphi}\). Hence, there exists an \(j_\ell \in \N\) such that \(\beta_j=\lambda_{j_\ell}\) and we can choose \(f_{j_\ell}:=F_je^{-\beta_j\varphi}\in V_{j_\ell}\).
\end{proof}
Let \(\operatorname{Map}(X,\R)\) denote the space of real valued functions on \(X\) and let \(1\!\!1\) denote the constant function \(x\mapsto 1\).  For \(N\in\N\) we define a map
\begin{eqnarray*}
	P_N\colon  (V_1)^N\times\cdots\times (V_N)^N\times \operatorname{Map}(X,\R)&\to& \operatorname{Map}(X,\R)\\
	(f_{1,1},\ldots,f_{N,N},\varphi)&\mapsto& \sum_{j=1}^N e^{2\lambda_{j}\varphi}\sum_{\ell=1}^N|f_{j,\ell}|^2.
\end{eqnarray*}
It is easy to check that \(P_N\) preserve the spaces \(\mathscr{C}^0(X,\R)\), \(\mathscr{C}^m(X,\R)\), \(\mathscr{C}_\mathcal{T}^m(X,\R)\), \(m\in\N\cup\{\infty\}\).
The following lemma is crucial and follows immediately from Lemma~\ref{lem:DescriptionOfBetaMaps}.
\begin{lemma}\label{lem:ConditionForVarphiInS}
	With the assumptions and notation above, for any \(\varphi\in \mathcal{S}'\) there exist \(N\in \N\) and \(F\in (V_1)^N\times\cdots\times (V_N)^N\) such that \(P_N(F,\varphi)=1\!\!1\).  
\end{lemma}
Furthermore, we have the following.
\begin{lemma}\label{lem:RegularityOfSN}
	Let  \(\varphi\colon X\to \R\) be a function and \(N\in\N\), \(F\in (V_1)^N\times\cdots\times (V_N)^N\) so that \(P_N(F,\varphi)= 1\!\!1\).  Then \(\varphi\in \mathscr{C}_\mathcal{T}^\infty(X,\R)\).
\end{lemma}
\begin{proof}
	First we prove a uniqueness statement: Given another function \(\tilde{\varphi}\colon X\to \R\)  with \(P_N(F,\tilde{\varphi})= 1\!\!1 \) we have \(\tilde{\varphi}=\varphi\). Let \(x\in X\) be arbitrary. Since \(\lambda_j>0\) and \( |f_{j,\ell}(x)|^2\geq 0\) for all \(j,\ell\) we find that the function \(\R \ni s\mapsto G(x,s):=\sum_{j=1}^N e^{2\lambda_{j}s}\sum_{\ell=1}^N|f_{j,\ell}(x)|^2\) is increasing. From \(P_N(F,\tilde{\varphi})(x)=1\) we conclude that \(|f_{j,\ell}(x)|^2>0\) for at least one pair \( (j,\ell)\in\{1,\ldots,N\}^2\) and hence that \(s\mapsto G(x,s)\) is strictly increasing. Since  \(G(x,\tilde{\varphi}(x))=1=G(x,\varphi(x))\) we find \(\tilde{\varphi}(x)=\varphi(x)\).\\
	Now we know that \(\varphi\colon X\to \R\) is the uniquely defined function which satisfies \(G(x,s)=1\) if and only if \(s=\varphi(x)\). Furthermore, we have that \(G\colon X\times \R\to\R\) is smooth with
	\begin{eqnarray}\label{eq:DerivativeInSNonEmbPart}
		\frac{\partial}{\partial s}G(x,s)=\sum_{j=1}^N 2\lambda_je^{\lambda_j s}\sum_{\ell=1}^N|f_{j,\ell}|^2>0.
	\end{eqnarray}
	(Here we again use \(\lambda_j>0\) for all \(j\in \N\) and the argument that for any \(x\in X\)  there is at least one \( (j,\ell)\in\{1,\ldots,N\}^2\) with \(|f_{j,\ell}(x)|^2>0\) since \(P_N(F,\varphi)(x)=1\).) With \eqref{eq:DerivativeInSNonEmbPart} we conclude \(\varphi\in \mathscr{C}^\infty(X,\R)\) from the implicit function theorem. Since \(-i\mathcal{T}f_{j,\ell}=\lambda_jf_{j,\ell}\) we have  \(\mathcal{T}|f_{j,\ell}|^2=0\). Since \(\mathcal{T}P_N(F,\varphi)=0\) we find for any \(x\in X\) that
	\[0=(\mathcal{T}\varphi)(x)\sum_{j=1}^N 2\lambda_j\sum_{\ell=1}^N|f_{j,\ell}|^2=(\mathcal{T}\varphi)(x)\frac{\partial}{\partial s}G(x,s)|_{s=\varphi(x)}.\]
	Using \eqref{eq:DerivativeInSNonEmbPart} we obtain \(\mathcal{T}\varphi\equiv 0\) and hence  \(\varphi\in \mathscr{C}_\mathcal{T}^\infty(X,\R)\).
\end{proof}
Now for \(N\in\N\) consider the space
\[S_N:=\{\varphi\in \mathscr{C}_\mathcal{T}^\infty(X,\R)\colon \exists F\in (V_1)^N\times\cdots\times (V_N)^N\text{ s.t. } P_N(F,\varphi)=1\!\!1.\}\]
and put \(S=\bigcup_{j=1}^\infty S_N\). 
Since \(\mathcal{S}'\subset S\) (by Lemma~\ref{lem:ConditionForVarphiInS}) and observing that \(\mathcal{M}'\) is open in \(\mathscr{C}_\mathcal{T}^\infty(X,\R)\) we will prove Theorem~\ref{thm:ForstnericForSasakian} by showing that the set \(S_N\) is nowhere dense in \(\mathscr{C}_\mathcal{T}^\infty(X,\R)\) with respect to the \(\mathscr{C}^\infty\)-topology. 
\begin{lemma}\label{lem:SNisClosed}
	With the assumptions and notations above let \(\varphi\colon X\to \R\) be a function. Assume that there is a constant \(C>0\) and a sequence \(\{\varphi_k\}_{k\in \N}\subset S_N\) such that \(|\varphi_k(x)|\leq C\) for all \(k\in \N\), \(x\in X\) and \(\lim_{k\to\infty}\varphi_k(x)=\varphi(x)\) for all  \(x\in X\). Then \(\varphi\in S_N\).  
\end{lemma}
\begin{proof}
	Since \(\varphi_k\in S_N\) there exists \(F^{(k)}\in(V_1)^N\times\cdots\times (V_N)^N\) such that \(P_N(F^{(k)},\varphi_k)=1\!\!1\). Then consider the sequence \(\{(F^{(k)},\varphi_k)\}_{k\in\N}\). For \(1\leq j,\ell\leq N\), \(k\in N\) and \(x\in X\) we have
	\[1=P_N(F^{(k)},\varphi_k)(x)= \sum_{j=1}^N e^{2\lambda_{j}\varphi_k(x)}\sum_{\ell=1}^N|f^{(k)}_{j,\ell}(x)|^2\geq e^{-2\lambda_NC}|f_{j,\ell}^{(k)}(x)|^2.\] 
	Hence, for each  \(1\leq j,\ell\leq N\) we have that \(\{f^{(k)}_{j,\ell}\}\) is a sequence in \(V_j\) which is bounded with respect to the sup-norm for functions. Since each \(V_j\) is finite dimensional we obtain that \(\{f^{(k)}_{j,\ell}\}_{k\in\N}\) has subsequence which converges to some \(f_{j,\ell}\in V_j\) with respect to the sup-norm.   
	It follows that by passing to a subsequence which we also denote by \(\{(F^{(k)},\varphi_k)\}_{k\in\N}\) we can assume that \(\lim_{k\to\infty}F^{(k)}=F \in(V_1)^N\times\cdots\times (V_N)^N\). Given any \(x\in X\) we have
	\begin{eqnarray*}
		1&=&\lim_{k\to\infty}P_N(F^{(k)},\varphi_k)(x)\\
		&=&\lim_{k\to\infty}\sum_{j=1}^N e^{2\lambda_j\varphi_k(x)}\sum_{\ell=1}^N|f^{(k)}_{j,\ell}(x)|^2=\sum_{j=1}^N e^{2\lambda_j\varphi(x)}\sum_{\ell=1}^N|f_{j,\ell}(x)|^2=P_N(F,\varphi)(x).
	\end{eqnarray*}
	It follows that \(P_N(F,\varphi)=1\!\!1\) and hence \(\varphi\in S_N\) by Lemma~\ref{lem:RegularityOfSN}. 
\end{proof}
\begin{lemma}\label{lem:EnoughReebInvFcts}
	With the assumptions and notations above let \(m\in \N\) be arbitrary. We have \(\mathscr{C}_\mathcal{T}^m(X,\R)\setminus\mathscr{C}_\mathcal{T}^\infty(X,\R)\neq \emptyset\).
\end{lemma}
\begin{proof}
	We first show that there exist a function \(g\in \mathscr{C}^\infty_\mathcal{T}(X,\R)\) and a point \(p\in X\) such that \(g(p)=0\) and \(dg_p\neq 0\). Let \(\tilde{p}\in X\) be arbitrary. By Theorem~\ref{thm:HHMSEmbedding} there exist \(j\in\N\), \(f\in V_j\) and \(\tilde{Z}\in T_{\tilde{p}}^{1,0}X\) such that \(\tilde{Z}(f)\neq 0\). Then we can find \(p\in X\) close to \(\tilde{p}\) and \(Z\in T_p^{1,0}X\) such that \(f(p)\neq 0\) and \(Z(f)\neq 0\). Since \(f\) is smooth with \(-i\mathcal{T}f=\lambda_jf_j\) we find that the function \(g\) defined by \(g(x):=|f_j(x)|^2-|f_j(p)|^2\) belongs to   \(\mathscr{C}_\mathcal{T}^\infty(X)\) and satisfies \(g(p)=0\). Furthermore, we find \(Z(g)=Z(f)\overline{f}(p)\neq 0\). It follows that \(d_pg\neq 0\). Now let \(\tau\colon \R\to \R\) be the function defined by \(\tau(t)=t^{m+1}\) for \(t\geq 0\) and \(\tau(t)=0\) for \(t<0\). It follows that \(\tau\circ g\in\mathscr{C}_\mathcal{T}^m(X,\R)\) but \(\tau\circ g\) is not \(\mathscr{C}^\infty\)-smooth in \(p\). 
\end{proof}
\begin{lemma}\label{lem:CinftyIsDense}
	With the assumptions and notations above let \(m\in \N\) be arbitrary. We have \(\mathscr{C}_\mathcal{T}^\infty(X,\R)\) is dense in \(\mathscr{C}_\mathcal{T}^m(X,\R)\) with respect to the \(\mathscr{C}^m(X,\R)\)-norm.
\end{lemma}
\begin{proof}
	From~\cite{Ru95} (see also~\cite{HHL22}) we find that the \(\R\)-action coming from the flow of \(\mathcal{T}\) is induced by a torus Lie group action \(\mu\colon\mathbb{T}^d\times X\to X\). Furthermore, we have that the linear map \(\operatorname{Inv}\colon \mathscr{C}^{0}(X)\to \mathscr{C}^0(X)\), \((\operatorname{Inv}(f))(x):=\int_{\mathbb{T}^d}f(\mu(\theta,x))d\theta\) satisfies \(\operatorname{Inv}\circ\operatorname{Inv}=\operatorname{Inv}\), \(\operatorname{Inv}|_{\mathscr{C}_\mathcal{T}^m(X,\R)}=\operatorname{Id}|_{\mathscr{C}_\mathcal{T}^m(X,\R)}\) and \(\operatorname{Inv}(\mathscr{C}^m(X,\R))\subset \mathscr{C}_\mathcal{T}^m(X,\R)\) for \(m\in\N\cup\{\infty\}\). Furthermore, we have that the restriction \(\operatorname{Inv}|_{\mathscr{C}^m(X)}\colon \mathscr{C}^m(X)\to\mathscr{C}^m(X)\) defines a bounded linear operator between Banach spaces. Now let \(f\in \mathscr{C}_\mathcal{T}^m(X,\R)\) be a function for some \(m\in\N\). There exists a sequence \(\{\tilde{f}_k\}_{k\in\N}\subset \mathscr{C}^\infty(X,\R)\) with \(\|f-\tilde{f}_k\|_{\mathscr{C}^m(X)}\to 0\) for \(k\to\infty\). Put \(f_k:=\operatorname{Inv}(\tilde{f}_k)\in \mathscr{C}_\mathcal{T}^\infty(X,\R)\). We find
	\[\|f-f_k\|_{\mathscr{C}^m(X)}=\|\operatorname{Inv}(f-\tilde{f}_k)\|_{\mathscr{C}^m(X)}\leq C_m\|f-\tilde{f}_k\|_{\mathscr{C}^m(X)}\]
	for some constant \(C_m>0\) independent of \(k\). The claim follows.
\end{proof}

\begin{lemma}\label{lem:NonDensitySN}
	With the assumptions and notations above let \(m\in \N\cup\{\infty\}\) be arbitrary. We have that \(S_N\) is nowhere dense in \(\mathscr{C}_\mathcal{T}^\infty(X)\) with respect to the \(\mathscr{C}^m\)-topology on \(\mathscr{C}_\mathcal{T}^\infty(X)\). 
\end{lemma}
\begin{proof}
	First, assume \(m\in \N\). Let \(\varphi\in S_N\) be arbitrary. By Lemma~\ref{lem:SNisClosed} and Lemma~\ref{lem:RegularityOfSN} we have that \(S_N\) is closed in \(\mathscr{C}^\infty_{\mathcal{T}}(X,\R)\) with respect to the \(\mathscr{C}^{m}\)-topology. Hence,  it is enough to show that for \(\varepsilon \geq 0\) with \(B_{\varepsilon,m}(\varphi)\cap \mathscr{C}_\mathcal{T}^\infty(X,\R)\subset S_N\) we find \(\varepsilon=0\). Here \(B_{\varepsilon,m}(\varphi)\) denotes the open ball in \(\mathscr{C}^m(X)\) with respect to the \(\mathscr{C}^m(X)\)-norm of radius \(\varepsilon\) around \(\varphi\) with \(B_{0,m}(\varphi)=\emptyset\). Assume \(\varepsilon>0\). From Lemma~\ref{lem:EnoughReebInvFcts} we can find \(\psi\in \left(B_{\varepsilon,m}(\varphi)\cap\mathscr{C}_\mathcal{T}^m(X,\R)\right)\setminus\mathscr{C}_\mathcal{T}^\infty(X,\R)\). Since \(\mathscr{C}_\mathcal{T}^\infty(X,\R)\) is dense in \(\mathscr{C}_\mathcal{T}^m(X,\R)\) with respect to the \(\mathscr{C}^m(X,\R)\)-norm (see Lemma~\ref{lem:CinftyIsDense}) we find a sequence \(\{\psi_k\}_{k\in\N}\subset B_{\varepsilon,m}(\varphi)\cap \mathscr{C}_\mathcal{T}^\infty(X,\R)\) with \(\|\psi-\psi_k\|_{\mathscr{C}^m(X)}\to 0\) for \(k\to \infty\). Since \(\{\psi_k\}_{k\in\N}\subset S_N\) we find by Lemma~\ref{lem:SNisClosed} that \(\psi \in \mathscr{C}_\mathcal{T}^\infty(X,\R)\). But \(\psi \notin \mathscr{C}_\mathcal{T}^\infty(X,\R)\) which is a contradiction. As a conclusion we have \(\varepsilon=0\).\\
	Now we consider the case \(m=\infty\). From Lemma~\ref{lem:SNisClosed} we immediately observe that  \(S_N\) is closed in \(\mathscr{C}^\infty_{\mathcal{T}}(X)\) with respect to the \(\mathscr{C}^{\infty}\)-topology. Since \(\{B_{\varepsilon,m}(\varphi)\cap \mathscr{C}^\infty(X)\colon \varepsilon>0,\,m\in\N,\,\varphi\in \mathscr{C}^\infty(X)\}\) is a basis for the \(\mathscr{C}^\infty\)-topology we obtain from the previous considerations that \(S_N\) does not have any interior point.  
\end{proof}

\begin{proof}[\bf{Proof of Theorem~\ref{thm:ForstnericForSasakian}}]
	Let \(\Pi\colon \mathscr{C}^\infty_\mathcal{T}(X,\R)\to \mathscr{C}^\infty_\mathcal{T}(X,\R)/\R\cdot1\!\!1\), \(\varphi\mapsto [\varphi]\) denote the canonical projection onto the quotient space. We have \(\Pi(\mathcal{M}')=\mathcal{M}\). Given \(\varphi\in \mathcal{M}'\) (\(\varphi\in S_N\)) we find \(\varphi+c\in \mathcal{M}'\) (\(\varphi+c\in S_N\)) for all \(c\in \R\) respectively. Hence we have \(\Pi^{-1}(\mathcal{M})=\mathcal{M}'\) which shows that \(\mathcal{M}\) is open in \(\mathscr{C}^\infty_\mathcal{T}(X,\R)/\R\cdot1\!\!1\) since \(\mathcal{M}'\) is open in \(\mathscr{C}^\infty_\mathcal{T}(X,\R)\). Furthermore, since \(\Pi^{-1}(\Pi(S_N))=S_N\) and \(S_N\) is closed in \(\mathscr{C}^\infty_\mathcal{T}(X,\R)\) (see Lemma~\ref{lem:SNisClosed}), we have that \(\Pi(S_N)\cap \mathcal{M}\) is closed in \(\mathcal{M}\). Let \(U\subset \Pi(S_N)\cap \mathcal{M}\) be open in \(\mathcal{M}\). We have that \(\Pi^{-1}(U)\subset S_N\cap \mathcal{M}'\) is open in \(\mathcal{M}'\). Since \(\Pi\) is surjective and \(S_N\cap \mathcal{M}'\) cannot contain any non-empty open set by Lemma~\ref{lem:NonDensitySN}, we find \(U=\emptyset\). As a conclusion we have that \(\Pi(S_N)\cap \mathcal{M}\) is nowhere dense in \(\mathcal{M}\). Since from   Lemma~\ref{lem:ConditionForVarphiInS} we get \(\mathcal{S}\subset \bigcup_{N=1}^\infty \Pi(S_N)\cap \mathcal{M} \)  we conclude that \(\mathcal{S}\) is a set of first category in \(\mathcal{M}\). 
\end{proof}

\subsection{Proof of Theorem~\ref{thm:NoIsolatedPoint}}\label{sec:DeformationWeightedSphere}
In this section we are going to prove Theorem~\ref{thm:NoIsolatedPoint}. Let \((X,T^{1,0}X,\mathcal{T})\) be a compact Sasakian manifold of dimension \(2n+1\), \(n\geq 1\). 
We fix a family of \(\mathscr{C}^m\)-norms on \(X\). As before, for \(m\in\N\cup\{\infty\}\) put \(\mathscr{C}^m_\mathcal{T}(X,\R)=\{\varphi\in \mathscr{C}^m(X,\R)\colon \mathcal{T}\varphi\equiv 0\}\). 
 Let \(f_1,\ldots,f_\nu \in  H_b^0(X)\cap \mathscr{C}^\infty(X)\) and \(\beta_1,\ldots,\beta_\nu\in \R_+\) such that \(-i\mathcal{T}f_j=\beta_jf_j\) and \(f_j\not\equiv 0\) for all \(1\leq j\leq \nu\). We consider the map
\[G\colon \R^{\nu}\times \operatorname{Map}(X,\R)\to \operatorname{Map}(X,\R),\,\, G(r,\varphi)=\sum_{j=1}^\nu r_j^2|f_j|^2e^{2\beta_j\varphi}.\]
	We find that \(G\) restricts to a Fr\'echet differentiable map \(\R^{\nu}\times \mathscr{C}_\mathcal{T}^m(X,\R)\to \mathscr{C}_\mathcal{T}^m(X,\R)\) between Banach spaces for all \(m\in \N\). Here we use the \(\mathscr{C}^m\)-norm on \(\mathscr{C}_\mathcal{T}^m(X,\R)\).
	
	\begin{lemma}\label{lem:ExistenceOfDeformationMap}
		Let \(m\in\N\) be arbitrary. Assume that there exist \(r_0\in \R^\nu\) and \(\varphi_0\in \mathscr{C}^m_\mathcal{T}(X,\R)\) with \(G(r_0,\varphi_0)=1\!\!1\) where \(1\!\! 1\) denotes the constant one function on \(X\). There exist open sets \(U\subset \R^\nu\) around \(r_0\), \(V\subset \mathscr{C}^m_\mathcal{T}(X,\R)\) around \(\varphi_0\) and a Fr\'echet differentiable map \(g\colon U\to V\) such that for  \((r,\varphi)\in U\times V\) we have
		\[G(r,\varphi)=1\!\!1 \Leftrightarrow \varphi=g(r).\]
		\end{lemma}
	
\begin{proof}
	Since \(G\) is Fr\'echet differentiable we can compute \[DG(r,\varphi)(0,\psi)= \psi\sum_{j=1}^\nu 2\beta_jr_j^2|f_j|^2e^{2\beta_j\varphi}.\] Since \(G(r_0,\varphi_0)=1\!\!1\), \(\beta_j>0\), \(1\leq j\leq \nu\) and \(X\) is compact we find that 
	\[x\mapsto\sum_{j=1}^\nu \beta_j(r_0)_j^2|f_j(x)|^2e^{2\beta_j\varphi_0(x)}\] is bounded from below by a positive constant. It follows that  \(\psi\mapsto DG(r_0,\varphi)(0,\psi)\) defines an automorphism of the Banach space \(\mathscr{C}_\mathcal{T}^m(X,\R)\). Then the claim follows from the implicit function theorem.
\end{proof}
\begin{lemma}
	Assume that there exist \(r_0\in \R^\nu\) and \(\varphi_0\in \mathscr{C}^m_\mathcal{T}(X,\R)\) with \(G(r_0,\varphi_0)=1\!\!1\). There exist an open set \(U\subset \R^\nu\) around \(r_0\) and a continuous map \(g\colon U\to \mathcal{C}_\mathcal{T}^\infty(X,\R)\) such that \(g(r_0)=\varphi_0\) and \(G(r,g(r))=1\!\!1\) for all \(r\in U\). Furthermore, we have
	\begin{eqnarray}\label{eq:differentialOfDeformation} 
		\lim_{\varepsilon\to 0}\varepsilon^{-1}\left\|g(r_0+\varepsilon v)-\varphi_0- \varepsilon A(v)\right\|_{\mathscr{C}^2(X)}=0
	\end{eqnarray}
 for any \(v\in \R^\nu\) where
\[A(v):=-\left(\sum_{j=1}^\nu 2\beta_j(r_0)_j^2|f_j|^2e^{2\beta_j\varphi_0}\right)^{-1}\sum_{j=1}^\nu 2(r_0)_jv_j|f_j|^2e^{2\beta_j\varphi_0}.\]
\end{lemma}
\begin{proof}
	From Lemma~\ref{lem:ExistenceOfDeformationMap} we find open neighborhoods \(U\subset \R^\nu\) around \(r_0\), \(V\subset \mathscr{C}^2_\mathcal{T}(X,\R)\) around \(\varphi_0\) and a Fr\'echet differentiable map \(g\colon U\to V\) such that for  \((r,\varphi)\in U\times V\) we have
	\[G(r,\varphi)=1\!\!1 \Leftrightarrow \varphi=g(r).\] 
	Furthermore, we obtain~\eqref{eq:differentialOfDeformation} just by differentiating \(G(r,g(r))=1\!\!1\) with respect to \(r\). 
	Note that from Lemma~\ref{lem:RegularityOfSN} it follows that \(g(U)\subset \mathscr{C}^\infty_\mathcal{T}(X,\R)\). We now show that \(g\colon U\to \mathscr{C}^\infty_\mathcal{T}(X,\R)\) is continuous with respect to any \(\mathscr{C}^m\)-norm.  Let \(m\in\N\), \(m\geq 2\), and \(r'\in U\) be arbitrary. Since \(G(r',g(r'))=1\!\! 1\) we find by Lemma~\ref{lem:ExistenceOfDeformationMap} that there exist open sets \(U'\subset \R^\nu\) around \(r'\), \(V'\subset \mathscr{C}^m_\mathcal{T}(X,\R)\) around \(g(r')\) and a Fr\'echet differentiable map \(g_m\colon U'\to V'\) such that for  \((r,\varphi)\in U'\times V'\) we have
	\[G(r,\varphi)=1\!\!1 \Leftrightarrow \varphi=g_m(r).\]
	We have that \(V\cap \mathscr{C}^\infty_\mathcal{T}(X,\R)\) is open in \(\mathscr{C}^\infty_\mathcal{T}(X,\R)\) with respect to the \(\mathscr{C}^m\)-topology. Hence \(V\cap V'\cap \mathscr{C}^\infty_\mathcal{T}(X,\R)\) is a neighborhood around \(g(r')=g_m(r')\) which is open in the \(\mathscr{C}^m\)-topology. Furthermore, with the same argument as above we have \(g_m(U')\subset \mathscr{C}^\infty_\mathcal{T}(X,\R)\). So \(g_m^{-1}(V\cap V')=g_m^{-1}(V\cap V'\cap \mathscr{C}^\infty_\mathcal{T}(X,\R))\) is an open neighborhood around \(r'\). Then by possibly shrinking \(U'\) we can assume that \(U'\subset U\cap g_m^{-1}(V\cap V')\). We find that \(g_m\colon U'\to V\cap V'\) is continuous with respect to the \(\mathscr{C}^m\)-topology and for \((r,\varphi)\in U'\times (V\cap V')\) we have \(G(r,\varphi)=1\!\!1 \Leftrightarrow \varphi=g_m(r)\). Now given \((r,\varphi)\in U'\times V\cap V'\) we have \((r,\varphi)\in U\times V\) and hence that \(G(r,\varphi)=1\!\!1\) implies \(\varphi=g(r)\). For \(r\in U'\) we conclude \((r,g_m(r))\in U\times V \) with \(G(r,g_m(r))=1\!\! 1\) which implies \(g_m(r)=g(r)\). It follows that \(g\colon U\to \mathscr{C}^\infty_\mathcal{T}(X,\R)\) coincides in a neighborhood of \(r'\in U\) with a map which is continuous with respect to the \(\mathscr{C}^m\)-topology. Since \(m\geq 2\) was arbitrary, we conclude that \(g\) is continuous.   
\end{proof}

\begin{lemma}\label{lem:ModuliOfEntriesCantBeConstantDeformation}
	Let \((X,T^{1,0}X,\mathcal{T})\) be a compact Sasakian manifold of dimension \(2n+1\), \(n\geq 1\), and let \(f_1,\ldots,f_\nu \in  H_b^0(X)\cap \mathscr{C}^\infty(X)\), \(\beta_1,\ldots,\beta_\nu\in \R_+\) such that \(-i\mathcal{T}f_j=\beta_jf_j\)  for all \(1\leq j\leq \nu\). Consider the CR map \(F\colon X\to \C^\nu\), \(F=(f_1,\ldots,f_\nu)\). If \(|f_j|^2\) is constant on \(X\) for any \(1\leq j\leq \nu\) the differential of \(F\) at any point \(p\) cannot be injective.
\end{lemma}
\begin{proof}
	Let \(p\in X\) be arbitrary. For any \(Z\in T_p^{1,0}X\) we find \(0=Z|f_j|^2=(Zf_j)f_j(p)\). If \(f_j(p)\neq0\) we have \(Zf_j=0\). If \(f_j(p)=0\) we have \(f_j\equiv 0\) by assumption and hence \(Zf_j=0\). We conclude \(YF=0\) for all \(Y\in T_p^{1,0}X\oplus T_p^{0,1}X\). It follows that the differential of \(F\) hat \(p\) has at most rank one.
\end{proof}
Recall that
\begin{eqnarray*}
	\mathcal{M}&=&\{\varphi\in \mathscr{C}_\mathcal{T}^\infty(X,\R)\colon \mathcal{V}(\varphi) \text{ is strictly pseudoconvex}\}/\R\cdot1\!\!1\\
	\mathcal{S}_\beta&=&\{[\varphi]\in \mathcal{M}\colon (\mathcal{V}(\varphi),\mathcal{T})\text{ is induced by a \(\beta\)-weighted sphere}\}
\end{eqnarray*}
where \(\beta\in \R_+^N\) for some \(N\in \N\). We will now prove Theorem~\ref{thm:NoIsolatedPoint}.

\begin{proof}[\textbf{Proof of Theorem~\ref{thm:NoIsolatedPoint}}]
	Let \([\varphi_0]\in \mathcal{S}_\beta\) for some \(\beta\in \R^N_+\), \(N\in \N\). By definition there exists a smooth embedding \(F\colon X\to \C^N\) which is CR with respect to the CR structure \(\mathcal{V}(\varphi_0)\) such that \(|F|^2\equiv 1\) and \(F_*\mathcal{T}=\mathcal{T}_\beta\circ F\). From Lemma~\ref{lem:DescriptionOfBetaMaps}  we can assume without loss of generality that \(F=(e^{\beta_1\varphi_0}f_1,\ldots,e^{\beta_\nu\varphi_0} f_{\nu},0,\ldots,0)\) for some \(1\leq \nu\leq N\) and \(f_j\in H_b^0(X)\cap \mathscr{C}^\infty(X)\) satisfy \(-i\mathcal{T}f_j=\beta_jf_j\), \(f_j\not\equiv 0\),  \(1\leq j\leq \nu\). Hence we find \(G(r_0,\varphi_0)=1\!\!1\) with \(r_0=(1,\ldots,1)\). From Lemma~\ref{lem:DescriptionOfBetaMaps} we can find an open neighborhood \(U\subset \R^\nu\) around \(r_0\) and a continuous map \(g\colon U\to \mathscr{C}^\infty_\mathcal{T}(X,\R)\) with \(g(r_0)=\varphi_0\) and \(G(r,g(r))=1\!\!1\) for all \(r\in U\). For \(r\in U\) consider the map \(F_r\colon X\to \C^N\), \(F_r=(r_1e^{\beta_1g(r)}f_1,\ldots,r_\nu e^{\beta_\nu g(r)} f_{\nu},0,\ldots,0)\). Then \(F_r\) is a CR map with respect to the CR structure \(\mathcal{V}(g(r))\) with \(|F_r|^2\equiv 1\) and \((F_r)_*\mathcal{T}=\mathcal{T}_\beta\circ F_r\). Furthermore, from the continuity of \(g\) by possibly shrinking \(U\) we can ensure that \(\mathcal{V}(g(r))\) is strictly pseudoconvex and that \(F_r\) is an embedding for all \(r\in U\). It follows that \(\tau\colon U\to \mathcal{M}\), \(\tau(r)=[g(r)]\) is a well-defined continuous map with \(\tau(r_0)=[\varphi_0]\) and \(\tau(U)\subset \mathcal{S}_\beta\). In order to verify Theorem~\ref{thm:NoIsolatedPoint} it remains to show that \(\tau\) is non-constant.\\
	We define the smooth function \(\rho\colon X\to \R\), \(\rho(x)=\sum_{j=1}^\nu 2\beta_j|f_j(x)|^2e^{2\beta_j\varphi_0(x)}\). Since \(\beta_j>0\) and \(|f_j|^2e^{2\beta_j\varphi_0}\geq 0\) for all \(1\leq j\leq \nu\) with \(\sum_{j=1}^\nu |f_j|^2e^{2\beta_j\varphi_0}=1\!\! 1\) we have \(\rho(x)>0\) for all \(x\in X\). \\
	\textbf{Claim:} There exists \(1\leq j\leq \nu\) such that \(|f_j|^2e^{2\beta_j\varphi_0}/\rho\) is a non-constant function.\\
	Assume that  \(|f_j|^2e^{2\beta_j\varphi_0}/\rho\) is constant for any \(j=\{1,\ldots,\nu\}\). By assumption we have \(|f_j|^2\not\equiv 0\). Then we find positive numbers \(c_2,\ldots,c_\nu>0\) with \(c_j|f_1|^2e^{2\beta_j\varphi_0}=|f_j|^2e^{2\beta_j\varphi_0}\) for \(2\leq j\leq \nu\). Since \(\sum_{j=1}^\nu |f_j|^2e^{2\beta_j\varphi_0}=1\!\! 1\) we find that \(|f_j|^2e^{2\beta_j\varphi_0}\) is constant on \(X\) for any \(1\leq j\leq \nu\). Now we apply Lemma~\ref{lem:ModuliOfEntriesCantBeConstantDeformation} to the Sasakian manifold \((X,\mathcal{V}(\varphi_0),\mathcal{T})\) and the map \(\tilde{F}=(f_1e^{\beta_1\varphi_0},\ldots,f_\nu e^{\beta_\nu\varphi_0})\). It follows  that \(\tilde{F}\) cannot be a CR embedding. But \(F=(\tilde{F},0,\ldots,0)\) was a CR embedding by assumption. This verifies the claim.\\
	So let us assume without loss of generality that \(d:=2|f_1|^2e^{2\beta_1\varphi_0}/\rho\) is not a constant. For sufficiently small \(\varepsilon>0\) we have that the map 
	\[h\colon [0,\varepsilon)\to \mathscr{C}^\infty_\mathcal{T}(X,\R),\,\, h(t)=g(r_0+(t,0,\ldots,0))\] is well defined. Put \(\gamma\colon [0,\varepsilon)\to \mathcal{M}\), \(\gamma(t)=[h(t)]\). Since
	 \[\gamma(t)=[g(r_0+(t,0,\ldots,0)]=\tau(r_0+(t,0,\ldots,0))\] the map is well defined and continuous with \(\gamma([0,\varepsilon))\subset  \mathcal{S}_\beta\). We will show by contradiction that \(\gamma\) is not constant. Let \(p_1,p_2\in X\) be two points with \(d(p_1)\neq d(p_2)\). Assuming that \(\gamma\) is constant we find \(h(t)-\varphi_0\in \R\cdot1\!\!1\) for all \(t\in [0,\varepsilon)\). For \(t>0\) we have
	\begin{eqnarray*}
		|d(p_1)-d(p_2)|&\leq& t^{-1}|h(t)(p_1)-\varphi_0(p_1)+td(p_1)|+t^{-1}|h(t)(p_1)-\varphi_0(p_1)+td(p_2)|\\
		&=&t^{-1}|h(t)(p_1)-\varphi_0(p_1)+td(p_1)|+t^{-1}|h(t)(p_2)-\varphi_0(p_2)+td(p_2)|.
	\end{eqnarray*}
	From~\eqref{eq:differentialOfDeformation} we have \(\lim_{t\to0}t^{-1}|h(t)(p_m)-\varphi_0(p_m)+td(p_m)|=0\), \(m=1,2\). But then \(d(p_1)=d(p_2)\) which is a contradiction. Hence \(\gamma\) cannot be constant.
	    
\end{proof}

\section{The Pseudohermitian Case}\label{sec:ProofOfApproximationGeneral}
\subsection{Functional Calculus on Strictly Pseudoconvex Domains}\label{sec:FunCalcOnSpscDom}
In this section we are going to prove Theorem~\ref{thm:MainThmIntroGeneral} and Theorem~\ref{thm:ContrastToForstneric}. We first consider the following set-up. Let \(M'\) be a complex manifold of dimension \(\dim_\C M'=n+1\) with complex structure \(T^{1,0}M'\) and let \(M\subset\subset M'\) be a relatively compact strictly pseudoconvex domain with \(\mathscr{C}^\infty\)-smooth boundary \(X:=bM\). We have that \((X,T^{1,0}X)\) with \(T^{1,0}X:=\C TX\cap T^{1,0}M'\) is a compact connected strictly pseudoconvex CR manifold which is CR embeddable into the complex Euclidean space (see Lemma~\ref{lem:CharacterizationEmbeddable}). Furthermore, we have that \(\operatorname{Re}(T^{1,0}X)\) defines a contact structure. Let \(\alpha\in \Omega^1(X)\) be a contact form such that the associated Levi form is positive definite and denote by \(\mathcal{T}\) the respective Reeb vector field.   Let \(\rho\colon M'\to \R\) be a smooth defining function for \(M\) 
such that \(\iota^*i(\overline{\partial}\rho-\partial\rho)=\alpha\) where \(\iota\colon X\to M'\) denotes the inclusion map. Let \(J\) with \(J^2=-\operatorname{Id}\) denote the section on \(M'\) with values in \(TM'\otimes T^*M'\) associated to the complex structure \(T^{1,0}M'\). 
Choose a Hermitian metric \(\Theta\) on \(M'\) such that
\begin{eqnarray}
&&T^{1,0}X\perp \C(\mathcal{T}-iJ\mathcal{T})\text{ and }|\mathcal{T}|=1 \text{ on } X,\label{eq:propHermMetric1}\\
&&\Theta(Z,W)=\frac{1}{2i\Lambda_n}d\alpha(Z,\overline{W}) \,\,\forall Z,W\in T^{1,0}_pX,\,p\in X\label{eq:propHermMetric2}	\end{eqnarray}
 where we put the constant \(\Lambda_n=\pi^{\frac{n+1}{n}}\) to ensure that the determinant of the Levi form \(\mathcal{L}^\alpha\) on \(X\) with respect to \(\Theta\) equals \(\pi^{n+1}\).  Denote by \(\nabla\rho\) the gradient of \(\rho\) with respect to the Riemannian metric \(g^{TM'}=\operatorname{Re}\Theta\) on \(M'\) and put \(\hat{T}=J\nabla\rho\). We have that \(\hat{T}\) is a vector field on \(M'\) which does not vanish in a neighborhood of \(X\). 
\begin{lemma}
	With the construction above we have \(\mathcal{T}=\hat{T}\) and \(|\nabla\rho|=|d\rho|=1\) on \(X\).
\end{lemma}
\begin{proof}
For the proof we only work on \(X\). Given \(V\in \operatorname{Re}T^{1,0}X\) we have \(JV\in \operatorname{Re}T^{1,0}X\subset TX\) and hence \(g^{TM'}(J\mathcal{T},V)=-g^{TM'}(\mathcal{T},JV)=0=d\rho(V)\). Also, we find \(g^{TM'}(J\mathcal{T},\mathcal{T})=-g^{TM'}(J\mathcal{T},\mathcal{T})\) and hence  \(g^{TM'}(J\mathcal{T},\mathcal{T})=0=d\rho(\mathcal{T})\) since \(\mathcal{T}\in TX\). Furthermore, we have 
\[ g^{TM'}(J\mathcal{T},J\mathcal{T})=g^{TM'}(\mathcal{T},\mathcal{T})=|\mathcal{T}|^2=1=\alpha(\mathcal{T})=-J^*\alpha(J\mathcal{T})=-d\rho(J\mathcal{T}).\]
From the definition of \(\nabla\rho\) we find \(\nabla \rho=-J\mathcal{T}\) and hence \(\mathcal{T}=J\nabla \rho=\hat{T}\). Moreover, \(|\nabla\rho|^2=-d\rho(J\mathcal{T})=\alpha(\mathcal{T})=1\) which shows \(|\nabla\rho|=1\) and hence \(|d\rho|=1\). 
\end{proof}

  We have that \(\hat{T}\) can be seen as a first order partial differential operator acting on smooth functions  on \(M'\). 
  Let \((\cdot,\cdot)_{M}\) and \((\cdot,\cdot)_{M'}\) denote the \(L^2\)-inner product on \(M\) and \(M'\) respectively with respect to the volume form induced by the Hermitian metric \(\Theta\).  Let \((-i\hat{T})^*\) denote the formal adjoint of \(-i\hat{T}\) with respect to \((\cdot,\cdot)_{M'}\) and put
\begin{eqnarray}
	R=\frac{1}{2}\left( -i\hat{T}+ (-i\hat{T})^* \right)
\end{eqnarray} 
  Since $M$ is strictly pseudoconvex it follows that
the Bergman projection preserves the space of
smooth functions up to the boundary $\cC^\infty(\ol M)$ (see \cite{BS75,Fef74}).
Put
\begin{eqnarray}\label{eq:DefToeplitzOperatorGeneral}
T_R:=BRB: \cC^\infty(\ol M)\To\cC^\infty(\ol M).
\end{eqnarray}
We extend $T_R$ to $L^2(M)$: 
\begin{equation}\label{e-gue230528yydq}\begin{split}
		&T_R:\operatorname{Dom}(T_R)\subset L^2(M)\To L^2(M),\\
		&\operatorname{Dom}(T_R)=\set{u\in L^2(M);\, BRBu\in L^2(M)},
\end{split}\end{equation}
where $BRBu$ is defined in the sense of distributions on $M$. We call \(T_R\) a Toeplitz operator on \(\overline{M}\).  It follows from~\cite{HM23}*{Theorem 3.1} that $T_R$ is self-adjoint.  Given \(\eta\in \mathscr{C}^\infty_c(\R_+)\) put \(\eta_k(t)=\eta(k^{-1}t)\), \(t,k>0\). The functional calculus \(\eta_k(T_R)\) was studied by Hsiao--Marinescu~\cite{HM23}. A consequence of their result applied to the set-up introduced above is the following.
\begin{theorem}[cf. \cite{HM23}*{Theorem 1.1}]\label{thm:HMMainResult}
Let \(M',M,X,\rho,T_R\) as above, with \(\Theta\) satisfying~\eqref{eq:propHermMetric1} and~\eqref{eq:propHermMetric2}, choose \(\eta\in \mathscr{C}_c^\infty(\R_+)\) and put \(\eta_k(t)=\eta(k^{-1}t)\), \(t,k>0\). For any $p\in X$ and any open local coordinate patch $U$ around $p$ in $M'$
	we have 
	\begin{equation}
		\label{e-gue230528ycdsz}
		\eta_k(T_R)(x,y)=\int_0^{+\infty} 
		e^{ikt\Psi(x,y)}b(x,y,t,k)dt+O(k^{-\infty})
		\quad\text{on $(U\times U)\cap(\ol M\times\ol M)$},
	\end{equation}
	where $b(x,y,t,k)\in S^{n+2}_{\mathrm{loc}}
	(1;((U\times U)\cap(\ol M\times\ol M))\times{\mathbb R}_+)$,
	\begin{equation}
		\label{e-gue230528ycdtz}
		\begin{split}
			&b(x,y,t,k)\sim\sum_{j=0}^\infty b_{j}(x,y,t)k^{n+2-j}~
			\mathrm{in}~S^{n+2}_{\mathrm{loc}}(1;((U\times U)
			\cap(\ol M\times\ol M))\times{\mathbb R}_+),\\
			&b_j(x,y,t)\in\mathscr{C}^\infty(((U\times U)\cap(\ol M\times\ol M))
			\times{\mathbb R}_+),~j=0,1,2,\ldots,\\
			&b_{0}(x,x,t)=
			\,\eta(t)t^{n+1}\not\equiv 0,\ \ x\in U\cap X,
		\end{split}
	\end{equation}
	and for some compact interval $I\subset\subset\mathbb R_+$, 
	\begin{equation}
		\begin{split}
			\supp_t b(x,y,t,k),\:\:\supp_t b_j(x,y,t)\subset I,\ \ j=0,1,2,\ldots,
		\end{split}
	\end{equation}
	and 
	\begin{equation}\label{e-gue230528yydzz}
		\begin{split}
			&\Psi(z,w)\in\cC^\infty((U\times U)\cap(\ol M\times\ol M)),\ \ {\rm Im\,}\Psi\geq0,\\
			&\Psi(z,z)=0, \ z\in U\cap X,\\
			&\mbox{${\rm Im\,}\Psi(z,w)>0$ if $(z,w)\notin(U\times U)\cap(X\times X)$},\\
			&d_x\Psi(x,x)=-d_y\Psi(x,x)=-2i\partial\rho(x),\ \ x\in U\cap X,\\
			&\Psi|_{(U\times U)\cap(X\times X)}=\varphi_-, 
		\end{split}
	\end{equation}  
	where $\varphi_-$ is a phase function as in~\cite{HM14}*{Theorem 4.1}.

	Moreover, using the local coordinates $z=(x_1,\ldots,x_{2n-1},\rho)$
	on $M'$ near $p$, where
	$x=(x_1,\ldots,x_{2n-1})$ are local coordinates on $X$ near 
	$p$ with $x(p)=0$, we have 
	\begin{equation}\label{e-gue230319ycdaIm}
		\Psi(z,w)=\Psi(x,y)-i\rho(z)(1+f(z))-
		i\rho(w)(1+\overline{f(w)}\,)+O(\abs{(z,w)}^3)\:\:\text{near $(p,p)$},
	\end{equation}
	where $f$ is smooth near $p$ and $f=O(\abs{z})$. 
\end{theorem}
\begin{remark}
	Note that in Theorem~\ref{thm:HMMainResult} it follows from the proof that the intervall  \(I\) can be taken to be \(\operatorname{supp}\eta\).
\end{remark}

 Let \(0<\lambda_1\leq \lambda_2\leq \cdots\) denote the positive eigenvalues of \(T_R\) counting multiplicities and let \(\{f_j\}_{j\in \N}\in \mathcal{O}^\infty(\overline{M})\) be a respective orthonormal system of eigenfunctions, that is, \(T_Rf_j=\lambda_jf_j\), \(\|f_j\|=1\), \((f_j,f_m)_M=0\) for all \(j,m\in \N\), \(j\neq m\). For \(0<\delta_1<\delta_2\) put \(\hat{N}_k:=\#\{j\in \N|\delta_1k\leq\lambda_j\leq \delta_2k\}\), \(I^-_k=\min\{j\in \N\colon \delta_1k\leq \lambda_j\}\), \(I^+_k=\max\{j\in \N\colon \lambda_j\leq \delta_2k\}\) and choose  \(\chi\in \mathscr{C}^\infty_c(\R_+)\), \(\chi\not\equiv 0\), \(\operatorname{supp}\chi\subset (\delta_1,\delta_2)\). We put \(\chi_k(t)=\chi(k^{-1}t)\), \(t,k>0\), and consider the map
 \begin{eqnarray}\label{eq:DefHkGeneral}
 	\mathcal{H}_k\colon \overline{M}\to \C^{\hat{N}_k},\,\, \mathcal{H}_k=\left(\chi_k(\lambda_{I^-_k})f_{I^-_k},\chi_k(\lambda_{I^-_k+1})f_{I^-_k+1},\ldots, \chi_k(\lambda_{I^+_k})f_{I^+_k}\right)
 \end{eqnarray}
which is holomorphic on \(M\) and smooth up to the boundary for each \(k>0\).
\begin{lemma}\label{lem:DimensionGrowGeneral}
	With the assumptions and notations above, we have \(\hat{N}_k=O(k^{n+1})\)
\end{lemma}
\begin{proof}
	See \cite{HM23}*{Corollary 1.2}.
\end{proof}
Putting \(\eta:=|\chi|^2\) we have
\begin{eqnarray}\label{eq:FundamentalIdentity}
\langle\mathcal{H}_k(z),\mathcal{H}_k(w)\rangle=\sum_{j=1}^\infty |\chi_k(\lambda_j)|^2f_j(z)\overline{f_j(w)}=\eta_k(T_R)(z,w)\text{ on } \overline{M}.
\end{eqnarray}
Hence the following is a direct consequence of Theorem~\ref{thm:HMMainResult}.

\begin{lemma}\label{lem:ConsequenceOfHM}
	 For \(0<\delta_1<\delta_2\) and \(\chi\in \mathscr{C}^\infty_c(\R_+)\), \(\chi\not\equiv 0\), \(\operatorname{supp}\chi\subset (\delta_1,\delta_2)\) let \(\mathcal{H}_k\) defined as in~\eqref{eq:DefHkGeneral}.  Let \(V\subset M'\) be an open neighborhood  around  \(X\) and \(\tilde{c}>0\) such that there is a diffeomorphism  \(\hat{E}\colon X\times (-\tilde{c},\tilde{c})\to V\) with \(\hat{E}(x,0)=\iota(x)\) and \(\hat{E}(X\times[0,c))=V\cap\overline{M}\).
	  Then there exists a \(0<c<\tilde{c}\) such that the following holds. For any \(x_0\in X\) there exist an open neighborhood \(U\subset X\) around \(x_0\) and smooth functions \(\tau\colon U\times[0,c)\to \C\), \(a_0\colon U\times[0,c)\times\R\to \C\) and \(a_1\colon  U\times[0,c)\times\R\times \R\to \C\)
	 such that
	 \begin{eqnarray}\label{eq:ResultOfHM}
	 	|(\mathcal{H}_k\circ E)(x,s)|^2=k^{n+2}\int_\R e^{-kt\tau(x,s)}\left(a_0(x,s,t)+a_1(x,s,t,k)\right)dt+O(k^{-\infty})
	 \end{eqnarray}
	 in \(\mathscr{C}^{\infty}\)-topology and
	\begin{itemize}
		\item [(i)]\(a_0(x,0,t)=t^{n+1}|\chi(t)|^2\) and \(\tau(x,0)=0\) for all \(x\in U\),
		\item[(ii)]\(a_0(x,s,t)=a_1(x,s,t,k)=0\) for all \((x,s,t,k)\in U\times[0,c)\times\R\times\R\) with \( t\notin(\delta_1,\delta_2)\),
		\item[(iii)] \((x,s,t)\mapsto a_1(x,s,t,k)\) is an \(O(k^{-1})\) in \(C^\infty\)-topology.
		\item [(iv)] there exist \(c_1,c_2>0\) with \(c_1s\leq\operatorname{Re}\tau(x,s)\leq c_2s\) for all \((x,s)\in U\times [0,c)\)
		\item [(v)] there exists \(C>0\) with \(1/C\leq\frac{\partial \operatorname{Re}\tau}{\partial s}(x,s)\leq C\) for all \((x,s)\in U\times[0,c)\)
	\end{itemize}
	Furthermore, we have
	 \begin{eqnarray}
		&&d_{(x,s)}\langle(\mathcal{H}_k\circ E)(x,s),(\mathcal{H}_k\circ E)(y,s')\rangle|_{(x,s)=(y,s')}\label{eq:ResultOfHM1}\\
		&&\,\,\,\,\,\,\,\,\,\,\,\,\,\,\,\,\,\,\,\,\,\,\,\,\,\,\,\,\,\,\,\,\,\,\,\,\,\,\,\,\,\,\,\,
		 =k^{n+3}\int_\R e^{-kt\tau(x,s)}\left(\omega_0(x,s,t)+\omega_1(x,s,t,k)\right)dt+O(k^{-\infty})\nonumber\\
		&&d_{(x,s)}\otimes d_{(y,s')}\langle(\mathcal{H}_k\circ E)(x,s),(\mathcal{H}_k\circ E)(y,s')\rangle|_{(x,s)=(y,s')}\label{eq:ResultOfHM2}\\
		&&\,\,\,\,\,\,\,\,\,\,\,\,\,\,\,\,\,\,\,\,\,\,\,\,\,\,\,\,\,\,\,\,\,\,\,\,\,\,\,\,\,\,\,\,
		 =k^{n+4}\int_\R e^{-kt\tau(x,s)}k^{n+3}\left(\zeta_0(x,s,t)+\zeta_1(x,s,t,k)\right)dt+O(k^{-\infty})\nonumber
	\end{eqnarray}
where
\begin{eqnarray*}
	&&\omega_0\in \vbunsec{U\times [0,c)\times \R,p_1^*\C T\overline{M}^*},\,\,\,\omega_1\in \vbunsec{U\times [0,c)\times \R\times \R,p_2^*\C T\overline{M}^*},\\
	&&\zeta_0\in \vbunsec{U\times [0,c)\times \R,p_1^*\C T\overline{M}^*\otimes p_1^*\C T\overline{M}^*},\\
	&&\zeta_1\in \vbunsec{U\times [0,c)\times \R\times \R,p_2^*\C T\overline{M}^*\otimes p_2^*\C T\overline{M}^*}
\end{eqnarray*} are smooth sections such that
\begin{itemize}
	\item [(vi)]  \((x,s,t)\mapsto (\omega_1)_{x,s,t,k}\) and \((x,s,t)\mapsto (\zeta_1)_{x,s,t,k}\) are \(O(k^{-1})\) in \(\mathscr{C}^\infty\)-topology,
	\item[(vii)] \(\omega_0=-2i\partial\rho t^{n+2}|\chi(t)|^2\) and \(\zeta_0=4\partial\rho\otimes \partial\rho t^{n+3}|\chi(t)|^2 \) on \(U\times\{0\}\times \R\). 
\end{itemize}				
Here \(p_1\colon U\times[0,c) \times \R\to  U\times[0,c)\) and \(p_2\colon U\times[0,c) \times \R\times\R\to  U\times[0,c)\)	denote the projections on the first two factors.		
\end{lemma}
\begin{proof}
	For any \(x_0\in X\) and any local coordinates \(x=(x_1,\ldots,x_{2n+1})\) in an open neighborhood \(U\) of \(x_0\) the diffeomorphism \(E\) induces local coordinates around \(\iota(x_0)\in \overline{M}\) which are (after possibly shrinking \(c\)) directly linked to coordinates of the form \(z=(x_1,\ldots,x_{2n+1},\rho)\) in Theorem~\ref{thm:HMMainResult}. Then, with \(\eta=|\chi|^2\), the statement follows immediately from~\eqref{e-gue230528ycdsz} and~\ref{e-gue230319ycdaIm} by shrinking \(U\) where \(c\) has to be chosen depending on the open neighborhood \(U\) around the point \(x_0\). Using that \(X\) is compact we find that \(c\) can be actually chosen independent of \(x_0\).
\end{proof}
\begin{lemma}\label{lem:ConsequenceOfHMGoodE}
 For \(0<\delta_1<\delta_2\) and \(\chi\in \mathscr{C}^\infty_c(\R_+)\), \(\chi\not\equiv 0\), \(\operatorname{supp}\chi\subset (\delta_1,\delta_2)\) let \(\mathcal{H}_k\) be  defined as in~\eqref{eq:DefHkGeneral}. There exist \(V\),  \(\hat{E}\), \(c>0\) as in Lemma~\ref{lem:ConsequenceOfHM} such that for each point \(x_0\) there is an open neighborhood \(U\subset X\) around \(x_0\) where \eqref{eq:ResultOfHM}, \eqref{eq:ResultOfHM1}, \eqref{eq:ResultOfHM2} with  (i)-(vii) hold and so that in addition we have
 \begin{itemize}
 	\item [(viii)] \(\frac{\partial \tau}{\partial s}(x,0)=1\) for all \(x\in X\).
 \end{itemize}
\end{lemma}
\begin{proof}
Since 	the normal bundle of \(X\) in \(M'\) is trivial we can find \(\tilde{c}>0\), an open neighborhood \(\tilde{V}\) of \(X\) in \(M'\) and a diffeomorphism \(\hat{E} \colon X\times (-\tilde{c},\tilde{c})\to V \) such that \(\hat{E}(x,0)=\iota(x)\), \(x\in X\), and \(\hat{E}(X\times[0,\tilde{c}))\subset \overline{M}\). Denote by \(E\) the restriction of \(\hat{E}\) to \(X\times [0,\tilde{c})\). By Lemma~\ref{lem:ConsequenceOfHM} we can choose \(0<c<\tilde{c}\) such that for any \(x_0\) there is an open neighborhood \(U\subset X\) around \(x_0\) where  \eqref{eq:ResultOfHM}, \eqref{eq:ResultOfHM1}, \eqref{eq:ResultOfHM2} with  (i)-(vii)  hold. Let \(x_0\in X\) be arbitrary and choose such an open neighborhood  \(U\subset X\) around \(x_0\). From (v) we find that there exists a constant \(C>0\) such that \(1/C<\frac{\partial \operatorname{Re}\tau}{\partial s}(x,0)<C\) for all \(x\in U\). Since \(x_0\) was arbitrary we find from~\eqref{eq:ResultOfHM} that
	\[f(x):=-\left(\int_\R t^{n+2}|\chi(t)|^2dt\right)^{-1}\lim_{k\to\infty}k^{-n-3}\frac{\partial }{\partial s}|(\mathcal{H}_k\circ E)(x,s)|^2_{s=0}\] defines a smooth real-valued function on \(X\) with \(f(x)=\frac{\partial \tau}{\partial s}(x,0)\) on \(U\). As a consequence \(f\) is positive on \(X\).  By composing \(E\) with  the diffeomorphism \((x,s)\mapsto (x,s/f(x))\) and possibly shrinking \(c\) shows that property (viii) can be satisfied as well. We note since \(X\) is compact we can choose \(c>0\) independent of \(x_0\) and \(U\).   
\end{proof}
In order to handle the integrals appearing in Lemma~\ref{lem:ConsequenceOfHM} the following is quite useful.
\begin{lemma}\label{lem:IntegralLemma}
	Let \(U,V\subset \R^n\) be open sets with \(U\subset \subset V\) and \(c,\delta_1,\delta_2>0\) with \(\delta_1<\delta_2\). Let \(\tau\colon V\times [0,2c)\to \C\), \(a_0\colon V\times [0,2c)\times \R\to \C\) and \(a_1\colon V\times [0,2c)\times \R\times \R\to \C\) be functions with the following properties
	\begin{itemize}
		\item [(1)] \(\tau(x,0)=0\), \(\operatorname{Re}\tau(x,s)\geq 0\) for all \(x\in V\) and \(s\in  [0,2c)\), 
		\item[(2)] \(\tau\), \(a_0\) and \((x,s,t)\mapsto  a_1(x,s,t,k)\) are smooth for each \(k>0\).
		\item[(3)] \((x,s,t)\mapsto a_1(x,s,t,k)\) is an \(O(k^{-1})\) in \(C^\infty\)-topology.
	\end{itemize}
	Put \(A(x,s,t,k)=a_0(x,s,t)+a_1(x,s,t,k)\). For each \(k>0\) consider the smooth function \(B_k\colon V\times [0,2c)\to\C\)
	\begin{eqnarray}\label{eq:DefBkIntegralLemma}
		B_k(x,s)=\int_{\delta_1}^{\delta_2}e^{-kt\tau(x,s)}A(x,s,t,k)dt.
	\end{eqnarray}
	Given \(C>0\) choose \(k_0>0\) such that \(Ck^{-2}\leq c\) for all \(k\geq k_0\). The following holds:
	\begin{itemize}
		\item[(i)]   There is a constant \(C_0>0\) with
		\[\left|B_{k}(x,s)-\int_{\delta_1}^{\delta_2}a_0(x,0,t)dt\right|\leq C_{0}k^{-1}\]
		for all \(k\geq k_0\), \(x\in U\) and \(0\leq s\leq Ck^{-2}\).
		\item [(ii)] For any \(\alpha\in \N^n_0\) there is a constant \(C_\alpha>0\) with
		\[\left|d^\alpha_xB_{k}(x,s)-\int_{\delta_1}^{\delta_2}d^\alpha_xa_0(x,0,t)dt\right|\leq C_{\alpha}k^{-1}\]
		for all \(k\geq k_0\), \(x\in U\) and \(0\leq s\leq Ck^{-2}\).
		\item[(iii)] Given \(C>0\) there exists \( k_0>0\) so that for any \(\alpha\in \N^n_0\) and any \(\ell\in \N_0\) there is a constant \(C_{\alpha,\ell}>0\) with \(\left|d_x^\alpha\left(\frac{\partial}{\partial s}\right)^\ell B_{k}(x,s)\right|\leq C_{\alpha,\ell}k^\ell\) 	for all \(k\geq k_0\), \(x\in U\) and \(0\leq s\leq Ck^{-2}\).
	\end{itemize}
	Now, let \(\{\varphi_k\}_{k\geq k_0}\) be a family of smooth functions \(\varphi_k\colon V\to (0,c)\) such that \(\varphi_k=O(k^{-2})\) in \(\mathscr{C}^\infty\)-topology. Define \(b_k,b_0\colon V\to \R\), \(b_k(x)=B_k(x,\varphi_k(x))\) and \(b_0(x):=\int_{\delta_1}^{\delta_2}a_0(x,0,t)dt\). 
	\begin{itemize}
		\item [(iv)] We have \(b_k=b_0+O(k^{-1})\) in \(\mathscr{C}^\infty\)-topology on \(U\). 
	\end{itemize}
\end{lemma}
\begin{proof}
	\underline{Claim (i):} Since \((x,s,t)\mapsto a_1(x,s,t,k)\) is an \(O(k^{-1})\) and \(|e^{-kt\tau(x,s)|}\leq 1\) we find
	\[\int_{\delta_1}^{\delta_2}e^{-kt\tau(x,s)}a_1(x,s,t,k)dt\leq C'k^{-1}\] for all  \(k\geq k_0\), \(x\in U\), \(0\leq s\leq c\). 
	For \(\tau\in \C\) with \(\operatorname{Re}\tau\geq 0\) we have \(|e^{-kt\tau}-1|\leq kt|\tau|\) and \(|e^{-kt\tau}|\leq 1\) for all \(t,k>0\). Furthermore, since \(\tau(x,0)=0\) and \(\tau,a_0\) are smooth we find a constant \(C''>0\) with \(|\tilde{\tau}(x,s)|\leq C''s\) and \(|a_0(x,s,t)-a_0(x,0,t)|\leq C'' s\) for all \(x\in \tilde{U}\), \(0\leq s\leq c\) and \(t\in[\delta_1,\delta_2]\). Then choosing \(k_0>0\) such that \(Ck^{-2}\leq c\) for all \(k\geq k_0\) the claim follows from
	\begin{eqnarray*}
		\int_{\delta_1}^{\delta_2}e^{-kt\tau(x,s)}a_0(x,s,t)dt&-&\int_{\delta_1}^{\delta_2}a_0(x,0,t)dt\\
		&=&\int_{\delta_1}^{\delta_2}(e^{-kt\tau(x,s)}-1)a_0(x,s,t)dt+\int_{\delta_1}^{\delta_2}a_0(x,s,t)-a_1(x,0,t)dt.
	\end{eqnarray*}
	\underline{Claim (ii):}
	Since  \(\tau(x,0)=0\) for all \(x\in V\) we find for any \(\alpha\in\N_0^{2n+1}\), \(|\alpha|\geq 1\) that there is a constant \(C'_\alpha>0\) such that 
	\[\max\{|d_x^\alpha \tau(x,s)|,|d_x^\alpha (a_0(x,s,t)-a_0(x,0,t))|\}\leq C'_\alpha s\]
	for all \(x\in U\), \(s\in [0,c]\) and \(t\in [\delta_1,\delta_2]\) which leads to
	\[\max\{|d_x^\alpha \tau(x,s)|,|d_x^\alpha  (a_0(x,s,t)-a_0(x,0,t))|\} \leq C'_\alpha Ck^{-2}\]
	for all \(k\geq k_0\), \(x\in U\), \(0\leq s\leq Ck^{-2}\) and \(t\in [\delta_1,\delta_2]\). Then, applying derivatives in \(x\) to~\eqref{eq:DefBkIntegralLemma}, we find a constant \(C''_\alpha>0\) such that 
	\[\left|d^\alpha_xB_{k}(x,s)-\int_{\delta_1}^{\delta_2}e^{-kt\tau(x,t)}d^\alpha_xa_0(x,s,t)dt\right| \leq C''_{\alpha}k^{-1}\]
	holds for all \(k\geq k_0\), \(x\in U\), \(0\leq s\leq Ck^{-2}\).  By applying (i) to \(\int_{\delta_1}^{\delta_2}e^{-kt\tau(x,t)}d^\alpha_xa_0(x,s,t)dt\)   the claim follows.\\
	\underline{Claim (iii):}
	We first observe that 
	\[(-1)^\ell k^{-\ell}\left(\frac{\partial}{\partial s}\right)^\ell B_{k}(x,s)=\int_{\delta_1}^{\delta_2}e^{-kt\tau(x,s)}(\tilde{a}_0(x,s,t)+\tilde{a}_1(x,s,t,k)dt\]
	for some smooth functions \(\tilde{a}_0,\tilde{a}_1\) satisfying the same properties as \(a_0\) and \(a_1\). Then the claim follows from (ii).\\
	\underline{Claim (iv):}	
	Given \(k>0\), \(\alpha\in \N_0^{n}\) and \(\ell\in \N_0\) define
	\[R_{k,\alpha,\ell}\colon V \to \C,\,\,\,R_{k,\alpha,\ell}(x)=\left(d_x^\alpha\left(\frac{\partial}{\partial s}\right)^\ell B_{k}(x,s)\right)|_{s=\varphi_k(x)}.\]
	We have \(0<\varphi_k\leq C'k^{-2}\) for some constant \(C'>0\) independent of \(k\).  Choose \(k_0>0\) such that \(C'k^{-2}\leq c\) for all \(k\geq k_0\). 
	For \(m\in \N\) and \(\beta\in \N_0^{2n+1}\) we put
	\begin{eqnarray*}
		I(m,\beta)&:=&\left\{(\tau(1),\ldots,\tau(m))\in(\N_0^{2n+1})^m\colon \sum_{j=1}^m\tau(j)=\beta,\,\,\,|\tau(j)|\geq 1\,\forall 1\leq j\leq m\right\}.
	\end{eqnarray*}
	Applying the chain rule to \(b_k\) we find
	\[d_x^\alpha b_k= R_{k,\alpha,0}+\sum_{\gamma\leq \alpha,|\gamma|<|\alpha|}\sum_{m=1}^{|\alpha|-|\gamma|}R_{k,\gamma,m}\sum_{\tau\in I(m,\alpha-\gamma)}c_{\gamma,m,\tau}\prod_{\nu=1}^{m}d_x^{\tau(\nu)}\varphi_k \]
	with integer coefficients \(c_{\gamma,m,\tau}\) independent of \(k\). Since \(|\tau(\nu)|\geq 1\) one has \(\prod_{\nu=1}^{m}d_x^{\tau(\nu)}\varphi_k=O(k^{-2m})\) by the assumptions on \(\varphi_k\). Furthermore,  we have \(R_{k,\alpha,0}-d_x^\alpha b_0(x)=O(k^{-1})\) by (ii) and  \(R_{k,\gamma,m}=O(k^{m})\) by (iii). The claim follows.  
	
\end{proof}
For \(0<\delta_1<\delta_2\) and \(\chi\in \mathscr{C}^\infty_c(\R_+)\), \(\chi\not\equiv 0\), \(\operatorname{supp}\chi\subset (\delta_1,\delta_2)\) let \(\mathcal{H}_k\) defined as in~\eqref{eq:DefHkGeneral}.
Put \[C^{(j)}_\chi:=\int_{\R}|\chi(t)|^2t^{n+1+j}dt, \,\,\, j=0,1,\ldots.\]
\begin{lemma}\label{lem:EnsuringFkGreaterOne}
	Let \(\mathcal{G}\colon \overline{M}\to \C^N\) be a  smooth map which is holomorphic holomorphic on \(M\). With the assumptions and notations above there exist constants \(k_0,d_0>0\) such that 
	\[\frac{k^{-n-2}}{C^{(0)}_\chi}|(\mathcal{H}_k\circ E)(x,0)|^2+e^{-2k}|(\mathcal{G}\circ E)(x,0)|^2> 1-d/k\]
	 holds for all \(k\geq k_0\) and \(d\geq d_0\), \(x\in X\).	
\end{lemma}
\begin{proof}
	We have \(e^{-2k}|(\mathcal{G}\circ E)(x,0)|^2=O(k^{-\infty})\). With \eqref{eq:ResultOfHM} we obtain \(\frac{k^{-n-2}}{C^{(0)}_\chi}|(\mathcal{H}_k\circ E)(x,0)|^2=1+O(k^{-1})\) from (i) in Lemma~\ref{lem:ConsequenceOfHM}.
\end{proof}
Given a smooth map \(\mathcal{G}\colon \overline{M}\to \C^N\) which is holomorphic on \(M\),   consider the map 
\begin{eqnarray}\label{eq:DefMapFkGeneral}
	\mathcal{F}_k\colon\overline{M}\to \C^{N+N_k},\,\, \mathcal{F}_k=\sqrt{\frac{k}{k-d}} \left(e^{-k}\mathcal{G},\frac{1}{\sqrt{k^{n+1}C^{(0)}_\chi}}\mathcal{H}_k\right).
\end{eqnarray}
for \(d\geq d_0\) where \(k\geq k_0\) is sufficiently large and \(k_0,d_0\) are chosen regarding Lemma~\ref{lem:EnsuringFkGreaterOne}.
 We now fix diffeomorphisms \(\hat{E}, E\) and a constant \(c>0\) as given by Lemma~\ref{lem:ConsequenceOfHMGoodE}. 
 Then the most important object of this section is the function given by
\begin{eqnarray}\label{eq:DefFunctionhk}
	h_k\colon X\times[0,c)\to \R_{\geq0}, \,\, h_k(x,s)=|(\mathcal{F}_k\circ E)(x,s)|^2.
\end{eqnarray}
From the definition it is clear that \((x,s,k)\mapsto h_k(x,s)\) defines a smooth function on \(X\times[0,c)\times [k_1,\infty)\) where \(k_1\geq k_0\) has to be chosen large when \(d\geq d_0\) is large. 
We will now prove a proposition verifying important properties of \(h_k\).

\begin{proposition}\label{pro:PropertiesFkGeneral}
	With the assumptions and notations above we have:
	\begin{itemize}
		\item[(i)] We have that \(x\mapsto h_k(x,0)=1+O(k^{-1})\) in \(\mathscr{C}^\infty\)-topology on \(X\)
		\item[(ii)] For any \(C>0\) there exist constants \(C_0,k_0>0\) such that for any \(k\geq k_0\), \(x\in X\) and \(0\leq s\leq Ck^{-1}\) with \(h_k(x,s)=1\) we have \(s\leq C_0k^{-2}\).
		\item[(iii)] For any \(C>0\) there exist constants \(A,k_0>0\) such that for any \(k\geq k_0\), \(x\in X\) and \(s\in \R\) with \(|s|\leq Ck^{-1}\) we have
		\[ 	\frac{\partial}{\partial s}h_k(x,s)\leq -k/A.\]
		\item[(iv)] For any \(C>0\) there exists a constants \(k_0>0\) such that for any \(k\geq k_0\) and \(x\in X\) the function \(s\mapsto h_k(x,s)\) is strictly decreasing on \([0,Ck^{-1}]\) with \(h_k(x,0)>1\) and \(h_k(x,Ck^{-1})<1\).
		\item[(v)] For \(k>0\), \(s\in [0,c)\) and \(\ell\in \N_0\) consider the smooth function
		\[r_{k,s,\ell}\colon X\to \R,\,\,\,\,r_{k,s,\ell}(x)=k^{-\ell}\left(\frac{\partial}{\partial s}\right)^\ell h_k(x,s)- b_\ell\]   with \(b_\ell=(-1)^\ell C^{(\ell)}_\chi/C^{(0)}_\chi\). Fix \(C>0\). There exists \(k_0>0\) such that for any \(\ell\in \N_0\) and \(m\in\N_0\) there is a constant \(C_{\ell,m}>0\) such that for all \(k\geq k_0\) and \(|s|\leq Ck^{-2}\) one has \(\|r_{k,s,\ell}\|_{\mathscr{C}^m(X)}\leq C_{m,\ell}k^{-1}\).
		\item[(vi)] For any \(C>0\) there exist constants \(k_0,d>0\) with \(d\geq d_0\) and \(d_0\) is as in Lemma~\ref{lem:EnsuringFkGreaterOne} such that for any \(k\geq k_0\), \(x\in X\) and  \(0\leq s\leq Ck^{-2}\) we have for \(h_k\) in~\eqref{eq:DefFunctionhk} defined with respect to this \(d\) that
		\[ 	\frac{\partial}{\partial k}h_k(x,s)\leq -k^{-2}.\]
	\end{itemize}
	
\end{proposition}

\begin{proof}
	\underline{Claim (i):} The claim follows immediately from Lemma~\ref{lem:ConsequenceOfHM} using that use that \(x\mapsto\frac{k}{k-d}e^{-2k}|(\mathcal{G}\circ E)(x,0)|^2=O(k^{-\infty})\) in \(\mathscr{C}^\infty\)-topology on \(X\) and \(k/(k-d)=1+O(k^{-1})\).\\
	\underline{Claim (ii):} Let \(x_0\in X\) be a point. We can find an open neighborhood \(U\) around \(x_0\) such that on \(U\times[0,c)\) we have 
	\begin{eqnarray}\label{eq:h_kInCoordinatesProposition}
		h_k(x,s)=\frac{1}{(1-d/k)C^{(0)}_\chi}\int_\R e^{-kt\tau(x,s)}A(x,s,k,t)dt+O(k^{-\infty}) 
	\end{eqnarray}
	in \(C^{\infty}\)-topology where \(A=a_0+a_1\) and \(\tau, a_0,a_1\) are as in Lemma~\ref{lem:ConsequenceOfHM}.  Here we use that \(\frac{k}{k-d}e^{-2k}|(\mathcal{G}\circ E)(x,s)|^2=O(k^{-\infty})\) in \(\mathscr{C}^\infty\)-topology on \(X\times [0,c)\).  Let \(U_0\subset\subset U\) be a smaller neighborhood around \(x_0\). Since \(a_0(x,0,t)= t^{n+1}|\chi(t)|^2\) we find  a constant \(B_1>0\) with \(|a_0(x,s,t)-t^{n+1}|\chi(t)|^2|\leq B_1s\) for \((x,s)\in U_0\times [0,c)\) and \(t\in \R\). Since \((1-d/k)^{-1}=1+O(k^{-1})\) we conclude \(h_k(x,s)-(C^{(0)}_\chi)^{-1}\int_{\delta_1}^{\delta_2}e^{-kt\tau(x,s)}t^{n+1}|\chi(t)|^2dt=O(k^{-1})\) for \(0\leq s\leq Ck^{-1}\). Hence there is a constant \(B_2>0\) such that \(h_k(x,s)\leq (C^{(0)}_\chi)^{-1}\left|\int_{\delta_1}^{\delta_2}e^{-kt\tau(x,s)}t^{n+1}|\chi(t)|^2dt\right| +B_2k^{-1}\) for all \(k\geq k_0\) and \(0\leq s\leq Ck^{-1}\). Since \(\operatorname{Re}\tau(x,s)\geq c_1s\) for \((x,s)\in U_0\times [0,c)\) we conclude  
	\begin{eqnarray*}
		h_k(x,s)\leq (C^{(0)}_\chi)^{-1}\left|\int_{\delta_1}^{\delta_2}e^{-kt\tau(x,s)}t^{n+1}|\chi(t)|^2dt\right| +B_2k^{-1}\leq e^{-k\delta_1c_1s} +B_2k^{-1}
	\end{eqnarray*}
	for all \(k\geq k_0\) and \(x\in U_0\) and \(0\leq s\leq Ck^{-1} \). Assuming \(h_k(x,s)=1\) with \(k\geq k_0\), \(0\leq s\leq Ck^{-1}\) and \(x\in U_0\) we find \(1\leq e^{-k\delta_1 c_1 s}+B_2k^{-1}\) and hence
	\[s\leq -\frac{1}{k\delta_1c_1}\log(1-B_2k^{-1})\leq \frac{1}{k\delta_1c_1}\log\left(1+\frac{B_2}{k-B_2}\right)\leq \frac{B_2}{k\delta_1c_1(k-B_2)}.\]
	Since \(X\) is compact the claim follows.\\
	\underline{Claim (iii):} Let \(x_0\in X\), \(U_0\subset \subset U\subset X\) be as above. We will compute the derivative of~\eqref{eq:h_kInCoordinatesProposition} with respect to \(s\). From Lemma~\ref{lem:ConsequenceOfHM} we find a constant \(B_1>0\) such that \(\frac{\partial \operatorname{Re}\tau}{\partial s}(x,s)>B_1\) for all \(x\in U\) and \(0\leq s<c\).  From~\eqref{eq:h_kInCoordinatesProposition} we find
	\[\frac{\partial h_k}{\partial s}(x,s)=-k\frac{\partial \tau}{\partial s}(x,s) \int_{\delta_1}^{\delta_2}e^{-kt\tau(x,s)}\frac{a_0(x,s,t)}{C_\chi^{(0)}}tdt+O(1).\] 
	Since \(a_0(x,0,t)= t^{n+1}|\chi(t)|^2\) we find  a constant \(B_2>0\) with
	 \[|\operatorname{Re}a_0(x,s,t)-t^{n+1}|\chi(t)|^2|  \leq B_2s, |\operatorname{Im}a_0(x,s,t)| \leq B_2s.\]
	   Since \(\frac{\partial \tau}{\partial s}(x,0)=1\)  and \(\tau(x,0)=0\) we find a constant \(B_3>0\) such that for all \(k\geq 1\) we have \(|\operatorname{Im}\tau(x,s)|\leq B_3k^{-2}\) and \(|\frac{\partial \operatorname{Im}\tau}{\partial s}(x,s)|\leq B_3k^{-1}\) for all \(x\in U_0\), \(0\leq s\leq Ck^{-1}\). Using \(|1-e^{-kti\operatorname{Im}\tau}|\leq kt|\operatorname{Im}\tau|\) and \(|e^{-kt\tau}|\leq 1\) we conclude
	  \begin{eqnarray*}
	  \frac{\partial h_k}{\partial s}(x,s)&=&-k\frac{\partial \tau}{\partial s}(x,s) \int_{\delta_1}^{\delta_2}e^{-kt\operatorname{Re}\tau(x,s)}\operatorname{Re}a_0(x,s,t)tdt+O(1)\\
	  &=&-k\frac{\partial \tau}{\partial s}(x,s) \int_{\delta_1}^{\delta_2}e^{-kt\operatorname{Re}\tau(x,s)}t^{n+2}|\chi(t)|^2dt+O(1)\\
	  \end{eqnarray*}
  Because \(\chi\not\equiv 0\) we  then find \(B_4>0\) and \(\delta_1\leq\delta_1'< \delta_2'\leq\delta_2\) such that for all sufficiently large \(k>1\) we have \(t^{n+2}|\chi(t)|^2>B_4\) for all \(t\in [\delta_1',\delta_2']\)  .
	Since \(\frac{\partial h_k}{\partial s}(x,s)\in \R\) and \(t^{n+2}|\chi(t)|^2\geq 0\) we conclude that for some \(B_5>0\) there exists \(k_0>0\) such that
	\[\frac{\partial h_k}{\partial s}(x,s)-B_5\leq -kB_1B_4\int_{\delta_1'}^{\delta_2'}e^{-c_1Ct}dt\leq -kB_1B_4e^{-c_1C\delta_2'}(\delta_2'-\delta_1')\]
	holds for all \(x\in U_0\), \(0\leq s\leq Ck^{-1}\) and \(k\geq k_0\).  Since \(X\) is compact the claim follows.\\
	\underline{Claim (iv):} The first part of the claim follows immediately from Claim~(iii). Furthermore, from the choice of the constant \(d>0\) in Lemma~\ref{lem:EnsuringFkGreaterOne} we find \(k_1>0\) such that \(h_k(x,0)>1\) holds  for all \(x\in X\) and all \(k\geq k_1\). Using Claim~(iii) we find \(k_2\geq k_1\) with
	\[h_k(x,Ck^{-1})-h_k(x,0)=\int_{0}^{Ck^{-1}}\frac{\partial h_k}{\partial s}(x,s)ds\leq-kBCk^{-1}=-CB\]
	for all \(x\in X\) and \(k\geq k_2\) where \(B>0\) is a constant independent of \(x\) and \(k\). From Claim~(i) we conclude
	\(x\mapsto h_k(x,Ck^{-1})\leq1-CB+O(k^{-1})\) on \(X\) and the claim follows.\\
	\underline{Claim (v):} Let \(x_0\in X\), \(U_0\subset \subset U\subset X\) be as above. From Lemma~\ref{lem:EnsuringFkGreaterOne} we find that there exists \(k_0>0\) such that  \(Ck^{-2}<c/2\) and
	\((x,s)\mapsto r_{k,s,\ell}(x)\) is well defined on \(U\times [0,c)\) for all \(\ell\in \N_0\) and \(k\geq k_0\).  From~\eqref{eq:h_kInCoordinatesProposition} we find
	\[(x,s)\mapsto r_{k,s,\ell}(x)+ b_\ell= \int_\R e^{-kt\tau(x,s)}A^{(\ell)}(x,s,t,k)dt+O(k^{-\infty})\]     
	in \(C^{\infty}\)-topology on \(U\times [0,c)\) with \(A^{(\ell)}(x,s,t,k)=a^{(\ell)}_0(x,s,t)+a^{(\ell)}_1(x,s,t,k)\) where \(a^{(\ell)}_0,a^{(\ell)}_1\) satisfy the assumptions in Lemma~\ref{lem:IntegralLemma} and
	\[a^{(\ell)}_0(x,s,t)=\frac{(-1)^\ell}{C^{(0)}_\chi}\left(\frac{\partial \tau}{\partial s}(x,s)\right)^\ell t^\ell a_0(x,s,t).\] 
	Since \(\frac{\partial \tau}{\partial s}(x,0)=1\) and \(a_0(x,0,t)=t^{n+1}|\chi(t)|^2\) for \(x\in U\) we find \[a^{(\ell)}_0(x,0,t)=(-1)^\ell(C^{(0)}_\chi)^{-1}t^{n+1+\ell}|\chi(t)|^2\]
	 is independent of \(x\). Observing that \(\int_\R a^{(\ell)}_0(x,0,t)dt=b_\ell\) the claim follows from (i), (ii) in Lemma~\ref{lem:IntegralLemma} and the compactness of \(X\).\\
	\underline{Claim (vi):} For fixed \(d\geq d_0\) in the definition of \(\mathcal{F}_k\) in~\eqref{eq:DefMapFkGeneral} we can choose \(k(d)>0\) such that \(h_k\) is well defined an smooth in \(k\) for all \(k\geq k(d)\).   A direct calculation shows
	\begin{eqnarray}
		\frac{\partial}{\partial k}h_k(x,s)&=&\frac{1}{(1-d/k)C_\chi^{(0)}k^{n+2}}\left(\frac{\partial}{\partial k}|\mathcal{H}_k(E(x,s))|^2-\frac{(n+2)}{k}|\mathcal{H}_k(E(x,s))|^2\right)\nonumber\\
		&&-\frac{d}{(1-d/k)k^2}h_k(x,s)+O(k^{-\infty}).\label{eq:DerivativeOfhk}
	\end{eqnarray}
Using  \(|\mathcal{H}_k(z)|^2=\sum_{j=1}^\infty\chi(k^{-1}\lambda_j)|f_j(z)|^2=\eta_k(T_R)(z,z)\) (see \eqref{eq:FundamentalIdentity}) where \(\eta(t)=|\chi(t)|^2\) we find
\[\frac{\partial}{\partial k}|\mathcal{H}_k(z)|^2=-\frac{1}{k}\sum_{j=1}^\infty k^{-1}\lambda_j2(\operatorname{Re}\overline{\chi}\chi')(k^{-1}\lambda_j)|f_j(z)|^2=-k^{-1}\tilde{\eta}_k(T_R)(z,z)\]
with \(\tilde{\eta}(t)=2t\operatorname{Re}\overline{\chi(t)}\chi'(t)\)
 it follows from Theorem~\ref{thm:HMMainResult} (see also Lemma~\ref{lem:ConsequenceOfHM} and Lemma~\ref{lem:IntegralLemma}) that there are constants \(C_1,k_1>0\) such that 
\[\max\left\{\left|\eta_k(T_R)(z,z)- k^{n+2}\int_\R t^{n+1} |\chi(t)|^2dt\right|,\left|\tilde{\eta}_k(T_R)(z,z)- k^{n+2}\int_\R t^{n+1}\tilde{\eta}(t)dt\right|\right\}\leq C_1k^{n+1}\]
for all \(k\geq k_1\), \(0\leq s\leq Ck^{-2}\) and \(x\in X\) with \(z=E(x,s)\). Since \[\int_\R t^{n+1}\tilde{\eta}(t)dt=-(n+2)\int_\R t^{n+1}|\chi(t)|^2dt\] it follows from~\eqref{eq:DerivativeOfhk} that there are constants \(C_2,k_2\) such that
\[\frac{\partial}{\partial k}h_k(x,s)\leq \frac{1}{(1-d/k)k^2}(-h_k(x,s)\cdot d+C_2)\]
holds for all \(k\geq k_1\), \(0\leq s\leq Ck^{-2}\) and \(x\in X\). We note that from the construction it follows that \(C_2\) is independent of \(d\). Then,   choosing \(d\geq d_0\)
 sufficiently large and using (v) with \(\ell=m=0\) proves the claim.
\end{proof}

\begin{theorem}\label{thm:ConstructionPhiGeneral}
	Let \(X\) be as above and \(h_k\) be the function defined in~\eqref{eq:DefFunctionhk} where the constant \(d\) (see \eqref{eq:DefMapFkGeneral}) is chosen large enough so that~\ref{pro:PropertiesFkGeneral}~(vi) holds. There exists \(k_0,C_0,\tilde{C}>0\) such that for any \(k\geq k_0\) there is a smooth function \(\varphi_k\colon X\to (0,C_0k^{-1})\) uniquely defined by the equation \(h_k(x,\varphi_k(x))=1\), \(x\in X\). Furthermore, for any \(m\in \N_0\) there exists  a constant \(C_m>0\) such that for all \(k\geq k_0\) we have \(\|\varphi_k\|_{\mathscr{C}^m(X)}\leq C_mk^{-2}\). 
	 In addition, we have that the function \((x,k)\mapsto \varphi_k(x)\) is smooth on \(X\times[k_0,\infty)\) with \(\frac{\partial\varphi_{k}(x)}{\partial k}\leq -\tilde{C}k^{-3}\) for all \(k\geq k_0\) and \(x\in X\). 
\end{theorem}
\begin{proof}
	Let \(C_0>0\)  be arbitrary. From Proposition~\ref{pro:PropertiesFkGeneral}~(iv) we find that there is \(k_1>0\) such that for any \(x\in X\) and \(k\geq k_1\) there is unique solution \(s=s(x,k)\) for equation \(h_k(x,s)=1\), \(0<s<C_0k^{-1}\). It follows from Proposition~\ref{pro:PropertiesFkGeneral}~(ii) that there exist \(C_2,k_2>0\), \(k_2\geq k_1\) with \(0<s(x,k)\leq C_2k^{-2}\) for all \(x\in X\) and \(k\geq k_2\). Then, for some \(k_3>k_2\), we have that for any \(k\geq k_3\) the function \(\varphi_k\colon X\to (0,C_0k^{-1})\), \(\varphi_k(x)=s(x,k)\) is well defined. Using Proposition~\ref{pro:PropertiesFkGeneral}~(iii) we find \(C_4>0\) and \(k_4>k_3\) such that
	\[ 	\frac{\partial}{\partial s}h_k(x,s)\leq -k/C_4\]
	holds for all \(k\geq k_4\) and \(0\leq s\leq C_0k^{-1}\). Since for fixed \(k\geq k_4\) we have that \(s=\varphi_k(x)\) is the unique solution of \(h_k(x,s)=1\) and we have by construction that 
	\[\frac{\partial}{\partial s}h_k(x,s)|_{s=\varphi_k(x)}\leq-\frac{k}{C_4}<0\]
	it follows from the implicit function theorem that \(\varphi_k\colon X\to (0,C_0k^{-1})\) is smooth. Furthermore, since \((x,s,k)\mapsto h_k(x,s)\) is smooth it also follows that \((x,k)\mapsto \varphi_k(x)\) is smooth with
	\[\frac{\partial \varphi_k(x_0)}{\partial k}=-\frac{\frac{\partial }{\partial k}h_k(x,s)|_{(x,s)=(x_0,\varphi_k(x_0))}}{\frac{\partial}{\partial s}h_k(x_0,s)|_{s=\varphi_k(x_0)}}.\]
	Since \(\|\varphi_k\|_{\mathscr{C}^0(X)}\leq C_2k^{-2}\) by construction it follows from  Proposition~\ref{pro:PropertiesFkGeneral}~(iii),(vi) that there is a constant \(\tilde{C}>0\) and \(k_5\geq k_4\) such that \(\frac{\partial\varphi_{k}(x)}{\partial k}\leq -\tilde{C}k^{-3}\) for all \(k\geq k_5\) and \(x\in X\).
	It remains to show that for any \(m\in\N\) we have \(\|\varphi_k\|_{\mathscr{C}^m(X)}=O(k^{-2})\) for \(k\to \infty\). From our construction it follows that \(\|\varphi_k\|_{\mathscr{C}^0(X)}\leq C_2k^{-2}\). Furthermore, from the implicit function theorem we obtain for any \(x_0\in X\) one has
	\[d_x\varphi_k|_{x=x_0}=-\frac{d_xh_k(x,s)|_{(x,s)=(x_0,\varphi_k(x_0))}}{\frac{\partial}{\partial s}h_k(x_0,s)|_{s=\varphi_k(x_0)}}.\]
	
	Let \(p\in X\) be a point and let \((U,x_1,\ldots,x_{2n+1})\) be local coordinates around \(p\). For \(\alpha\in \N_0^{{2n+1}}\) put \(|\alpha|=\sum_{j=1}^{2n+1}\alpha_j\) and \(d^{\alpha}_x=\left(\frac{\partial}{\partial x_1}\right)^{\alpha_1}\ldots\left(\frac{\partial}{\partial x_{2n+1}}\right)^{\alpha_{2n+1}}\).
	We need to show that  for any \(\alpha\in\N_0^{2n+1}\) and any \(1\leq j\leq 2n+1\) we have
	\begin{eqnarray}\label{eq:estimatesVarphiGeneral}
		d_x^\alpha \frac{\partial}{\partial x_j}\varphi_k=O(k^{-2})
	\end{eqnarray}
	on \(U\). Using Proposition~\ref{pro:PropertiesFkGeneral}~(v) we obtain~\eqref{eq:estimatesVarphiGeneral} by repeating the induction argument for~\eqref{eq:induction1} in the proof of Theorem~\ref{thm:ConstructionOfVarphik}.

\end{proof}

\subsection{Proof of Theorem~\ref{thm:ContrastToForstneric}}
In this section we are going to prove Theorem~\ref{thm:ContrastToForstneric}. The result will be deduced from the following theorem. 
\begin{theorem}\label{thm:ProofOnExhaustionBySphericalDom}
	Let \(M'\) be a Stein manifold of dimension \(\dim_\C M'=n+1\). Let \(M\subset \subset M'\) be a relatively compact strictly pseudoconvex domain with smooth boundary. There is \(k_0>0\) so that for each  \(k\geq k_0\) there is \(N_k\in \N\), a smooth embedding \(\mathcal{F}_k\colon \overline{M}\to \C^{N_k}\) which is holomorphic on \(M\) and a strictly pseudoconvex  domain \(D_k\subset\subset M\) with real analytic boundary satisfying \(\mathcal{F}_k(D_k)\subset \mathbb{B}^{N_k}\) and \(\mathcal{F}_k(bD_k)\subset \mathbb{S}^{2N_k-1}\) such that  
	\(\overline{D_k}\to \overline{M}\) for \(k\to \infty\) in the \(\mathscr{C}^\infty\)-topology for domains. In particular,  one has \(\bigcup_{k\in \N, k\geq k_0}D_k=M\). Moreover, the family of domains \(\{D_k\}_{k\geq k_0}\) can be chosen to depend smoothly on \(k\) with \(D_{k'}\subset\subset D_{k''}\) for all \(k_0\leq k'<k''\).
\end{theorem}
\begin{proof}
 Since \(M'\) is Stein we can find a holomorphic embedding \(G\colon M'\to \C^N\) for some \(N\in \N\). Then define \(\mathcal{G}\) in~\eqref{eq:DefMapFkGeneral} to be \(\mathcal{G}=G|_{\overline{M}}\). It follows that for any sufficiently large \(k>0\) the map \(\mathcal{F}_k\colon \overline{M}\to \C^{\hat{N}_k+N}\) (see~\eqref{eq:DefMapFkGeneral}) is a smooth embedding which is holomorphic on \(M\). We note that we choose the constant \(d\) in the definition of \(\mathcal{F}_k\) (see \eqref{eq:DefMapFkGeneral}) large enough so that~\ref{pro:PropertiesFkGeneral}~(vi) holds. Put \(N_k:=\hat{N}_k+N\). Note that \(N_k=O(k^{n+1})\) from Corollary~\ref{lem:DimensionGrowGeneral}. Let \(\{\varphi_k\}_{k\in k_0}\) be the family of smooth functions from Theorem~\ref{thm:ConstructionPhiGeneral}. Let \(V\subset M'\) be an open neighborhood around \(X=bM\),  \(c>0\) and \(\hat{E}\colon X\times(-c,c)\to V\) a diffeomorphism as in Lemma~\ref{lem:ConsequenceOfHMGoodE}. Let \(\tilde{\eta}_1\colon (-c,c)\to [0,1]\) be a smooth function such that \(\tilde{\eta}_1(s)=1\) for \(t\leq c/2\) and \(\tilde{\eta}_1(s)=0\) for \(t>3c/4\). Define \(\eta_1\colon V\cup M\to [0,1]\) by \(\eta_1(z):=\tilde{\eta_1}\circ\hat{E}^{-1}(z)\) for \(z\in V\) and \(\eta_1(z)=0\) otherwise. Put \(\eta_2\colon V\cup M\to [0,1]\), \(\eta_2(z)=1-\eta_1(z)\). We have that \(\eta_2\) is smooth on \(M\cup V\) with \(\operatorname{supp}\eta_2\subset M\). 
 For  \(k\geq k_0\) define \(\tilde{\rho}_k\colon X\times (-c,c)\to \R\), \(\tilde{\rho}_k(x,s)=\varphi_k(x)-s\) and put 
 \[\rho_k\colon M\cup V\to \R,\,\, \rho_k(z)=\eta_1(z)(\tilde{\rho}_k\circ \hat{E}^{-1})(z)-\eta_2(z)c/2\]
 and \(D_k=\{z\in M\cup V\colon \rho_k(z)<0\}\). We will show now that for \(k_0\) large enough we have that \(\{D_k\}_{k\geq k_0}\)  satisfies the desired properties with respect to the family of maps \(\{\mathcal{F}_k\}_{k\geq k_0}\). Define \(\tilde{\rho}_0\colon X\times (-c,c)\to \R\), \(\tilde{\rho}_0(x,s)=-s\), and put 
 \[\rho_0\colon M\cup V\to \R,\,\, \rho_k(z)=\eta_1(z)(\tilde{\rho}_0\circ \hat{E}^{-1})(z)-\eta_2(z)c/2.\]\\
 \textbf{Claim (i):} We have that \(\rho_0\) is a smooth defining function for \(M\).\\
 From the construction it follows that \(\rho_0\) is smooth. We first observe that since \(ds\neq 0\) on \(X\times (-c,c)\) and \(\hat{E}\) is a diffeomorphism we have that \(d\rho_0\) does not vanish in a neighborhood of \(X=bM\). Furthermore, since \(\hat{E}(x,0)=\iota(x)\), \(x\in X\), we find that \(\rho_0\) vanishes on \(bM\). Given \(z\in M\setminus V\) we have that \(\rho_0(z)=-c/2<0\). Given \(z\in M\cap V\) we have for \((x,s)=\hat{E}^{-1}(z)\) that \(0<s<c\) and \(\rho_0(z)=-\tilde{\eta}_1(s)s-c/2(1-\tilde{\eta}_1(s))=-\tilde{\eta}_1(s)(s-c/2)-c/2<0\). For \(0<s\leq c/2\) we find \(\rho_0(z)=-s<0\). For \(c/2<s<c\) we find \(\rho_0(z)<-c/2<0\). Hence \(\rho_0(z)<0\) for all \(z\in M\). Given \(z\in V\setminus \overline{M}\) we have for \((x,s)=\hat{E}^{-1}(z)\) that \(-c<s<0\), \(\tilde{\eta}_1(s)=1\) and hence \(\rho_0(z)=-\tilde{\eta}_1(s)s-c/2(1-\tilde{\eta}_1(s))=-s>0\). If follows that \(\rho_0(z)>0\) for all \(z\in V\setminus \overline{M}\). As a conclusion we obtain \(M=\{z\in M\cap V\colon \rho_0(z)<0\}\) and the claim follows.    \\
 \textbf{Claim (ii):} We have \(D_k\subset \subset M\) for all \(k\geq k_0\). \\
 Fix \(k\geq k_0\). It is enough to show that \(\rho_0\leq \rho_k\) on \(M\cap V\) and \(\rho_0+\varepsilon<\rho_k\) on \(X=bM\) for some \(\varepsilon >0\). Since \(\varphi_k>0\) we have \(\tilde{\rho}_0=-s\leq -s+\varphi_k=\tilde{\rho}_k\). It follows that \(\rho_0\leq \rho_k\) holds on \(M\cup V\). Furthermore, since \(X\) is compact, we find \(\varepsilon (=\varepsilon_k)>0\)  with  \(\varepsilon <\varphi_k\) and hence \(\tilde{\rho}_0(x,0)+\varepsilon=\varepsilon<\varphi_k(x)=\tilde{\rho}_k(x,0)\) for all \(x\in X\). As a consequence \(\rho_0+\varepsilon<\rho_k\) on \(X=bM\).  The claim follows.\\
 \textbf{Claim (iii):} For \(k\to \infty\) we have \(\rho_k\to \rho_0 \) on \(M\cup V\) in \(\mathscr{C}^\infty\)-topology. \\
 Since \(\varphi_k\to 0\) on \(X\) in \(\mathscr{C}^\infty\)-topology we obtain \(\tilde{\rho}_k\to \tilde{\rho}_0\) on \(X\times (-c,c)\) in \(\mathscr{C}^\infty\)-topology. From the fact that \(\hat{E}\) is independent of \(k\) the claim follows.\\
 \textbf{Claim (iv):} We have that \((z,k)\mapsto \rho_{k}(z)\) is a smooth function on \((M\cup V)\times [k_0,\infty)\).\\
 Since \((x,k)\mapsto \varphi_k(x)\) is smooth on \(X\times [k_0,\infty)\) the claim follows.\\
  \textbf{Claim (v):} We can choose \(k_0\) large enough such that for all \(k\geq k_0\) we have that \(\rho_k\) is a smooth defining function for \(D_k\). In particular,   \(D_k\subset\subset M\) is a smoothly bounded domain for all \(k\geq k_0\). \\
  This is a direct consequence of Claim~(iii). \\ 
   \textbf{Claim (vi):} For \(k\to \infty\) we have \(\overline{D_k}\to \overline{M} \) in the \(\mathscr{C}^\infty\)-topology for domains. Furthermore, the family \(\{D_k\}_{k\geq k_0}\) depends smoothly on \(k\)\\
   This immediately  follows from Claim~(iii), Claim~(iv) and Claim~(v).\\
    \textbf{Claim (vii):} We can choose \(k_0\) large enough such that for each \(k''>k'\geq k_0\) we have that \(D_{k'}\subset \subset D_{k''}\).\\
    From the properties of \(\{\varphi_k\}_{k\geq k_0}\) (see Theorem~\ref{thm:ConstructionPhiGeneral}) it follows that \(\varphi_k\) is strictly decreasing in \(k\) for \(k\geq k_0\) when \(k_0\) is large enough. From the construction of \(\rho_k\) we find that for all \(k''>k'\geq k_0\) we have \(\rho_{k''}\leq\rho_{k'}\) on \(M\cup V\) and \(\rho_{k''}<\rho_{k'}\) on \(V'\cap \overline{M}\) where \(V'\subset M'\) is an open neighborhood of \(bM\). From Claim~(vi) it follows that \(bD_k\subset V'\cap \overline{M}\) for all sufficiently large \(k\). The claim follows.\\   
   \textbf{Claim (viii):} We can choose \(k_0\) large enough such that for each \(k\geq k_0\) we have that \(\nu_k\colon M\to \R\), \(\nu_k(z):=|\mathcal{F}_k(z)|^2-1\), is a strictly plurisubharmonic defining function for \(D_k\).\\
   Fix \(k\geq k_0\). Given \(Z\in T^{1,0}M\) we have that \(d\mathcal{F}_kZ\neq 0\) since \(\mathcal{F}_k\) is a embedding. Since \(\mathcal{F}_k\) is holomorphic on \(M\) we observe that \(\partial\overline{\partial}\nu_k(Z,\overline{Z})=|dF_kZ|^2>0\). It follows that \(\nu_k\) is strictly plurisubharmonic for each \(k\geq k_0\).  From Proposition~\ref{pro:PropertiesFkGeneral}~(iii) it follows that for \(k_0\) large enough we have that \((d \nu_k\circ \hat{E})(x,s)\neq 0\) for all \(k\geq k_0\) and \(0\leq s<Ck^{-1}\) for some constant \(C>0\) independent of \(k\). Since \(\hat{E}\) is a diffeomorphism, \(\varphi_k=O(k^{-2})\) and \(z\in bD_k\) if and only if \(z\in V\) with \(\tilde{\rho}_k\circ\hat{E}^{-1}(z)=0\) we conclude that \(d\nu_k\neq 0\) on \(bD_k\) for all \(k\geq k_0\). For each \(k\geq k_0\) it follows from the strict maximum principle for strictly plurisubharmonic functions that we have \(\nu_k(z)<0\) for all \(z\in D_k\). Furthermore,  from the construction of \(\varphi_k\) we obtain that for any \(z\in M\) we have \(\rho_k(z)>0\) implies \(\nu_k(z)>0\). We conclude that \(\nu_k\) is positive on \(M\setminus\overline{D_k}\). We have shown that for each \(k\geq k_0\) we have \(D_k=\{z\in M\colon \nu_k(z)<0\}\) where \(\nu_k\) is a smooth strictly plurisubharmonic on \(M\) with \(d\nu_k\neq 0\) on \(bD_k\).\\
   \textbf{Claim (ix):} For each \(k\geq k_0\) we have that \(D_k\subset \subset M\) is a strictly pseudoconvex domain with real analytic boundary such that  \(\mathcal{F}_k(D_k)\subset \mathbb{B}^{N_k}\) and \(\mathcal{F}_k(bD_k)\subset \mathbb{S}^{2N_k-1}\).\\
   This immediately follows from Claim~(viii) since \(\nu_k\) is real analytic on \(M\). 
\end{proof}
As a direct consequence of Theorem~\ref{thm:ProofOnExhaustionBySphericalDom} we have the following.
\begin{corollary}\label{cor:ContrastToForstneric}
		Let \(Y\) be a Stein manifold of dimension \(\dim_\C Y=n+1\). Let \(M\subset \subset Y\) be a relatively compact strictly pseudoconvex domain with smooth boundary. In every neighborhood of \(M\) in the \(\mathscr{C}^\infty\)-topology for domains there exist a strictly pseudoconvex domain \(D\subset \subset M\) with real analytic boundary and a smooth embedding \(F\colon \overline{M}\to \C^N\) for some \(N\in \N\) which is holomorphic on \(M\) such that \(F\) restricts to a proper holomorphic map of \(D\) into the unit ball of \(\C^N\).  In particular, \(F\) defines a smooth embedding of \(\overline{D}\) into \(\C^N\) which is holomorphic on \(D\) and such that \(bD\) is mapped into the unit sphere.
	\end{corollary}	
\begin{proof}[\textbf{Proof of Theorem~\ref{thm:ContrastToForstneric}}]
	Theorem~\ref{thm:ContrastToForstneric} follows immediately from  Lemma~\ref{lem:ApproximationBySmoothDomains} and Corollary~\ref{cor:ContrastToForstneric}.
\end{proof}
\subsection{Proof of Theorem~\ref{thm:MainThmIntroGeneral}}
In order to prove  Theorem~\ref{thm:MainThmIntroGeneral} we need the following.
\begin{lemma}\label{lem:ConvergenceOfAlphaT}
	Let \(\mathcal{F}_k\colon \overline{M}\to \C^{\hat{N}_k+N}\), \(k\geq k_0\), be the map defined in~\eqref{eq:DefMapFkGeneral} where \(k_0\) is large enough. For \(k\geq k_0\) put \(N_k=N+\tilde{N}_k\) and define \(\beta(k)\in \R_+^{N_k}\) by
	\begin{eqnarray}\label{eq:DefBetakGeneral}
		\beta(k)=(\delta_1k,\ldots,\delta_1k, \lambda_{I^-_k},\ldots,\lambda_{I^+_k}).
	\end{eqnarray}
	where \(\hat{N}_k,  \lambda_{I^-_k},  \lambda_{I^+_k}\) are as in~\eqref{eq:DefHkGeneral}.
	Let \(\hat{E}, E\) be as in Lemma~\ref{lem:ConsequenceOfHMGoodE}, \(\alpha_{\beta(k)}, \mathcal{T}_{\beta(k)}\) be as Example~\ref{ex:PseudoHermSpheres} and \(\{\varphi_k\}_{k\geq k_0}\) as in Theorem~\ref{thm:ConstructionPhiGeneral}. Define \(G_k\colon X\to \C^{N_k}\), \(G_k(x)=\mathcal{F}_k(E(x,\varphi_k(x)))\) and put \(\alpha^k=G_k^{*}\alpha_{\beta(k)}\). Recall that \(\mathcal{T}\) denotes the Reeb vector field associated to \(\alpha\) as explained in the beginning of Section~\ref{sec:FunCalcOnSpscDom}. We have that  \(\alpha^k\) is well defined for all \(k\geq k_0\) and \((x,k)\mapsto (\alpha_{k})_x\) is smooth on \(X\times[k_0,\infty)\). Furthermore, one has
	\begin{itemize}
		\item [(i)]  \(\alpha^k=\alpha+O(k^{-1})\),
		\item[(ii)]  \(|k^{-1}\mathcal{T}_{\beta(k)}\circ G_k|^2=C^{(2)}_\chi/C^{(0)}_\chi + O(k^{-1})\),
		\item [(iii)] \(|k^{-1}(G_k)_*\mathcal{T}|^2=C^{(2)}_\chi/C^{(0)}_\chi + O(k^{-1})\), 
			\item[(iv) ]\(|k^{-1}(G_k)_*\mathcal{T}-k^{-1}\mathcal{T}_{\beta(k)}\circ G_k|^2=O(k^{-1})\) 
	\end{itemize}	
in \(\mathscr{C}^{\infty}\)-topology on \(X\). Moreover, one has
\begin{itemize}
	\item [(v)] For any smooth families of vector fields \(\{Y_k\}_{k\geq k_0},\{Z_k\}_{k\geq k_0}\in \vbunsec{X,\C TX} \) with \(Y_k=O(k^{-1})\) and \(Z_k=O(1)\) in \(\mathscr{C}^{\infty}\)-topology one has
	 \[\langle k^{-1}(G_k)_*Y_k,k^{-1}(G_k)_*Z_k\rangle=O(k^{-1}) \text{ and }  \langle k^{-1}(G_k)_*Y_k,k^{-1}\mathcal{T}_{\beta(k)}\circ G_k\rangle=O(k^{-1})\]   in \(\mathscr{C}^{\infty}\)-topology. 
\end{itemize}
\end{lemma}
\begin{proof}
	Recall that \(\mathcal{F}_k=\sqrt{\frac{k}{(k-d)}}(e^{-k}\mathcal{G},(C_\chi^{(0)}k^{n+2})^{-1/2}\mathcal{H}_k)\). We first observe that \(\alpha^k\) is well defined for all \(k\geq k_0\) since \(|\mathcal{F}_k(E(x,\varphi_k(x)))|^2=1\neq 0\), \(x\in X\), by construction. Put \(\alpha_{N_k}=\frac{1}{2i}\sum_{j=1}^{N_k}  \overline{z}_jdz_j-z_jd\overline{z}_j\). It follows that 
	\begin{eqnarray*}
		(\mathcal{F}_k^*\alpha_{N_k})_z=\frac{k}{(k-d)}\operatorname{Im}\left(e^{-2k}d_z\langle \mathcal{G}_k(z),\mathcal{G}_k(w)\rangle|_{z=w}+(C_\chi^{(0)}k^{n+2})^{-1}d_z\langle \mathcal{H}_k(z),\mathcal{H}_k(w)\rangle|_{z=w}\right).
	\end{eqnarray*}
	Using \eqref{eq:FundamentalIdentity} and \(d_z\langle \mathcal{H}_k(z),\mathcal{H}_k(w)\rangle|_{z=w}=d_z\eta_k(T_R)(z,w)|_{z=w}\) with \(\eta=|\chi|^2\) we immediately find that \((z,k)\mapsto(\mathcal{F}_k^*\alpha_{N_k})_z\) is smooth. With a similar argument we find that \((z,k)\mapsto r_{\beta(k)}(\mathcal{F}_k(z))\) is smooth where \(r_{\beta(k)}\colon \C^{N_k}\to \R\) is defined by \(r_{\beta(k)}=\langle \beta(k),(|z_1|^2,\ldots,|z_{N_k}|^2)\rangle\). Since
	\begin{eqnarray}\label{eq:DecomposePullbackalphaAsQuotient}
		\mathcal{F}_k^*\alpha_{\beta(k)}=\frac{\mathcal{F}_k^*\alpha_{N_k}}{r_{\beta(k)}(\mathcal{F}_k)}
	\end{eqnarray}
it follows from the smoothness of \((x,k)\mapsto \varphi_k(x)\) and the smoothness of \(E\) that  \((x,k)\mapsto (\alpha_{k})_x\) is smooth on \(X\times[k_0,\infty)\).
	  Put \(H_k\colon X\to \C^{N_k}\), \(H_k(x)=(0,\ldots,0,\mathcal{H}_k(E(x,\varphi_k(x))\). From Lemma~\ref{lem:IntegralLemma} we find that \(|H_k|^2=k^{n+2}(C_\chi^{(0)}+O(k^{-1}))\) in \(\mathscr{C}^\infty\)-topology. Possibly increasing \(k_0\) then shows that \(|H_k(x)|^2\neq 0\) for all \(x\in X\) and \(k\geq k_0\).\\
	\underline{Claim (i):}
	Since \(\mathcal{G}\) is independent of \(k\), from the properties of \(\{\varphi_k\}_{k\geq k_0}\) and \(\alpha_{\beta(k)}\)  using~\eqref{eq:DecomposePullbackalphaAsQuotient} we immediately observe that 
	\[\alpha^k=G_k^*\alpha_{\beta(k)}=H_k^*\alpha_{\beta(k)}+O(k^{-\infty})\]
	in \(\mathscr{C}^\infty\)-topology.
	Define (see~\eqref{eq:DefHkGeneral})  \(b_k(x)=\sum_{j=1}^{\infty}k^{-1}\lambda_j|\chi_{k}(\lambda_j)|^2|f_j(E(x,\varphi_k))|^2\), \(x\in X\), and put \(\alpha_{N_k}=\frac{1}{2i}\sum_{j=1}^{N_k}  \overline{z}_jdz_j-z_jd\overline{z}_j\). We have \(H_k^*\alpha_{\beta(k)}=(kb_k)^{-1}H_k^*\alpha_{N_k}\). 
	Applying Lemma~\ref{lem:ConsequenceOfHM} to the function \(t\mapsto \sqrt{t}|\chi(t)|^2\) and using  Lemma~\ref{lem:IntegralLemma} shows that \(k^{-n-2}b_k(x)=C_{\chi}^{(1)}+O(k^{-1})\) in \(\mathcal{C}^\infty\)-topology. Furthermore, we find
	\[H_k^*\alpha_{N_k}=\kappa_{\varphi_k}^*\operatorname{Im}d_{(x,s)}\langle \mathcal{H}_k(E (x,s)),\mathcal{H}_k(E(x',s'))\rangle|_{(x,s)=(x',s')}\]
	where \(\kappa_{\varphi_k}\colon X\to X\times[0,c)\), \(\kappa_{\varphi_k}(x)=(x,\varphi_k)\). Now let \(V\) be a smooth vector field on \(X\). We have that the vector field \((\kappa_{\varphi_k})_*V\) on \(X\times [0,c)\) can be written as  \((\kappa_{\varphi_k})_*V=V+(V\varphi_k)\frac{\partial}{\partial s}\). Then, from the properties of \(\{\varphi_k\}_{k\geq k_0}\), Lemma~\ref{lem:ConsequenceOfHM} and Lemma~\ref{lem:IntegralLemma} we find
	\[k^{-n-3}(H_k^*\alpha_{N_k})(V)=C_{\chi}^{(1)}(-i)\partial\rho(V+0\frac{\partial}{\partial s})+O(k^{-1})=C_{\chi}^{(1)}\alpha(V)+O(k^{-1})\]
	in \(\mathcal{C}^\infty\)-topology.  We conclude that \(H_k^*\alpha_{N_k}=k^{n+3}C_{\chi}^{(1)}\alpha+O(k^{-1})\) in \(\mathcal{C}^\infty\)-topology. The claim follows.\\
	\underline{Claim (ii):} 	Since \(\mathcal{G}\) is independent of \(k\), from the properties of \(\{\varphi_k\}_{k\geq k_0}\) and \(\mathcal{T}_{\beta(k)}\) we find that  
	\[|k^{-1}\mathcal{T}_{\beta(k)}\circ G_k|^2=\frac{1}{(k-d)k^{n+3}C_\chi^{(0)}}|\mathcal{T}_{\beta(k)}\circ H_k|^2+O(k^{-\infty})\]
	in \(\mathscr{C}^\infty\)-topology. Define \(b_k=|\mathcal{T}_{\beta(k)}\circ H_k|^2\). We find
	\[b_k(x)=\sum_{j=1}^{\infty}|\chi_k(\lambda_j)|^2\lambda_j^2|f_j(E(x,\varphi_k(x)))|^2.\]
	Then applying Lemma~\ref{lem:ConsequenceOfHM} to the function \(t\mapsto t\chi(t)\) and using Lemma~\ref{lem:IntegralLemma} shows that \(b_k=k^{n+4}(C^{(2)}_\chi+O(k^{-1})) \) in  \(\mathscr{C}^\infty\)-topology. The claim follows.\\
	 \underline{Claim (iii):} 	Since \(\mathcal{G}\) is independent of \(k\), from the properties of \(\{\varphi_k\}_{k\geq k_0}\) it follows that \(|(G_k)_*\mathcal{T}|^2=\frac{1}{(k-d)k^{n+1}C_\chi^{(0)}}|(H_k)_*\mathcal{T}|^2+O(k^{-\infty})\) in  \(\mathscr{C}^\infty\)-topology. We find
	 \[|(H_k)_*\mathcal{T}|^2(x)=d_{(x,s)}\otimes d_{(x',s')}\langle \mathcal{H}_k(E (x,s)),\mathcal{H}_k(E(x',s'))\rangle|_{(x,s)=(x',s')=(x,\varphi_k(x))}((\kappa_{\varphi_k})_*\mathcal{T}\otimes (\kappa_{\varphi_k})_*\mathcal{T}).\]
	 with \(\kappa_{\varphi_k}\colon X\to X\times[0,c)\), \(\kappa_{\varphi_k}(x)=(x,\varphi_k)\). We can write \((\kappa_{\varphi_k})_*\mathcal{T}=\mathcal{T}+(\mathcal{T}\varphi_k)\frac{\partial}{\partial s}\). Then it follows from Lemma~\ref{lem:ConsequenceOfHM} and Lemma~\ref{lem:IntegralLemma} that
	 \begin{eqnarray*}
	 	k^{-n-4}|(H_k)_*\mathcal{T}|^2&=&-C_\chi^{(2)}\partial\rho(\mathcal{T}+0\frac{\partial}{\partial s})\cdot\overline{\partial}\rho(\mathcal{T}+0\frac{\partial}{\partial s})+O(k^{-1})\\
	 	&=& C_\chi^{(2)}\alpha(\mathcal{T})\alpha(\mathcal{T})+O(k^{-1})=C_\chi^{(2)}+O(k^{-1})
	 \end{eqnarray*}
	 in \(\mathscr{C}^\infty\)-topology. The claim follows.\\
	    \underline{Claim (iv):} We have \[|k^{-1}(G_k)_*\mathcal{T}-k^{-1}\mathcal{T}_{\beta(k)}\circ G_k|^2=|k^{-1}(G_k)_*\mathcal{T}|^2-|k^{-1}\mathcal{T}_{\beta(k)}\circ G_k|^2-2\operatorname{Re}\langle k^{-1}(G_k)_*\mathcal{T} , k^{-1}\mathcal{T}_{\beta(k)}\circ G_k\rangle.\]
	    Using Claim (ii) and Claim (iii) it is enough to show that
	    \[\langle k^{-1}(G_k)_*\mathcal{T} , k^{-1}\mathcal{T}_{\beta(k)}\circ G_k\rangle=C^{(2)}_\chi/C^{(0)}_\chi + O(k^{-1})\]
	    in \(\mathscr{C}^\infty\)-topology. 
	    	Since \(\mathcal{G}\) is independent of \(k\), from the properties of \(\{\varphi_k\}_{k\geq k_0}\) it follows that
	    	\[\langle k^{-1}(G_k)_*\mathcal{T} , k^{-1}\mathcal{T}_{\beta(k)}\circ G_k\rangle= \frac{1}{(k-d)k^{n+3}C_\chi^{(0)}}\langle (H_k)_*\mathcal{T} , \mathcal{T}_{\beta(k)}\circ H_k\rangle+ O(k^{-\infty})\]
	    	in \(\mathscr{C}^\infty\)-topology. As above, applying Lemma~\ref{lem:ConsequenceOfHM} to the function \(t\mapsto \sqrt{t}|\chi(t)|^2\) and using Lemma~\ref{lem:IntegralLemma}, we find
	    	\[\langle (H_k)_*\mathcal{T} , \mathcal{T}_{\beta(k)}\circ H_k\rangle=k^{n+4}C_\chi^{(2)}+O(k^{-1})\]
	    	in \(\mathscr{C}^\infty\)-topology. The claim follows.  \\
	    \underline{Claim (v):} Since \(\mathcal{G}\) is independent of \(k\) and from the properties of \(\{\varphi_k\}_{k\geq k_0}\), \(\{Y_k\}_{k\geq k_0}\),  \(\{Z_k\}_{k\geq k_0}\) it is enough to prove the statement for \(k^{-n-4}\langle (H_k)_*Y_k,(H_k)_*Z_k\rangle\). We have that \(\langle (H_k)_*Y_k,(H_k)_*Z_k\rangle(x)=\)
	    \begin{eqnarray*}
	    d_{(x,s)}\otimes d_{(x',s')}\langle \mathcal{H}_k(E (x,s)),\mathcal{H}_k(E(x',s'))\rangle|_{(x,s)=(x',s')=(x,\varphi_k(x))}((\kappa_{\varphi_k})_*Y_k\otimes (\kappa_{\varphi_k})_*Z_k).
	    \end{eqnarray*}
	    with \(\kappa_{\varphi_k}\colon X\to X\times[0,c)\), \(\kappa_{\varphi_k}(x)=(x,\varphi_k(x))\). Since \((\kappa_{\varphi_k})_*Y_k=Y_k+(Y_k\varphi_k)\frac{\partial}{\partial s}\), \((\kappa_{\varphi_k})_*Z_k=Z_k+(Z_k\varphi_k)\frac{\partial}{\partial s}\) and \(Y_k=O(k^{-1}),\, Z_k=O(1)\) it follows from Lemma~\ref{lem:ConsequenceOfHM}  and Lemma~\ref{lem:IntegralLemma} that \(k^{-n-4}\langle (H_k)_*Y_k,(H_k)_*Z_k\rangle=0+O(k^{-1})\) in \(\mathscr{C}^\infty\)-topology. That \(\langle k^{-1}(G_k)_*Y_k,k^{-1}\mathcal{T}_{\beta(k)}\circ G_k\rangle=O(k^{-1})\)  in \(\mathscr{C}^\infty\)-topology follows analogously. 
\end{proof}

\begin{proof}[\textbf{Proof of Theorem~\ref{thm:MainThmIntroGeneral}}]
	Since \((X,T^{1,0}X)\) is CR embeddable into the complex Euclidean space it follows from Lemma~\ref{lem:CharacterizationEmbeddable} that there exists a complex manifold \(M'\) of complex dimension \(\dim_\C M'=n+1\) such that \((X,T^{1,0}X)\) is CR isomorphic to the boundary of a smoothly bounded strictly pseudoconvex domain \(M\subset\subset M'\). Hence it is enough to consider the case \(X=bM\) for \(M',M, \alpha,\mathcal{T}\) as introduced in the beginning of Section~\ref{sec:FunCalcOnSpscDom}.\\ 
	  There exist \(N\in \N\) and a CR embedding \(\tilde{G}\colon X\to \C^N\). Extend \(\tilde{G}\) holomorphically to a smooth map \(\mathcal{G}\colon \overline{M}\to \C^N \) which is holomorphic on \(M\) (see Kohn--Rossi~\cite{KR65}, also  Hörmander~\cite{Hoermander_2000}). Since \(d\tilde{G}\) is injective, \(\overline{\partial}\mathcal{G}|_X=0\) and \(\mathcal{G}|_X=\tilde{G}\) we find that there exists an open neighborhood \(U\subset M'\) around \(X\) such that \(\mathcal{G}\) is a smooth embedding of \(\overline{M}\cap U\) which is holomorphic on \(M\cap U\). Let \(\mathcal{F}_k\colon \overline{M}\to \C^{\hat{N}_k+N}\) be the map defined in~\ref{eq:DefMapFkGeneral} with respect to this \(\mathcal{G}\).  We choose the constant \(d\) in the definition of \(\mathcal{F}_k\) large enough so that Proposition~\ref{pro:PropertiesFkGeneral}~(vi) holds. It follows that there is \(k_1>0\) such that for all \(k\geq k_1\) the map \(\mathcal{F}_k\colon \overline{M}\to \C^{\hat{N}_k+N}\) is well defined and defines a smooth embedding of \(\overline{M}\cap U\) into \(\C^{N+N_k}\) which is holomorphic on \(M\cap U\) . With \(N_k=\hat{N}_k+N\) we have \(N_k=O(k^{n+1})\) by Lemma~\ref{lem:DimensionGrowGeneral}. Let \(V\subset M'\) be an open neighborhood around \(X=bM\),  \(c>0\) and \(\hat{E}\colon X\times(-c,c)\to V\) a diffeomorphism as in Lemma~\ref{lem:ConsequenceOfHMGoodE} and \(E=\hat{E}|_{X\times[0,c)}\) its restriction. Without loss of generality we can assume that \(V\subset\subset U\) holds. Then \(\hat{E}\) induces a complex structure \(T^{1,0}Y\) on \(Y:=X\times(-c,c)\). For each smooth function \(\varphi\colon X\to (-c,c)\) consider the graph \(X_\varphi:=\{(x,s)\in X\times (-c,c)\colon s=\varphi(x)\}\) over \(X\) with induced CR structure given by \(T^{1,0}X_\varphi=\C TX_\varphi\cap T^{1,0}Y\). We have that the projection \(\operatorname{pr}(x,s)= x\) restricts to a diffeomorphism between \(X_\varphi\) and \(X\).  Let \(\iota_\varphi\colon X_\varphi\to Y\) denote the inclusion and \(\kappa_\varphi\colon X\to X_\varphi\) the diffeomorphism \(\kappa_\varphi(x)=(x,\varphi(x)) \).  Let \(\{\varphi_k\}_{k\geq k_2}\) be the family of smooth functions from Theorem~\ref{thm:ConstructionPhiGeneral} where \(k_2\geq k_1\). Define \(F_k=\mathcal{F}_k\circ E\). Since \(\varphi_k>0\) it follows from the considerations above and the properties of \(\varphi_k\), \(\mathcal{F}_k\colon X\to \C^{N_k}\), \(\hat{E}\) and \(E\) that
	  \begin{itemize}
	  	\item[-] \(\kappa_0\) is a CR diffeomorphism between \((X,T^{1,0}X)\) and \((X_0,T^{1,0}X_0)\), 
	  	\item[-] \(F_k\) is a smooth embedding of \(X\times[0,c)\) into \(\C^{N_k}\) for all \(k\geq k_2\)
	  	\item[-] \(F_k\) is holomorphic on \(X\times (0,c)\) for all \(k\geq k_2\),
	  	\item[-] \(\|\varphi_k\|_{\mathscr{C}^m(X)}=O(k^{-2})\) for all \(m\in \N\),
	  	\item[-] \(N_k=O(k^{n+1})\),
	  	\item[-]  \((x,k)\mapsto \varphi_k(x)\) is smooth \(X\times[k_0,\infty)\),
	  	\item[-] \(\frac{\partial \varphi_k(x)}{\partial k}\leq -k^{-4}\) for all \(x\in X\) and \(k\geq k_0\).
	  \end{itemize}
	  For each \(k\geq k_2\) define \(\beta(k)\in \R^{N_k}_+\) by
	  \[\beta(k)=(\delta_1k,\ldots,\delta_1k, \lambda_{I^-_k},\ldots,\lambda_{I^+_k})\] with \(I^-_k,I^+_k\) as in~\eqref{eq:DefHkGeneral} and put \(\gamma_k=\iota_{\varphi_k}^*F_k^*\alpha_{\beta(k)}\). We have to show that \((X_{\varphi_k},T^{1,0}X_{\varphi_k},\gamma_k )\) is a real analytic pseudohermitian manifold with \(F_k(X_{\varphi_k})\subset \mathbb{S}^{2N_k-1}\) for all sufficiently large \(k\). Since \(X_{\varphi_k}\) is a smooth real hypersurface of \(Y\) it follows from the definition of \(T^{1,0}X_{\varphi_k}\) and the properties of \(\varphi_k\) that \((X_{\varphi_k},T^{1,0}X_{\varphi_k})\) is a (codimension one) CR manifold. Hence \((\mathcal{F}_k(X_{\varphi_k}),(\mathcal{F}_k)_*T^{1,0}X_{\varphi_k})\) is a CR submanifold of \(\C^{N_k}\)  with \(\mathcal{F}_k(X_{\varphi_k})\subset \mathbb{S}^{2N_k-1}\). It follows from Lemma~\ref{lem:CRsubmanifoldsOfSpheresAreStrictlyPseudoconvex} that \((F_k(X_{\varphi_k}),(\mathcal{F}_k)_*T^{1,0}X_{\varphi_k})\) is strictly pseudoconvex and that \(\alpha_{\beta(k)}\) defines a pseudohermitian structure on \((X_{\varphi_k},T^{1,0}X_{\varphi_k})\). As a conclusion we have that \((X_{\varphi_k},T^{1,0}X_{\varphi_k},\gamma_k )\) is a pseudohermitian submanifold.\\
	   Note that the complex structure \(T^{1,0}Y\) on \(Y\) automatically fixes a real analytic structure on \(Y\). In order to show that \((X_{\varphi_k},T^{1,0}X_{\varphi_k})\) is a real analytic CR submanifold of \(Y\) it is enough to show that \(\nu_k\colon X\times (0,c)\to \R\), \(\nu_k(x,s)=|F_k(x,s)|^2-1\), is a real analytic defining function for \(X_{\varphi_k}\) near \(X_{\varphi_k}\).   Since \(\mathcal{F}_k\circ E\) defines a holomorphic embedding of \(X\times(0,c)\) into \(\C^{N_k}\) we have that \(\nu_k\) is real analytic. 
	   Furthermore, we have \(\nu_k(x,s)=0\) for all \((x,s)\in X_{\varphi_k}\). Since \(\|\varphi_k\|_{\mathscr{C}^0(X)}=O(k^{-2})\) it follows from Proposition~\ref{pro:PropertiesFkGeneral} (since \(\nu_k=h_k-1\)) that we can choose \(k_3\geq k_2\) large enough such that \(\frac{\partial \nu_k}{\partial s}(x,\varphi_k(x))\neq 0\) for all \(x\in X\) and  \(k\geq k_3\).  Hence, we find \(d\nu_k\neq 0\) on \(X_{\varphi_k}\) for all \(k\geq k_0\).  Since \(F_k\) is holomorphic in a neighborhood of \(X_{\varphi_k}\) and \(\alpha_{\beta(k)}\) is real analytic on \(C^{N_k}\setminus\{0\}\) we find that \(\gamma_k\) is real analytic on \(X_{\varphi_k}\).\\
	  Now we consider the CR structure \(\mathcal{V}(\varphi)=\operatorname{pr}_*T^{1,0}X_\varphi=(\kappa_\varphi)^{-1}_*T^{1,0}X_\varphi\) on \(X\). By the properties of \(\hat{E}\) we have that \(\mathcal{V}(0)=T^{1,0}X\). For \(k\geq k_3\) put \(\mathcal{V}^k:=\mathcal{V}(\varphi_k)\) and \(\alpha_k:=\kappa^*_{\varphi_k}\gamma_k\). From the construction it follows that \(\kappa_{\varphi_k}\) is a CR diffeomorphism between \((X,\mathcal{V}^k)\) and \((X_{\varphi_k},T^{1,0}X_{\varphi_{k}})\) with \(\alpha^k=\kappa^*_{\varphi_k}\gamma_k\). Hence \((X,\mathcal{V}^k, \alpha^k)\) is pseudohermitian manifold. 
	  From Lemma~\ref{lem:CRStructureConvergenceGeneral} we find that \(\mathcal{V}^k\to T^{1,0}X\) in \(\mathscr{C}^\infty\)-topology for \(k\to\infty\) and that \(\mathcal{V}^k\) depends smoothly on \(k\).\\
	  For \(k\geq k_3\) define \(G_k\colon X\to \C^{N_k}\), \(G_k=F_k\circ \kappa_{\varphi_k}\). From the construction we have that \(G_k\) is a CR embedding of \((X,\mathcal{V}^k)\) into \(\C^{N_k}\) with \(G_k(X)\subset \mathbb{S}^{2N_k-1}\),  \(\alpha^k=G_k^*\alpha_{\beta(k)}\). 
	  From Lemma~\ref{lem:ConvergenceOfAlphaT} it follows that \(\alpha_k=\alpha+O(k^{-1})\) in \(\mathscr{C}^\infty\)-topology. 
	  Summing up, we have shown so far that for each \(k\geq k_3\) we have that  \((X,\mathcal{V}^k,\alpha_k)\) is a pseudohermitian manifold and \(G_k\colon X\to \C^{N_k}\) defines an isomorphism to a real analytic pseudohermitian submanifold of the \(\beta(k)\)-sphere. Furthermore, we have \(\alpha^k=\alpha+O(k^{-1})\) in \(\mathscr{C}^\infty\)-topology and \(\alpha_k\) depends smoothly on \(k\).  Let \(\mathcal{T}^k\) denote the Reeb vector field associated to \(\alpha^{k}\). We note that \(\mathcal{T}^k\) is uniquely determined by
	  \[\mathcal{T}^k\lrcorner\alpha^k\equiv1,\,\,\text{ and } \mathcal{T}^k\lrcorner d\alpha^k\equiv0.\] 
	   This defines a pointwise  non-degenerated system of linear equations since \(d\alpha_k\) is non-degenerated on \(\operatorname{Re}\mathcal{V}^k\). Then with \(\alpha^k=\alpha+O(k^{-1})\) in \(\mathscr{C}^\infty\)-topology using Lemma~\ref{lem:SmoothInverseGeneral} below we find that \(\mathcal{T}^k=\mathcal{T}+O(k^{-1})\)  in \(\mathscr{C}^\infty\)-topology and \(\mathcal{T}^k\) depends smoothly on \(k\).\\
	   We have 
	   \[|(G_k)_*\mathcal{T}^k|^2=|(G_k)_*\mathcal{T}^k|^2+|(G_k)_*(\mathcal{T}^k-\mathcal{T})|^2+2\operatorname{Re}\langle (G_k)_*\mathcal{T},(G_k)_*(\mathcal{T}^k-\mathcal{T})\rangle\]
	   and 
	   \[\langle(G_k)_*\mathcal{T}^k,\mathcal{T}_{\beta(k)}\circ G_k\rangle=\langle(G_k)_*(\mathcal{T}^k-\mathcal{T}),\mathcal{T}_{\beta(k)}\circ G_k\rangle + \langle(G_k)_*\mathcal{T},\mathcal{T}_{\beta(k)}\circ G_k\rangle.\]
	   Since \(\mathcal{T}^k-\mathcal{T}=O(k^{-1})\) in \(\mathscr{C}^\infty\)-topology we obtain (iv) in Theorem~\ref{thm:MainThmIntroGeneral} from Lemma~\ref{lem:ConvergenceOfAlphaT}. 
\end{proof}
\begin{remark}\label{rmk:RemarksProofGeneral}.
	\begin{itemize}
		\item [(1)] Since  \(\nu_k\)  constructed in in the proof of Theorem~\ref{thm:MainThmIntroGeneral} is a defining function for \(X_{\varphi_k}\) it follows that there is an open neighborhood \(U_k\subset X\times (0,c)\) of \(X_{\varphi_k}\) such that \(F_k(U_k)\) is an \((n+1)\)-dimensional complex submanifold of \(\C^{N_k}\) and \(F_k(X_{\varphi_k})=G_k(X)\) is the transversal intersection of \(F_k(U_k)\) with \(\mathbb{S}^{2N_k-1}\).
		\item[(2)] Let \(\ell^2(\C)=\{(a_j)_{j\in\N}\colon \sum_{j=1}^\infty|a_j|^2<\infty\}\).
		Given \(N\geq 1\), \(1\leq m^-\leq m^+\), define the projection \(\operatorname{Pr}_{m^-,m^+}\colon \ell^2(\C)\to \C^{N+m^+-m^-+1}\) as in Remark~\ref{rmk:MainThmGeneral}. 
		We define a family of maps 
		\[\{\hat{F}_k\colon X\times[0,c)\to \ell^2(\C)\}_{k\geq k_0}\]
		 by
		\[\hat{F}_k=\sqrt{\frac{k}{k-d}}\left(e^{-k}\mathcal{G}\circ E,\frac{\chi_{k}(\lambda_1)}{\sqrt{k^{n+2}C^{(0)}_\chi}}f_1\circ E,\frac{\chi_{k}(\lambda_2)}{\sqrt{k^{n+2}C^{(0)}_\chi}}f_2\circ E,\ldots\right).\]
		 We have that almost all entries of \(\hat{F}_k\) are zero, \(F_k=\operatorname{Pr}_{I_k^-,I_k^+}\circ \hat{F}_k\) (with \(I_k^-,I_k^+\) as in~\eqref{eq:DefHkGeneral}) and \(|F_k|^2=|\hat{F}_k|^2\). Furthermore, for any entry \((\hat{F}_k)_j\), \(j\in \N\), of \(\hat{F}_k\) we have that \((x,s,k)\mapsto (\hat{F}_k)_j(x,s)\) is smooth on \(X\times [0,c)\times [k_0,\infty)\). Roughly speaking, we have that the maps \(F_k\) and \(G_k\) in Theorem~\ref{thm:MainThmIntroGeneral} are induced by maps which depend smoothly on \(k\). 
	\end{itemize}
\end{remark}
	
 \begin{lemma}\label{lem:SmoothInverseGeneral}
 	Let \(\mathbb{K}\in \{\R,\C\}\), \(U,V\subset \R^n \) be open with \(U\subset\subset V\), \(A\colon V\to \operatorname{Mat}_{n\times n}(\mathbb{K})\) and \(b\colon V\to \mathbb{K}^n\) be smooth maps such that \(A(x)\) is invertible for any \(x\in V\). Given a smooth map \(E\colon V\times \R\to \operatorname{Mat}_{n\times n}( \mathbb{K})\) such that \(x\mapsto E_k(x):=E(x,k)\) is an \(O(k^{-1})\) in \(\mathscr{C}^\infty\)-topology on \(V\) as \(k\to \infty\) there exist \(k_0>0\) and a smooth map \(v\colon U\times (k_0,\infty) \to \mathbb{K}^n\) such that \((A(x)+E_k(x))v(x,k)=b(x)\) for all \(x\in U\), \(k>k_0\) and for \(v_k(x):=v(x,k)\) we have that \(v_k=A^{-1}b+O(k^{-1})\) in \(\mathcal{C}^\infty\)-topology on \(U\) as \(k\to \infty\).
 \end{lemma}
 \begin{proof}
 	From the implicit function theorem it follows that \(x\mapsto(A(x))^{-1}\) is smooth on \(V\). Since \((A(x)+E_k(x))v(x,k)=b(x)\) if and only if  \((\operatorname{Id}+(A(x))^{-1}E_k(x))v(x,k)=(A(x))^{-1}b(x)\) it is enough to show the statement for the case \(A\equiv \operatorname{Id}\). Let \(W\subset\subset V\) be open with \(U\subset\subset W\). Since \(E_k=O(k^{-1})\) it follows from the continuity of the determinant that there exists \(k_1>0\) such that \(\operatorname{Id}+E_k(x)\) is invertible for all \(k\geq k_1\) and all \(x\in W\). Since \((x,k)\mapsto E(x,k)\) is smooth it follows with \(R(x,k):=(\operatorname{Id}+E_k(x))^{-1}\) that \(R\) is smooth on \(W\times (k_1,\infty)\). Hence \(v\colon W\times  (k_1,\infty)\to \mathbb{K}^n\)  \(v(x,k):=R(x,k)b(x)\) is smooth with    \((\operatorname{Id}+E_k(x))v(x,k)=b(x)\) for all \(x\in W\) and \(k>k_1\). We will now show that for \(R_k(x):=R(x,k)\) we have that \(R_k=\operatorname{Id}+O(k^{-1})\)  in \(\mathscr{C}^\infty\)-topology on \(U\). Then it follows that for \(v_k(x):=v(x,k)\) we have that \(v_k=b+O(k^{-1})\) in \(\mathscr{C}^\infty\)-topology on \(U\). We will show by induction with respect to \(|\alpha|\) that there exists \(k_2>k_1\) such that for any \(\alpha\in \N^{n}_0\) there is \(C_\alpha>0\) with  \(|d_x^{\alpha}(R_k-\operatorname{Id})|\leq C_\alpha k^{-1}\) on \(U\) for all \(k\geq k_2\) . Choose \(k_2>k_1\) such that \(\|E_k(x)\|_{\operatorname{op}}\leq 1/2\) for all \(x\in W\) and \(k\geq k_2\).
 	 Then \(\|(\operatorname{Id}+E_k(x))-\operatorname{Id}\|_{\operatorname{op}}\leq 1/2<1\) and hence \(\|R_k(x)\|_{\operatorname{op}}\leq \frac{1}{1-\|E_k(x)\|_{\operatorname{op}}}\leq 2\) for all \(x\in U\) and \(k\geq k_2\). For \(\alpha=0\) it follows from \(R_k(\operatorname{Id}+E_k)=\operatorname{Id}\) that \(R_k-\operatorname{Id}=-R_kE_k\). We conclude that there exists \(C_0>0\) with  \(|(R_k-\operatorname{Id})|\leq C_0 k^{-1}\) on \(U\) for all \(k\geq k_2\). Now assume that for some \(N\geq 0\) the statement holds for all \(\alpha\in \N_0^n\) with \(|\alpha|\leq N\). Let \(\alpha\in  \N_0^n\) with \(|\alpha|=N+1\) be arbitrary.
 	 From the Leibniz rule we obtain
 	 \[d^\alpha_x (R_k(\operatorname{Id}+E_k))=(d^\alpha_x R_k)(\operatorname{Id}+E_k)+\sum_{|\beta|\leq N}c_{\beta,\alpha}d^{\beta}_xR_k(d_x^{\alpha-\beta}(\operatorname{Id}+E_k))\]
 	 for integer coefficients \(c_{\alpha,\beta}\) independent of \(k\) and \(x\).
 	 Since \(|\alpha|\geq 1\) we find \(d^\alpha_x (R_k(\operatorname{Id}+E_k))=0\) and \(d^\alpha_x R_k=d^\alpha_x (R_k-\operatorname{Id})\). Then, using the induction hypothesis, the assumptions on \(E_k\) and \(\|R_k(x)\|_{\operatorname{op}}\leq 2\) for \(x\in U\) and \(k\geq k_2\), the claim follows.  
 \end{proof}
\begin{definition}\label{def:BundleConvergence}
	 Let \(E\) be a Hermitian vector bundle over a compact manifold \(X\) and let \(\mathcal{V}\) be a complex subbundle in \(E\). Given a family of  complex subbundles \(\{\mathcal{V}^k\}_{k\geq k_0}\) in \(E\) we say that \(\mathcal{V}^k\to \mathcal{V}\) in \(\mathscr{C}^\infty\)-topology if there exist a Hermitian vector bundle \(Q\) over \(X\), a smooth section \(S \in \vbunsec{X,\operatorname{Lin}(E,Q)}\) and for each \(k\geq k_0\)  a smooth section \(S^k\in \vbunsec{X,\operatorname{Lin}(E,Q)}\) such that \(S_x,S^k_x\colon E_x\to Q_x\) are surjective, \(\ker S_x=\mathcal{V}\), \(\ker S^k_x=\mathcal{V}^k\) for all \(x\in X\) and \(S_k\to S\) in \(\mathscr{C}^\infty\)-topology as \(k\to\infty\). We say that the family \(\{\mathcal{V}^k\}_{k\geq k_0}\) is smooth in \(k\) if \((x,k)\mapsto S_x^k\) is smooth.   
\end{definition}
\begin{lemma}\label{lem:CRStructureConvergenceGeneral}
		Let \(X\) be a compact manifold with \(\dim_\R X=2n+1\), \(n\geq 1\), and let \(T^{1,0}Y\) be an almost complex structure on \(Y=X\times(-3c,3c)\) for some \(c>0\).
	Given a smooth function \(\varphi\colon X\to (0,c) \) put \(X_\varphi=\{(x,\varphi(x))\colon x\in X\}\), \(T^{1,0}X_{\varphi}=\C TX_\varphi\cap T^{1,0}Y\), \(\kappa_\varphi\colon X\to X_{\varphi}\), \(\kappa_{\varphi}(x)=(x,\varphi(x))\).  Define \(\mathcal{V}(\varphi)=(\kappa_\varphi)^{-1}_*T^{1,0}X_\varphi\). Now let \(\{\varphi_k\colon X\to (0,c)\}_{k\geq k_1}\) for some \(k_1>0\) be a family of smooth functions such that \((x,k)\mapsto\varphi_k(x)\) is smooth on \(X\times (k_1,\infty)\) and \(\|\varphi_k\|_{\mathscr{C}^m(X)}=O(k^{-1})\) for all \(m\in \N\). There exists \(k_0>k_1\) such that for \(\mathcal{V}^k:=\mathcal{V}(\varphi_k)\)  we have that \(\{\mathcal{V}^k\}_{k> k_0}\) is a smooth family of subbundles in \(\C TX\) with \(\mathcal{V}^k\to \mathcal{V}(0)\) in \(\mathscr{C}^\infty\)-topology. 
\end{lemma}
 \begin{proof}
 	    Consider the bundle \(E=\C TX\times \C\) over \(X\) and choose any Hermitian metric on \(E\). Denote by \(\operatorname{pr}\colon Y\to X\), \((x,s)\mapsto x\), the projection onto the first factor. The isomorphism \((V,a)\mapsto V+a\frac{\partial }{\partial s}\)  between \(\C T_xX\times \C\) and \(\C T_{(x,s)}Y\) defines an isomorphism between the vector bundles \(\operatorname{pr}^*E\) and \(\C TY\)  and leads to an identification of \(E\) with \(\C TY|_{X\times\{0\}}\). Since \(T^{1,0}Y\) is an almost complex structure we have \(\C TY=T^{1,0}Y\oplus T^{0,1}Y\) where \(T^{0,1}Y=\overline{T^{1,0}Y}\) and hence \(\operatorname{rank}_\C T^{1,0}Y=\operatorname{rank}_\C T^{0,1}Y=n+1\). Let \(\hat{P}\in \vbunsec{Y,\operatorname{End}(\C TY)}\) denote the projection onto \(T^{1,0}Y\) with respect to the decomposition \(\C TY=T^{1,0}Y\oplus T^{0,1}Y\). In particular, we have for any \((x,s)\in Y\) that \(\hat{P}_{(x,s)}\hat{P}_{(x,s)}=\hat{P}_{(x,s)}\) with \(\operatorname{ran}\hat{P}_{(x,s)}=T^{1,0}_{(x,s)}Y\) and \(\ker\hat{P}_{(x,s)}=T^{0,1}_{(x,s)}Y\).   Define \(P,P^k\in \vbunsec{X,\operatorname{End}(E)}\) as follows: For \(x\in X\), \((V,a)\in E_x \) and \(k\geq k_1\) put \(P_x(V,a)=(W,b)\) where \(W+b\frac{\partial}{\partial s}=\hat{P}_{x,0}(V+a\frac{\partial}{\partial s})\) and \(P^k_x(V,a)=(W,b)\) where \(W+(b+V\varphi_k)\frac{\partial}{\partial s}=\hat{P}_{(x,\varphi_k(x))}(V+(a+V\varphi_k)\frac{\partial}{\partial s})\). From the construction it follows that \(P_xP_x=P_x\) and \(P^k_xP^k_x=P^k_x\) for all \(x\in X\) and \(k>k_1\).\\
 	    \textbf{Claim (i):} We have that \((x,k)\to P_k\) is smooth on \(X\times (k_1,\infty)\) and \(P_k=P+O(k^{-1})\) in \(\mathscr{C}^\infty\)-topology.\\
 	    To see this, we take a point \(x_0\) and choose a local frame \(V_1,\ldots,V_{2n+1}\) of \(\C TX\) around \(x_0\). Then \(V_1,\ldots,V_{2n+1},\frac{\partial}{\partial s}\) defines a frame for \(\C TY\) in an open neighborhood around \((x_0,0)\) in \(Y\).  With respect to this frame we can write \(\hat{P}_{(x,s)}\) as a matrix \((\hat{P}_{l,j}(x,s))\). From the smoothness of \(\hat{P}_{l,j}\) and the properties of \(\varphi_k\) we find \(\hat{P}_{l,j}(x,\varphi_k(x))=\hat{P}_{l,j}(x,0)+O(k^{-1})\) in \(\mathscr{C}^\infty\)-topology on \(X\) and that \((x,k)\mapsto \hat{P}_{l,j}(x,\varphi_k(x))\) is smooth on \(X\times (k_1,\infty)\). Writing \(P^k_x\) as a matrix \((P^k_{l,j}(x))\) with respect to the local frame \((V_1,0),\ldots,(V_{2n+1},0),(0,1)\) of \(E\) around \(x_0\)  we find  
 	    \[(P^k_{l,j}(x))=   \begin{pmatrix} \operatorname{Id}& 0\\
 	    	-(V_1\varphi_k,\ldots,V_{2n+1}\varphi_k)& 1\end{pmatrix}(\hat{P}_{l,j}(x,\varphi_k(x)))\begin{pmatrix} \operatorname{Id}& 0\\
 	     (V_1\varphi_k,\ldots,V_{2n+1}\varphi_k)& 1\end{pmatrix}.\]
      Then the claim follows from the properties of \(\varphi_k\) and the compactness of \(X\).\\
      Consider the section \(r\in \vbunsec{X,\operatorname{Lin}(E,\C)}\),  defined by \(r_x(V,a)=a\) for \(x\in X\), \((V,a)\in E_x\).
      We then consider  \(\C TX\) as the  subbundle \(\{(V,a)\in E\colon a=0\}=\ker r\) of \(E\) using the identification \(V\mapsto (V,0)\).\\
      \textbf{Claim (ii):} We have  \(\mathcal{V}(0)=\ker (1-P)\cap \ker r\) and for each \(k>k_1\) we have \(\mathcal{V}^k=\ker (1-P^k)\cap \ker r\).\\
      Let \(x\in X\) and \((V,a)\in E_x\) be arbitrary. We find \((V,a)\in \mathcal{V}(0)_x\) if and only if \(V+a\frac{\partial}{\partial s}\in T^{1,0}_{(x,0)}Y\) and \(V+a\frac{\partial}{\partial s}\in \C T_xX\). This is equivalent to \((1-\hat{P}_{(x,0)}) (V+a\frac{\partial}{\partial s})=0\) and \(a=0\) which is true if and only if  \((1-P_x)(V,a)=0\) and \(r(V,a)=0\). Fix \(k>k_1\). We find \((V,a)\in \mathcal{V}^k_x\) if and only if \(V+(a+V\varphi_k)\frac{\partial}{\partial s}\in T^{1,0}_{(x,\varphi_k(x))}Y\) and \(V+(a+V\varphi_k)\frac{\partial}{\partial s}\in \C TX_{\varphi_k}\). Note that a vector \(W\in \C T_{(x,\varphi_k(x))}Y\) is in \(\C T_{(x,\varphi_k(x))}X_{\varphi_k}\) if and only if \(W(s-\varphi_k)=0\). Moreover, a vector \(W\in \C T_{(x,\varphi_k(x))}Y\) is in \(\C T^{1,0}_{(x,\varphi_k(x))}Y\) if and only if \((1-\hat{P}_{(x,\varphi_k(x))})W=0\). We conclude \((V,a)\in \mathcal{V}^k_x\) if and only if \((1-P^k_x)(V,a)=0\)  and \(r(V,a)=0\).\\
      Put \(\mathcal{C}=\ker (1-P)\), \(\mathcal{C}^k=\ker (1-P^k)\).\\
       \textbf{Claim (iii)}: We have \(\overline{\mathcal{C}}=\ker P\) and \(\overline{\mathcal{C}^k}=\ker P^k\) for all \(k> k_1\)\\
       Let \(x\in X\) and \((V,a)\in E_x\) be arbitrary. We have \(P^k_x(V,a)=0\) if and only if \(\hat{P}_{(x,\varphi_k(x))}(V+(a+V\varphi_k)\frac{\partial}{\partial s})=0\) which is equivalent to \(V+(a+V\varphi_k)\frac{\partial}{\partial s}\in T^{0,1}_{(x,\varphi_k(x))}Y\). Hence \(P^k(V,a)=0\) if and only if \(\overline{V}+(\overline{a}+\overline{V}\varphi_k)\frac{\partial}{\partial s}\in T^{1,0}_{(x,\varphi_k(x))}Y\) which means \((1-P^k_x)(\overline{V},\overline{a})\)=0 and is equivalent to \((\overline{V},\overline{a})\in \mathcal{C}^k\). That \(\overline{\mathcal{C}}=\ker P\) follows analogously.\\
      \textbf{Claim (iv)}: There exists \(k_2>k_1\) such that \(\mathcal{C}\cap\overline{\mathcal{C}^k}=0\) for all \(k\geq k_2\).\\
      Assume that for any \(j\in \N\) there exist \(l_j\in \R\), \(V_j\in \mathcal{C}\cap\overline{\mathcal{C}^{l_j}}\) such that \(k_1<l_1<l_2<\ldots\) with \(\lim_{j\to\infty}l_j=\infty\) and \(|V_j|=1\) for all \(j\in \N\). Since the unit sphere bundle of \(E\) is compact we can assume (by possibly passing to a subsequence) that \(\lim_{j\to \infty}V_j=V\in E\) with \(|V|=1\). Since \(V_j\in\overline{\mathcal{C}^{l_j}}\) we have \(V_j=(1-P_{l_j})V_j=(1-P)V_j+R_jV_j\) with \(R_j=O({l_j}^{-1})\). Since \(V_j\in \mathcal{C}\) we find \(V_j=PV_j\) and hence
      with \(P(1-P)=0\) that \(V_j=R_jV_j\) for all \(j\in \N\). It follows that \(0=\lim_{j\to\infty}|V_j|=|V|=1\) which is a contradiction.\\
      \textbf{Claim (v)}: There exists \(k_0>k_1\) such that \(\{\mathcal{V}^k\}_{k> k_0}\) is a smooth family of subbundles in \(E\) with \(\mathcal{V}^k\to \mathcal{V}(0)\) in \(\mathscr{C}^\infty\)-topology.\\ 
      Put \(Q=\overline{\mathcal{C}}\times \C\). Since \(\overline{\mathcal{C}}=\ker P\) it follows from the properties of \(P\) that \(Q\) defines a Hermitian vector bundle  over \(X\). Let \(k>k_1\) be arbitrary. Define \(S,S^k\in \vbunsec{X,\operatorname{Lin}(E,Q)}\), \(S(V,a)=((1-P)(V,a),r(V,a))\) and \(S^k(V,a)=((1-P)(1-P_k)(V,a),r(V,a))\). Recall that \(r(V,a)=a\). From Claim~(iii) it follows that \(S,S_k\) are well defined. 
      From Claim~(ii) it follows that  \(\ker S= \mathcal{V}(0)\). Let \(x\in X\) be arbitrary. From Claim~(iii), using \(P_xP_x=P_x\) it follows that \(E_x=\mathcal{C}_x\oplus\overline{\mathcal{C}}_x\) and hence that \(\dim_\C\overline{\mathcal{C}}_x=n+1\). Since \(\mathcal{V}(0)\) is an almost CR structure on \(X\) we have \(\dim_\C \mathcal{V}(0)_x=n\). Using \(\ker S_x= \mathcal{V}(0)_x\) it follows from the dimension formula that \(S_x\) is surjective.
      It follows from Claim~(vi) that we can choose \(k_2>k_1\) large enough such that   \(\mathcal{C}\cap\overline{\mathcal{C}^k}=0\) for all \(k\geq k_2\). Let \(x\in X\) and \(k\geq k_2\) be arbitrary. From Claim~(ii) we find \(\mathcal{V}^k_x\subset \ker S^k_x\). Given \((V,a)\in \ker S^k_x\) it follows \((1-P^k_x)(V,a)\in\ker (1-P_x)\) which implies \((1-P^k_x)(V,a)\in\overline{ \mathcal{C}^k}_x \cap\mathcal{C}_x\) and hence \((1-P^k_x)(V,a)=0\). We conclude that \((V,a)\in \ker S^k_x\) implies that \((V,a)\in \ker (1-P^k_x)\cap \ker r=\mathcal{V}^k\). As a conclusion \(\mathcal{V}^k_x= \ker S^k_x\).  Since \(\mathcal{V}^k\) is an almost CR structure we have  \(\dim_\C \ker S^k_x=\dim_\C \mathcal{V}^k_x=n\). By the dimension formula it follows that \(S_x^k\) is surjective. From Claim~(i), using \((1-P)(1-P)=1-P\), it follows that \((x,k)\mapsto S^k_x\) is smooth and \(S^k\to S\)  in \(\mathscr{C}^\infty\)-topology as \(k\to\infty\).  The claim follows.\\
      We have that \(\C TX\) is identified with the subbundle \(\ker r\subset E\) via \(V\mapsto (V,0)\). This finishes the proof.
 	     
 \end{proof}

\bibliographystyle{plain}


\end{document}